\documentclass[12pt]{amsart}

\usepackage[export]{adjustbox}  
\usepackage{amsmath}
\usepackage{amssymb}
\usepackage{amsthm}
\usepackage[title]{appendix}
\usepackage{array}
\usepackage[backend=bibtex,giveninits=true,style=alphabetic,sorting=nyt,doi=false,isbn=false,url=false,eprint=true]{biblatex}
\usepackage{blindtext}
\usepackage{booktabs} 
\usepackage{caption}
\usepackage{color}
\usepackage{changepage}
\usepackage{csquotes}
\usepackage{enumitem}
\usepackage{float}
\usepackage[margin=1in]{geometry}
\usepackage{graphicx}
\usepackage[hidelinks,linktoc=all]{hyperref}
\usepackage{nicefrac}
\usepackage{parskip}
\usepackage{rotating}
\usepackage{stackengine}
\usepackage{stmaryrd}   
\usepackage{tikz}
\usetikzlibrary{cd}
\usepackage{titlesec}
\usepackage{mathtools}
\usepackage{multicol} 
\usepackage{verbatimbox}

\usepackage{lipsum}
\usepackage{todonotes}

\counterwithout{figure}{section}
\counterwithout{table}{section}
\titleformat{\section}{\normalfont\normalfont\bfseries}{\thesection}{1em}{}
\titleformat{\subsection}{\normalfont\normalfont\bfseries}{\thesubsection}{1em}{}
\titleformat{\subsubsection}{\normalfont\normalfont\bfseries}{\thesubsubsection}{1em}{}
\titleformat{\paragraph}{\normalfont\normalfont\bfseries}{\thesection}{1em}{}

\theoremstyle{definition}
\newtheorem{definition}{Definition}
\newtheorem{remark}{Remark}
\newtheorem{convention}{Convention}

\newtheorem{claim}{Claim}
\theoremstyle{plain}
\newtheorem{theorem}{Theorem}
\newtheorem{lemma}{Lemma}
\newtheorem{corollary}{Corollary}
\newtheorem{conjecture}{Conjecture}

\newcommand{\defw}[1]{{\bf{\textit{#1}}}}

\newcommand{\ordth}{\textsuperscript{th}\,}
\newcommand{\pieq}{\stackrel{\pi}{=}}

\newcommand{\ZZ}{{\mathbb{Z}}}

\newcommand{\id}{\textnormal{id}}
\newcommand{\cobIV}{\mathcal{C}ob^4}

\newcommand{\chcobIII}{\mathcal{C}hron\mathcal{C}ob^3}
\newcommand{\cobIIIgen}{\mathcal{C}hron\mathcal{C}ob^3_{gen}}

\newcommand{\mat}[1]{\textnormal{Mat}(#1)}
\newcommand{\kom}[1]{\textnormal{Kom}(#1)}
\newcommand{\komb}[1]{\textnormal{Kom}^b(#1)}

\newcommand{\kombpmp}[1]{\textnormal{Kom}^b_{\pm,\pi}(#1)}

\newcommand{\includeMov}[1]{\adjustbox{valign = c}{\includegraphics[scale=0.0375]{Figures/#1.pdf}}}
\newcommand{\includeFig}[1]{\adjustbox{valign = c}{\includegraphics[scale=1.0]{Figures/#1.pdf}}}

\newcommand{\includeCob}[1]{\adjustbox{valign = c}{\includegraphics[scale=1.25]{Figures/#1.pdf}}}
\newcommand{\includeSym}[1]{\adjustbox{valign = c}{\includegraphics[scale=1.0]{FigureAssembly/Symbols/Symbol-#1/symbol.pdf}}}

\newcommand{\includeTang}[1]{\adjustbox{valign = c}{\includegraphics[scale=0.1]{Figures/#1.pdf}}}

\newcolumntype{C}[1]{>{\centering\let\newline\\\arraybackslash\hspace{0pt}}m{#1}}

\title{Functoriality of Odd and Generalized Khovanov Homology in $\mathbb{R}^3\times I$}
\author{Jacob Migdail}
\address{Washington and Lee University; Mathematics Department; Chavis Hall; Lexington, VA 24450}
\email{jmigdail-smith@wlu.edu}

\author{Stephan Wehrli}
\thanks{SW was partially supported by
a grant from the Simons Foundation (\#632059 Stephan
Wehrli).}
\address{Syracuse University; Department of Mathematics; 215 Carnegie; Syracuse, NY 13244}
\email{smwehrli@syr.edu}
\date{\today}

\bibliography{bib}

\begin{document}

\begin{abstract} 
We extend the generalized Khovanov bracket to smooth link cobordisms in \mbox{$\mathbb{R}^3\times I$} and prove that the resulting theory is functorial up to global invertible scalars. The generalized Khovanov bracket can be specialized to both even and odd Khovanov homology.  Particularly by setting $\pi=-1$, we obtain that odd Khovanov homology is functorial up to sign. We end by showing that odd Khovanov homology is not functorial under smooth link cobordisms in $S^3\times I$.
\end{abstract}

\maketitle

\section{Introduction}\label{s:intro}

In his seminal paper \cite{Kh1999}, Khovanov categorified the Jones polynomial \cite{Jo1985}
by introducing a powerful new link invariant now known as Khovanov homology.
Shortly after, Jacobsson \cite{Ja2004} proved Khovanov's 
conjecture that Khovanov homology is functorial 
under smooth link cobordisms up to sign.
Bar-Natan built upon this in  \cite{Bn2005} by providing an alternative formulation
of Khovanov homology---in terms of
an abstract cobordism category---which allowed him to give
a shorter and more conceptual proof of Jacobsson's result.
Khovanov homology was later truly realized as a functor
by multiple authors
 \cite{Ca2007,CMW2009,Bl2010,BHPW2019,Vo2020,Sa2021},
 who addressed the sign indeterminacy in the definition of the Khovanov cobordism maps
 in various ways.

Over the past two decades, the functoriality properties of Khovanov and Lee homology \cite{Le2005}
have been used in many of the main applications of these
theories. In particular, these applications include Rasmussen's combinatorial proof of the topological Milnor conjecture \cite{Ra2004},
Piccirillo's proof that the Conway knot is not smoothly slice \cite{Pi2020}, and the recent work of Hayden and Sundberg \cite{HS2022}, which shows that Khovanov homology can distinguish smooth surfaces in the 4-ball that are
topologically---but not smoothly---ambient isotopic.
In addition, Morrison-Walker-Wedrich used the full functoriality
of Khovanov homology (and more generally Khovanov-Rozansky $sl(N)$ homology \cite{KR2008, EK2010})
to define the Lasagna skein modules for smooth 4-manifolds \cite{MWW2022}.

Parallel to these developments, Ozsv\'ath-Rasmussen-Szab\'o \cite{ORS2007} found an alternative categorification of the Jones polynomial,
which they called odd Khovanov homology, and which 
is conjectured to provide the $E_2$ page of a spectral sequence abutting to
Heegaard Floer homology  \cite{OS2004}.
While this odd Khovanov homology theory agrees with the original ``even'' theory when
coefficients are taken in $\mathbb{Z}_2$, computations by Shumakovitch \cite{Sh2011} showed that---over rational coefficients---each
theory can distinguish links that the other cannot.

A geometric framework for odd Khovanov homology, in the spirit of \cite{Bn2005},
was developed by Putyra in  \cite{Pu2015}, who defined a generalized Khovanov bracket
that specializes to both even and odd Khovanov homology
(see \cite{BW2010} for a related construction).
Putyra also conjectured that his generalized Khovanov bracket extends to a functor
on smooth link cobordisms.
In the present paper, we prove this conjecture:

\begin{theorem}\label{thm:main} The generalized Khovanov bracket is functorial under smooth
link cobordims up to homotopy and overall invertible scalars.\end{theorem}

In this theorem, we are assuming that Putyra's cobordism category
is defined over the ring
$\Bbbk:=\mathbb{Z}[\pi]/(\pi^2-1)$
where $\pi$ denotes a formal variable.
In particular, the invertible scalars that can appear
are $\pm 1$ and $\pm \pi$.
After setting $\pi=-1$, the generalized Khovanov bracket specializes to
an odd Khovanov bracket, and thus Theorem~\ref{thm:main} implies:

\begin{corollary}\label{cor:main}
Odd Khovanov homology is functorial under smooth link cobordisms up to sign.
\end{corollary}

Odd Khovanov homology differs from even Khovanov homology in that, to construct
the odd chain complex, one needs to make certain additional choices.
Specifically, one needs to choose arrows at the crossings of the link diagram (called
crossing orientations) along with a corresponding valid sign assignment.
There are two overall types of valid sign assignments: type X and type Y.
It was claimed in  \cite[Lemma 2.4]{ORS2007} that both types yield
isomorphic odd Khovanov complexes, but the proof given there was incorrect.
Putyra's construction of the generalized Khovanov bracket was originally based on a generalization of
type Y assignments  \cite{Pu2015}, but it works equally well for type X assignments.
We will prove the following, and as a byproduct, obtain a correct proof of
 \cite[Lemma 2.4]{ORS2007}:

\begin{theorem}\label{thm:equivalence}
Type X and type Y assignments yield naturally isomorphic generalized Khovanov functors.
\end{theorem}

To prove Theorem~\ref{thm:main}, we will assign a chain map $\Phi_F$
to each smooth link cobordism $F\subset\mathbb{R}^3\times I$.
The definition of this chain map was previously outlined by Putyra \cite{Pu2015}
and is based on presenting $F$ by a movie of link diagrams.
To prove that $\Phi_F$ is well-defined, up to homotopy and invertible scalars,
we must show that it is invariant under the 15 Carter-Saito movie moves \cite{CS1997}
and additional movie moves that correspond to time-reordering distant portions of a link cobordism.

A key result in this context will be our Lemma~\ref{lem:t2}, which asserts that the generalized Khovanov bracket
of a link diagram $D$ has no interesting automorphisms induced by Reidemeister moves acting
on certain types of tangles in $D$. This lemma provides a substitute for Bar-Natan's approach via
$Kh-$simple tangles \cite{Bn2005}, but its proof does not rely on any
planar composition properties of the generalized Khovanov bracket. In particular,
our proof is independent of recent constructions that extend odd Khovanov homology
to tangles \cite{NP2020,SV2023,Sp2024}.

While Lemma~\ref{lem:t2} is enough to prove invariance of $\Phi_F$ under movie moves 6-10,
some of the other movie moves require explicit computations of the chain maps induced by the
two sides of the move. Such is the case with movie moves 13 and 14, where we can simplify the
proof of invariance by choosing appropriate sign assignments.

It was shown in  \cite{MWW2022} that even Khovanov homology is functorial under smooth link cobordisms in $S^3\times I$.
In contrast,
we will see that odd Khovanov homology is only functorial in $\mathbb{R}^3\times I$:

\begin{theorem}\label{thm:noninvariance} 
There is an infinite family of smooth link cobordisms in $\mathbb{R}^3\times I$ which are mutually ambient isotopic in $S^3\times I$, but which induce different maps on odd Khovanov homology.
\end{theorem}

The examples that we will use to prove this theorem will also show that the odd Khovanov cobordism
maps are not invariant under ribbon moves (cf.  \cite{Og2000}), which is
another difference with the even setting \cite{CSS2006}.

Because of its construction through exterior algebras, there is no obvious odd Lee homology.
Moreover, any link cobordism that is a connected sum with a trivial, unlinked torus induces
the zero map on odd Khovanov homology. This means that all of the arguments that were
advanced in the even setting \cite{Ra2005,CSS2006,Ta2006} to show that certain cobordism maps
are uninteresting fail in the odd setting.

This opens the possibility of using odd Khovanov homology to construct
a nontrivial invariant for smooth 2-knots embedded in the 4-sphere. Specifically, let $K\subset S^4$ be such a
2-knot. After removing a small neighborhood of a point $p\in K$,
this 2-knot becomes a slice disk for the unknot, $U$,
or equivalently, a smooth link cobordism $\emptyset\rightarrow U$.
This cobordism
induces a map
\[
\Phi\colon\mathbb{Z}\longrightarrow\Lambda^*(U)=\mathbb{Z}\{1,U\}
\]
on odd Khovanov homology. For degree reasons, it follows
that this map must send the generator $1\in\mathbb{Z}$ to an integer multiple
\[
\Phi(1)=n1
\]
We can thus define an invariant $n(K)\in\mathbb{N}$ by taking
the absolute value of the integer $n$ from the above equation.

While the corresponding invariant derived from even Khovanov homology is
always equal to $1$  \cite{Ra2005,Ta2006}, our invariant $n(K)$ can be any positive odd number.
(The fact that it must be odd follows from a comparison with $\mathbb{Z}_2$-Khovanov
homology using the universal coefficient theorem.)
We conjecture:

\begin{conjecture}\label{conj:2knots}
For every smooth 2-knot $K\subset S^4$, the number
$n(K)$ is equal to the order of the first homology
of the branched double-cover of $S^4$, branched along $K$.
\end{conjecture}

In two upcoming papers \cite{MW2024a,MW2024b}, we will further explore this conjecture.
In the first of these, \cite{MW2024a}, we will show that the reduced odd Khovanov homology
of a knot or link $L$ is naturally a module over the exterior algebra
\[
\Lambda^*H_1(\Sigma(L);\mathbb{Z}),
\]
where $\Sigma(L)$ denote the branched double-cover of the 3-sphere $S^3$, branched along $L$.
While this module structure is defined in purely diagrammatic terms, it will allow us to compute
certain odd cobordism maps geometrically.
In particular, we will use it to prove Conjecture~\ref{conj:2knots}
for all ribbon 2-knots, and to show that Levine-Zemke's main result from \cite{LZ2019} (about
ribbon concordances inducing injective maps on even Khovanov homology)
goes through for odd Khovanov homology with rational coefficients. The
action of $\Lambda^*H_1(\Sigma(L);\mathbb{Z})$ also provides a better understanding of
an observation of Shumakovitch about the presence of torsion in the odd Khovanov homology
of certain Pretzel knots \cite{Sh2011}.

In our second upcoming paper, \cite{MW2024b}, we will describe an action of the Hecke algebra $H(-1,n)$
on the odd Khovanov homology of the $n$-cable of an even-framed link $L\subset\mathbb{R}^3$.
Using a related functor, we will then prove Conjecture~\ref{conj:2knots} for all
even-twist spun knots.
A similar approach should lead to a proof of Conjecture~\ref{conj:2knots} for all odd-twist spun knots, in which case
$n(K)$ should be equal to 1, and the action of $H(-1,n)$ should be replaced
by a symmetric group action similar to the one from \cite{GLW2017}.

\subsection{Organization}

The remainder of this paper is organized in the following manner.  
In \textbf{Section~\ref{s:preliminaries}} we recall the initial link cobordism category and the target category for defining the generalized Khovanov bracket.  
In \textbf{Section~\ref{s:generalized}} we recall the construction of generalized Khovanov bracket as an invariant, prove the equivalence of type X and type Y theories, and define the chain maps that the generalized Khovanov bracket assigns to elementary link cobordisms.  
In \textbf{Section~\ref{s:functoriality}} we prove the main theorem of this paper, the functoriality of generalized Khovanov bracket up to
homotopy and invertible scalars.  
In \textbf{Section~\ref{s:noninvariance}} we prove that odd Khovanov homology is not functorial in $S^3\times I$.

\section{Preliminaries}\label{s:preliminaries}

\begin{convention}
	For cobordisms, the front is the surface \enquote{closest} to the reader.  Similarly, for movies, the front is the bottom of the frame.  Additionally movies and cobordisms should be read from the bottom to the top of the page.  The diagram at the bottom is initial with respect to time, and the diagram at the top is final or terminal.
\end{convention}

\subsection{The Source Category: $\cobIV$}

Our generalized Khovanov functor will be a functor on the category $\cobIV$ of smooth
oriented link cobordisms in $\mathbb{R}^3\times I$, considered up to smooth ambient isotopy rel boundary. 
We will assume that the objects in $\cobIV$ are oriented links in $\mathbb{R}^3$ which are in general position
with respect to the projection onto the $xy$-plane, so that they can
be represented by oriented link diagrams. For technical reasons,
the construction of the generalized Khovanov bracket requires that link diagrams be enhanced
with additional data consisting of crossing orientations and a valid sign assignment. However,
we will not assume that the objects of
$\cobIV$ are equipped with this extra data. Instead, we will show in Subsection~\ref{subs:colimit}
that the generalized Khovanov brackets for different choices of the auxiliary data
can be repackaged as a single invariant.

We note that any smooth link cobordism $S\subset\mathbb{R}^3\times I$ is smoothly
ambient isotopic to a cobordism $S'$ such that the height function $\mathbb{R}^3\times I\rightarrow I$
restricts to a Morse function on $S'$. Furthermore, it can be assumed that all critical
points of this Morse function occur at different heights. By intersecting $S'$
with hyperplanes $\mathbb{R}^3\times\{t_i\}$ for sufficiently many generic $t_i\in I$ with $0=t_0<t_1<\ldots<t_n=1$,
one can then represent $S'$ as a \defw{movie} of link diagrams $D_0,\ldots,D_n$, such that any two
consecutive link diagrams differ by one of the following: a planar isotopy, and elementary birth, death, or saddle move,
or a Reidemeister move.
Carter and Saito \cite{CS1991,CS1993} showed that two movies represent smoothly ambient
isotopic link cobordisms in $\mathbb{R}^3\times I$ if and only if they are related by certain \defw{movie moves}. In more detail,
their result can be stated as follows:

\begin{theorem}[\cite{CS1993}]\label{thm:CS}
	Two movies represent smoothly ambient isotopic link cobordisms if and only if they differ by a finite sequence of 
	the following moves:
\begin{itemize}
\item Chronological movie moves.
\item Planar ambient isotopy moves.
\item The fifteen particular movie moves
depicted in Figures \ref{fig:4:TIMM}, \ref{fig:4:TIIMM}, \ref{fig:4:MM11}, \ref{fig:4:MM12}, \ref{fig:4:MM13}, \ref{fig:4:MM14}, and \ref{fig:4:MM15}.
\end{itemize}
\end{theorem}

By a chronological movie move, we mean a movie move that corresponds to vertically reordering
distant critical points in a link cobordism, or to reordering distant Reidemeister moves with each other
or with critical points. A planar ambient isotopy move is a movie move that consists in replacing
 a movie of the form $D_0, f(D_0), f(D_1)$ by the movie $D_0, D_1, f(D_1)$, where
 the transition $D_0\rightarrow f(D_0)$ corresponds to a planar ambient isotopy, and $f$
 is the time-one map in this planar ambient isotopy.

For Carter and Saito's fifteen particular movie moves, the first ten movie moves are \enquote{do nothing moves} in that each given strip in Figures~\ref{fig:4:TIMM} and \ref{fig:4:TIIMM} is equivalent to the identity cobordism from the first to the final frame.  The final five moves are non-reversible, and must be considered as separate movie moves in the forward and backward directions.  Like Reidemeister moves, where there are multiple versions of some of the moves associated with changing particular crossings, the movie moves also have multiple variants when applicable. In \cite{Bn2005}, the fifteen movie moves are divided into three subgroups, each with five members titled type I, type II, and type III movie moves. We will follow this convention.  The type I movie moves correspond to cobordisms generated by doing and undoing the same Reidemeister move. The type III movie moves are the only ones that involve births or deaths.

Note that Theorem~\ref{thm:CS} implies that $\cobIV$ can be described more combinatorially as the category
whose objects are oriented link diagrams in $\mathbb{R}^2$, and whose morphisms are movies of oriented
link cobordisms, modulo the moves from Theorem~\ref{thm:CS}.

\subsection{The Target Category: $\cobIIIgen$}\label{subs:targetcategory}

The target category of our functor was first developed by Putrya in \cite{Pu2015} and further refined in \cite{PL2016}.
While the following constructions are due to Putrya our notation is closer to Bar-Natan's in \cite{Bn2005}.  First consider $\chcobIII$, the category whose objects are
unoriented closed 1-manifolds embedded in the plane  and whose morphisms are chronological cobordisms embedded in $\mathbb{R}^2\times I$.
The cobordisms are chronological in the sense that
\begin{enumerate}
\item
the restriction of the projection $\mathbb{R}^2\times I\rightarrow I$ to the cobordism constitutes a Morse function
whose critical points occur at different heights, and
\item
each descending manifold of an index-1 or -2
critical point is equipped with a local orientation around the critical point.
\end{enumerate}
Cobordisms in $\chcobIII$ are considered up to smooth ambient isotopies which preserve the chronological structure.
Explicitly, this means that each intermediate cobordism in the isotopy satisfies properties (1) and (2) above,
and that the local orientations on the descending manifolds are preserved by the ambient isotopy.
The composition in $\chcobIII$ is given by vertical stacking of cobordisms and identifying $\mathbb{R}^2\times[0,2]$
with $\mathbb{R}^2\times I$.

Note that any morphism in $\chcobIII$ is a composition of finitely many chronological cobordisms without critical points
or only one critical point. Chronological cobordisms with exactly one critical points fall into the following types: births, clockwise deaths,
counterclockwise deaths, merges, and splits. These types of cobordisms are shown schematically
in Figure~\ref{fig:generatingCobordisms}, where the last two pictures do not necessarily reflect
the embedding into $\mathbb{R}^2\times I$, and any additional components without critical points are
are omitted.

\begin{figure}[H]
	\centering
	\includeCob{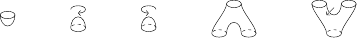}
	\caption{Cobordisms with only one critical point}
	\label{fig:generatingCobordisms}
\end{figure}

  While merges have local orientations, they are rarely relevant in our setting and will often be omitted.  To simplify pictures of cobordisms, we will often also omit the orientations of the splits and deaths. In such cases, we will use the following convention:

\begin{convention}\label{conv:2:orientations}
	The default orientations are clockwise on deaths,
	and forward or to the right on saddles.
	Examples of this convention are shown in Figure~\ref{fig:defaultOrientations}.
\end{convention}

\begin{figure}[H]
	\centering
\noindent\begin{minipage}[b]{0.333\textwidth}
	\begin{equation}
		\label{eq:deathSignChange}
		\includeCob{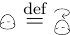}
	\end{equation}
\end{minipage}%
\begin{minipage}[b]{0.334\textwidth}
	\begin{equation}
		\label{eq:splitSignChange}
		\includeCob{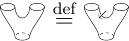}
	\end{equation}
\end{minipage}%
\begin{minipage}[b]{0.333\textwidth}
	\begin{equation}
		\label{eq:mergeSignChange}
		\includeCob{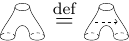}
	\end{equation}
\end{minipage}
\caption{Default orientations used in pictures}\label{fig:defaultOrientations}
\end{figure}

To define the target category of the generalized Khovanov bracket, we need
to replace $\chcobIII$ by a
$\Bbbk$-linear category $\cobIIIgen$, which has the same objects as $\chcobIII$, and
where $\Bbbk$ denotes the ring
\[
\Bbbk:=\mathbb{Z}[\pi]/(\pi^2-1)
\]
from the introduction.\footnote{In \cite{Pu2015} generalized Khovanov homology was developed over the ring
$\ZZ[X,Y,Z^{\pm1}]/(X^2=Y^2=1)$ 
which can be specialized to odd or even Khovanov homolgy by combinations of setting $X$, $Y$, and $Z$ to $1$ or $-1$.
Our ring $\Bbbk$ can be seen as a specialization of this larger ring for $X=Z=1$ and $Y=\pi$.
In \cite{PL2016}, our ring $\Bbbk$ was denoted $\ZZ_\pi$, and it was shown that the generalized
Khovanov bracket over $\ZZ_\pi$ carries the same information as the original generalized Khovanov
bracket from \cite{Pu2015}. Note that setting $\pi=1$ or $\pi=-1$ in the generalized theory results
in even or odd Khovanov homology, respectively.
}
The morphisms in $\cobIIIgen$
are given by formal $\Bbbk$-linear combinations of morphisms in $\chcobIII$,
where we define the composition in $\cobIIIgen$ bilinearly using the composition
in $\chcobIII$. To arrive at $\cobIIIgen$, we also mod out by the following relations,
where we use Convention~\ref{conv:2:orientations} in cases where
local orientations are not indicated:

\subsubsection*{Orientation Reversal Relations}

\noindent\begin{minipage}[b]{0.333\textwidth}
	\begin{equation}
		\label{eq:deathSignChange}
		\includeCob{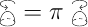}
	\end{equation}
\end{minipage}%
\begin{minipage}[b]{0.334\textwidth}
	\begin{equation}
		\label{eq:splitSignChange}
		\includeCob{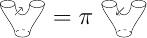}
	\end{equation}
\end{minipage}%
\begin{minipage}[b]{0.333\textwidth}
	\begin{equation}
		\label{eq:mergeSignChange}
		\includeCob{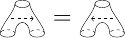}
	\end{equation}
\end{minipage}

\subsubsection*{Frobenius and Associativity Relations}

{\noindent}\begin{minipage}{0.333\textwidth}
    \begin{equation}
        \includeCob{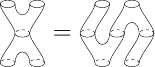}
    \end{equation}
\end{minipage}%
\begin{minipage}{0.334\textwidth}
	\begin{equation}
		\includeCob{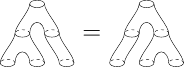}
	\end{equation}
\end{minipage}%
\begin{minipage}{0.333\textwidth}
    \begin{equation}
        \includeCob{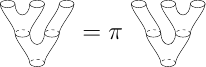}
    \end{equation}
\end{minipage}

\subsubsection*{Commutativity Relations}

{\noindent}\begin{minipage}{0.333\textwidth}
    \begin{equation}
        \includeCob{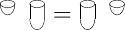}
    \end{equation}
\end{minipage}%
\begin{minipage}{0.334\textwidth}
    \begin{equation}
        \includeCob{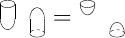}
    \end{equation}
\end{minipage}%
\begin{minipage}{0.333\textwidth}
    \begin{equation}
        \includeCob{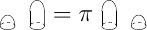}
    \end{equation}
\end{minipage}

{\noindent}\begin{minipage}{0.333\textwidth}
    \begin{equation}
        \includeCob{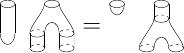}
    \end{equation}
\end{minipage}%
\begin{minipage}{0.334\textwidth}
	\begin{equation}
		\includeCob{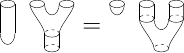}
	\end{equation}
\end{minipage}%
\begin{minipage}{0.333\textwidth}
    \begin{equation}
        \includeCob{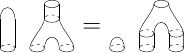}
    \end{equation}
\end{minipage}

{\noindent}\begin{minipage}{0.5\textwidth}
    \begin{equation}
        \includeCob{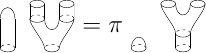}
    \end{equation}
\end{minipage}%
\begin{minipage}{0.5\textwidth}
    \begin{equation}
        \includeCob{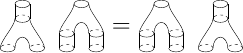}
    \end{equation}
\end{minipage}

{\noindent}\begin{minipage}{0.5\textwidth}
    \begin{equation}
        \includeCob{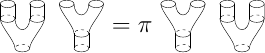}
    \end{equation}
\end{minipage}%
\begin{minipage}{0.5\textwidth}
    \begin{equation}
        \includeCob{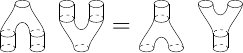}
    \end{equation}
\end{minipage}

\subsubsection*{Cross Relations}

{\noindent}\begin{minipage}{0.5\textwidth}
    \begin{equation}
        \includeCob{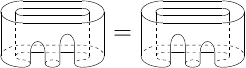}
    \end{equation}
\end{minipage}%
\begin{minipage}{0.5\textwidth}
    \begin{equation}
        \includeCob{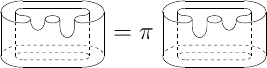}
    \end{equation}
\end{minipage}\\

\subsubsection*{Diamond Relation}

\begin{equation}
    \includeCob{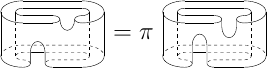}
\end{equation}

\subsubsection*{Pruning Relations}

\begin{equation}
	\includeCob{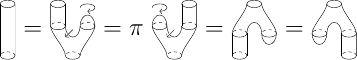}
\end{equation}

\subsubsection*{Sphere and Torus Relations}

{\noindent}\begin{minipage}{0.5\textwidth}
	\begin{equation}
		\includeCob{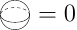}
	\end{equation}
\end{minipage}%
\begin{minipage}{0.5\textwidth}
	\begin{equation}
		\includeCob{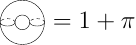}
	\end{equation}
\end{minipage}

\subsubsection*{Four Tube Relation}

\begin{equation} 
	\includeCob{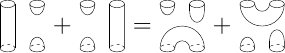}
\end{equation}\\

\begin{remark}
	The cobordisms in this paper are always embedded. However, the above pictures
show non-embedded versions of the relations and should be interpreted as providing
all possible ways in which the relation can be embedded into $\mathbb{R}^2\times I$.
For example, the two circles at the top of the split saddle in the second orientation reversal
relation could be nested within one another.
\end{remark}

\begin{remark}
In \cite{Pu2015}, Putyra works with non-embedded cobordisms except for a brief period where
it is strictly necessary that he work with embedded cobordisms.
Working with non-embedded cobordisms simplifies the presentation of the relations.
\end{remark}

\begin{remark}
The associativity relations are so named because they are reminiscent of the
associativity and coassociativity relations that hold in a bialgebra.
In contrast, the commutativity relations are concerned with the effect of vertically
commuting critical points past each other. The sphere, torus, and four tube
relations are essentially Bar-Natan's relations from \cite{Bn2005},
where they are called the (S), (T), and (4Tu) relation, respectively.
\end{remark}

\begin{remark} 
The two cross relations are not strictly needed because they can be deduced
from the orientation change relations. In fact,
the cobordisms that appear on the two sides of these
relations are ambient isotopic via an ambient isotopy that fixes 
the heights of the critical points, while reversing the local orientation at one
of the critical points.
For similar reasons, the diamond relation would not be needed
if we were working in a non-embedded setting.
We will see below that the cobordisms that appear in the diamond relation
remain unchanged under multiplication by $\pi$. The factor
of $\pi$ on the right-hand side could thus be omitted.
\end{remark}

Following \cite{Pu2015}, we define the chronological degree of a cobordism $S$ as the pair
$(a,b)\in\mathbb{Z}\times\mathbb{Z}$ where
\[
a:=\#\mbox{births}-\#\mbox{merges}\qquad\mbox{and}\qquad
b:=\#\mbox{deaths}-\#\mbox{splits}.
\]
All of the relations in  $\cobIIIgen$ are homogenous with respect
to this degree, and thus the chronological degree gives rise to
a $(\mathbb{Z}\times\mathbb{Z})$-grading on morphisms sets of $\cobIIIgen$.
By its definition,
 the chronological degree is additive under
compositions of cobordisms and under the disjoint union operations considered
in \cite{Pu2015}. 
For $(a,b)$ as above, \defw{quantum degree} of $S$ is defined as the sum $a+b\in\mathbb{Z}$
and the \defw{superdegree} as the reduction of $b$ modulo $2$.
Accordingly, we will say that  $S$ is even or odd depending on the parity of $b$.
Note that births and merges are even, while deaths and splits are odd.

The commutativity relations imply that the superdegree determines
a cobordism's commutativity behavior, in the following sense.
Suppose $S_1$ and $S_2$ are two cobordisms in $\mathbb{R}^2\times I$
which have disjoint projections to $\mathbb{R}^2$. Suppose further that all critical points
in $S_1$ initially occur at lower heights than the ones in $S_2$, and let $S_1'$
and $S_2'$ be vertically shifted versions of $S_1$ and $S_2$ such that the
critical points in $S_1'$ occur at greater heights than the ones in $S_2'$.
Then $S_1\cup S_2=S_1'\cup S_2'$ if one of $S_1$ and $S_2$ is even,
and $S_1\cup S_2=\pi S_1'\cup S_2'$ if $S_1$ and $S_2$ are both odd.

\begin{definition}
	We will say that two cobordisms or morphisms
	are \defw{$\pi$-commuting} if exchanging their order induces multiplication by $\pi$. 
\end{definition}

\subsection{Emergent Relations}

One particular consequence of the relations described above is the handle
(H) relation:
\begin{equation}
	\includeCob{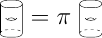}
\end{equation}
If the handle cobordism is embedded in $\mathbb{R}^2\times I$ as shown above, then this relation follows from the orientation reversal relations because
we can reverse the local orientation at the split saddle by rotating
the cobordism by $180^\circ$ about a
vertical axis (the resulting isotopy can be chosen to be the
identity near the boundaries). The (H) relation also holds for different
embeddings of the handle because one use the diamond relation to switch between different
embeddings. Note that the (H) relation implies that a handle
is annihilated by $1-\pi$. On the other hand, the following result from \cite{Pu2015}
tells us that this cannot happen for genus zero cobordisms without closed components:

\begin{lemma}[\cite{Pu2015}] Suppose $kS=0$ for a chronological cobordism and a nonzero $k\in\Bbbk$.
Then $S$ has either positive genus or closed components, and $k$ is
divisible by $1-\pi$.
\end{lemma}

The following consequences of the (4Tu) relation will be used later:

{\noindent}\begin{minipage}{0.5\textwidth}
	\begin{equation}
		\includeCob{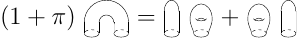}
	\end{equation}
\end{minipage}%
\begin{minipage}{0.5\textwidth}
	\begin{equation}
		\includeCob{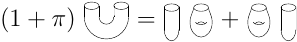}
	\end{equation}
\end{minipage}

\begin{equation}
	\includeCob{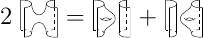}
\end{equation}

\subsection{The Target Category: $\textnormal{Kom}^b(\mat{\cobIIIgen})$}

Given a preadditive category $\mathcal{A}$, we will denote by
 $\mat{\mathcal{A}}$ its additive closure. The objects of $\mat{\mathcal{A}}$ are finite (possibly empty) sequences of objects in $\mathcal{A}$, and its morphisms are given by matrices of morphisms in $\mathcal{A}$.  The composition is modeled on ordinary matrix multiplication.

If $\mathcal{A}$ is already additive and $\Bbbk$-linear, then we will denote by $\textnormal{Kom}^b(\mathcal{A})$ the
category of bounded chain complexes and chain maps in $\mathcal{A}$. Let $\textnormal{Kom}^b_h(\mathcal{A})$
denote the corresponding homotopy category in which homotopic chain maps are identified.
Moreover, let $\textnormal{Kom}^b_{\pm h,\pm\pi h}(\mathcal{A})$ denote the ``projectivization'' of 
$\textnormal{Kom}^b_h(\mathcal{A})$ in which two homotopy classes
are identified if they differ by a global invertible scalar $s\in\Bbbk^\times=\{\pm 1,\pm\pi\}$.
Note that the morphism sets in this category are no longer abelian groups or $\Bbbk$-modules, but they still
carry a multiplicative action of the multiplicative semigroup $\Bbbk/\Bbbk^\times$.

The target category for the generalized Khovanov bracket is now given by the category
$\textnormal{Kom}^b_{\pm h,\pm\pi h}(\mat{\cobIIIgen})$, and it is in this category that
the generalized Khovanov bracket is functorial. It is often simpler
to work with $\textnormal{Kom}^b(\mat{\cobIIIgen})$, in which the formal Khovanov
bracket is only functorial up to homotopy and invertible scalars.

\begin{remark}
	The category where the odd Khovanov bracket resides is obtained by setting $\pi=-1$
	in the definition of $\cobIIIgen$, and otherwise using the same definitions as above.
	Similarly, the target category for the even Khovanov bracket
	is obtained by setting $\pi=1$.
\end{remark}

\subsection{Type X and Type Y Configurations}

From the diamond relation, there emerges a pair of exceptional arrangements which
are both commuting and $\pi$-commuting because the involved cobordisms are annihilated by $1-\pi$.
These two configurations are called type X and type Y and are shown in Figure~\ref{fig:xytype}.

\begin{figure}[H]
    \centering
    \includeCob{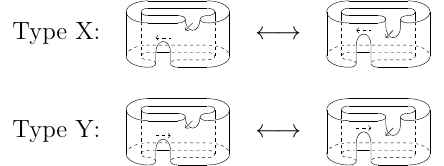}
    \caption{Cobordisms of type X and type Y configurations}
    \label{fig:xytype}
\end{figure}

Note that Figure~\ref{fig:xytype} correctly reflects the embeddings of the shown cobordisms into the oriented
manifold $S^2\times I\supset\mathbb{R}^2\times I$. In fact, the two types can be transformed into
each other by applying an orientation-reversing homeomorphism of the plane $\mathbb{R}^2$.

In order to define a generalized Khovanov theory, one must artificially decide that one of these configurations commutes and that the other $\pi$-commutes. 
By consequence, there are two different generalized Khovanov theories, which are called type X or type Y
based on which pair of cobordisms
in Figure~\ref{fig:xytype} $\pi$-commutes.
It was first observed in \cite{ORS2007} that either choice produces the same final invariant, although the proof given there was incorrect.  Putyra in \cite{Pu2015} outlined a corrected proof, but his
argument is not clear to us and does not address the functoriality of the type X and type Y theories under link cobordisms.
This will be done in Lemma \ref{lem:XY4} in Subection~\ref{subs:XYequivalence}, where we will show that
the type X and type Y theories provide naturally isomorphic functors. For the most part of this
paper, we will use type Y sign assignments following Putyra \cite{Pu2015}.

\section{Generalized Khovanov Theory}\label{s:generalized}

We will now construct the generalized Khovanov bracket of an oriented link diagram $D$.  Then we will review portions of the proof of invariance, as they will be needed to extend the generalized Khovanov bracket to link cobordisms and to prove that it is functorial.  We will end this section by defining the chain maps that are assigned to each of the elementary cobordisms in $\cobIV$.

The generalized Khovanov bracket is not constructed from a simple oriented link diagram, but
from an oriented link diagram that has been enhanced with a local orientation at each of its crossings.  Each crossing has two possible orientations, which are displayed in Figure \ref{fig:orientedCrossing} and represented by arrows.

\begin{figure}[H]
    \centering
    \includeTang{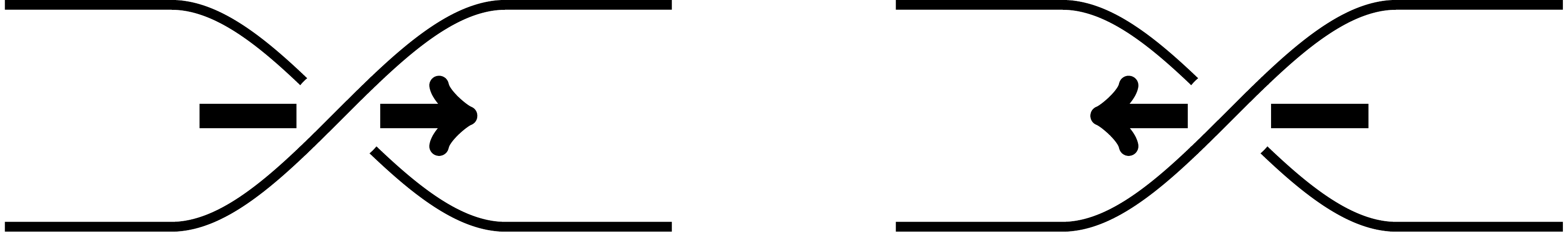}
    \caption{The two possible orientations on a crossing}
    \label{fig:orientedCrossing}
\end{figure}

Note that the arrows specifying the crossing orientation are chosen so that they connect the regions that lie to the left if one approaches the crossing along the overstrand from either side.
In the following, we will assume that link diagrams are oriented and equipped with orientations on their crossings,
unless otherwise stated.

\subsection{Crossing Resolution}

A link diagram $D$ with $n$ crossings gives rise to $2^n$ planar diagrams corresponding to all possible combinations of 
replacing each crossing with 
the vertical \defw{$0$-resolution} or the horizontal \defw{$1$-resolution}. Here the terms
\enquote{vertical} and \enquote{horizontal} refer to the pictures in Figures~\ref{fig:orientedCrossing} and \ref{fig:resolutions}.  
In intrinsic terms, the $0$-resolution is the resolution whose arcs veer to the left as one approaches the
original crossing along the overstrand from either side, and the $1$-resolution is the resolution whose arcs veer to the right.

If the crossings of $D$ are labeled from 1 to $n$, then each resolved diagram, $D_\alpha$, can be assigned a binary label $\alpha\in\{0,1\}^n$ with a 0 in the $i$\ordth place if the $i$\ordth crossing was replaced with the $0$-resolution, and a 1 if it was replaced with the $1$-resolution.  

\begin{figure}[H]
    \centering
    \captionsetup{width=\textwidth}
    \includeFig{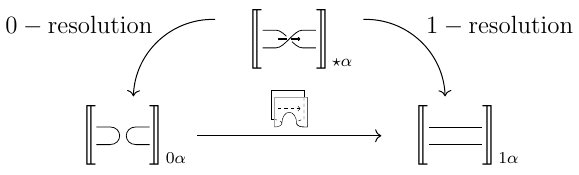}
     \caption{The two resolutions of a crossing}
    \label{fig:resolutions}
\end{figure}

\begin{convention}
    To denote that a
    digit in a binary string is in a \enquote{superposition} of a $0$ and a $1$, we use a $\star$ for that digit.
\end{convention}
We assign a cobordism $d_{{\dots}\star{\dots}}:D_{{\dots}0{\dots}}\rightarrow D_{{\dots}1{\dots}}$ to each pair of resolved diagrams that differ at a single crossing $c$ of $D$.  
Specifically, $d_{{\dots}\star{\dots}}$ is a saddle cobordism from the diagram $D_{{\dots}0{\dots}}$, in which $c$ is
replaced by a $0$-resolution, the diagram $D_{{\dots}1{\dots}}$, in which $c$ is replaced by a 1-resolution
(as depicted in Figure \ref{fig:resolutions}).  The saddle point in $d_{{\dots}\star{\dots}}$ inherits a local orientation from the orientation
of the resolved crossing $c$,
and thus $d_{{\dots}\star{\dots}}$ be viewed as a morphism in $\cobIIIgen$.

\begin{convention}
    Greek characters will be used to denote binary strings and the concatenation of Greek characters should be read as the concatenation of the binary strings.  Furthermore, $\zeta$ will be used to denote the string consisting of only zeros.
\end{convention}

\begin{definition}
    The \defw{degree} of a resolution is the number of $1$-resolutions in the resulting diagram.
\end{definition}

The degree of a resolution is thus the number of ones in the diagram's binary label, and it is denoted $\deg(\alpha)$.

\subsubsection{The Generalized Khovanov Bracket}

\begin{definition}
    The \defw{cube of resolutions} of a link diagram $D$ with $n$ crossings is an $n$-dimensional cube with the $2^n$ planar diagrams $D_\alpha$ for $\alpha\in\{0,1\}^n$ as its vertices and the corresponding cobordisms as its edges.
\end{definition}

Each 2-dimensional face of the cube of resolutions
corresponds to a sequence in $\{0,1,\star\}^n$ with exactly two $\star$'s, and thus
corresponds to an (unoriented) link diagram $D_{\dots{\star}\dots{\star}\dots}$
in which all but two of the crossings of $D$ are resolved.
In order to define a chain complex, we must make sure that all 2-dimensional
faces of the cube of resolutions anticommute, so that the
resulting differential squares to zero.

Each 2-dimensional face falls in one of a few basic types,
which are
shown in Figure~\ref{fig:commutativitytypes}.
In this figure, we are actually depicting the 2-crossing link diagrams $D_{\dots{\star}\dots{\star}\dots}$
corresponding to 2-dimensional faces. In the cases where we have
drawn a tangle diagram instead of a link diagram, the convention is that the tangle diagram
can be closed up with any crossingless tangle that connects
oppositely decorated endpoints.
In diagrams where we have drawn crossing orientations, it is understood
that one either has to choose the single arrows at both crossings, or the double arrows
at both crossings.
Moreover, all link diagrams in Figure~\ref{fig:commutativitytypes}
are to be considered up to planar isotopy.
As indicated in the figure, each face has an element $\sigma_{\dots{\star}\dots{\star}\dots}\in\{1,\pi\}$ assigned
to it that depends on whether the face is naively commuting or $\pi$-commuting.

\begin{figure}
\centering
\caption{Commutativity of faces for type $Y$ sign assignments}\label{fig:commutativitytypes}
\rule[0pt]{\textwidth}{1pt}\vspace{0.5em}
Commuting Faces (Group C) --- $\sigma_{\dots{\star}\dots{\star}\dots} = 1$\vspace{1em}

\hfill\includeTang{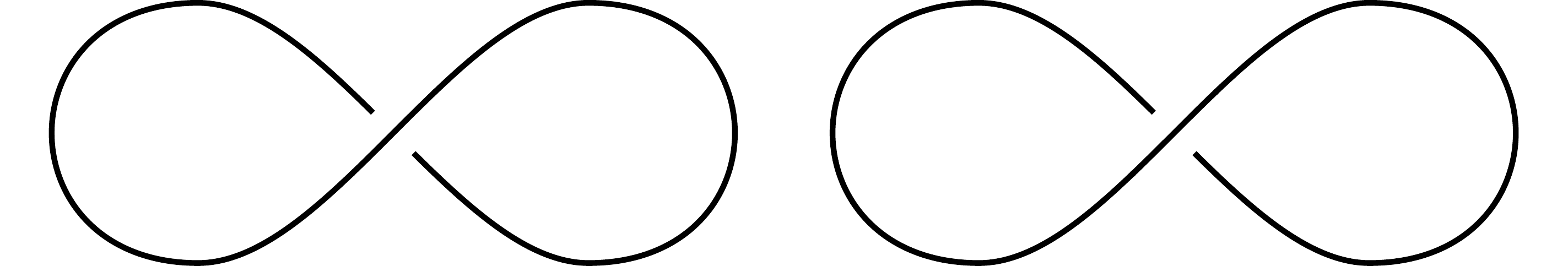}\hfill\includeTang{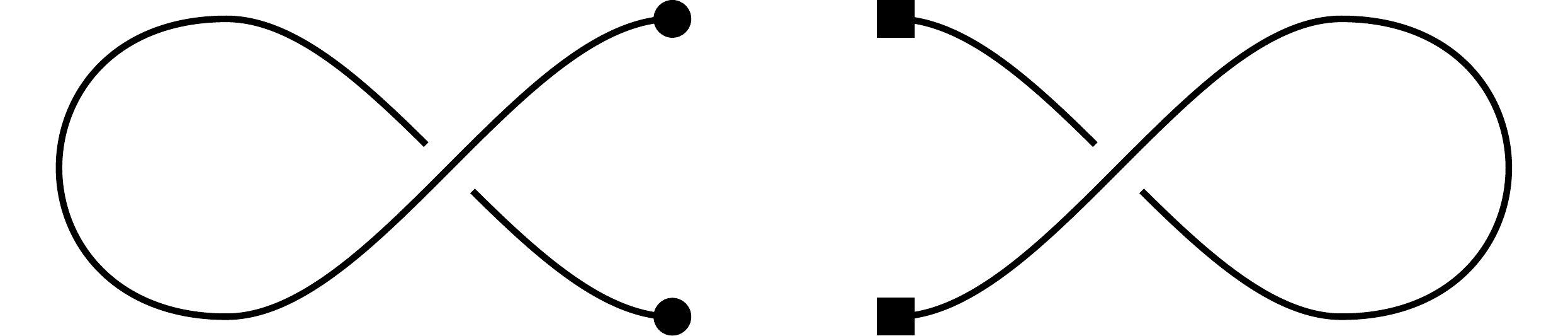}\hfill\phantom{.}\vspace{1em}

\includeTang{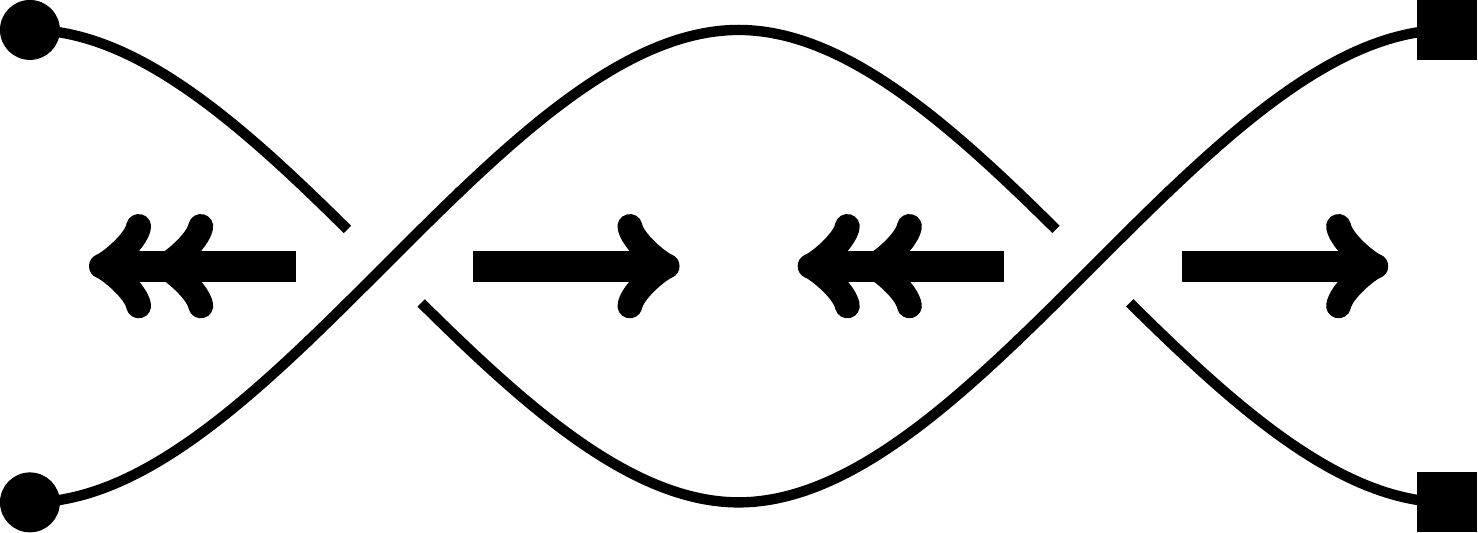} or \includeTang{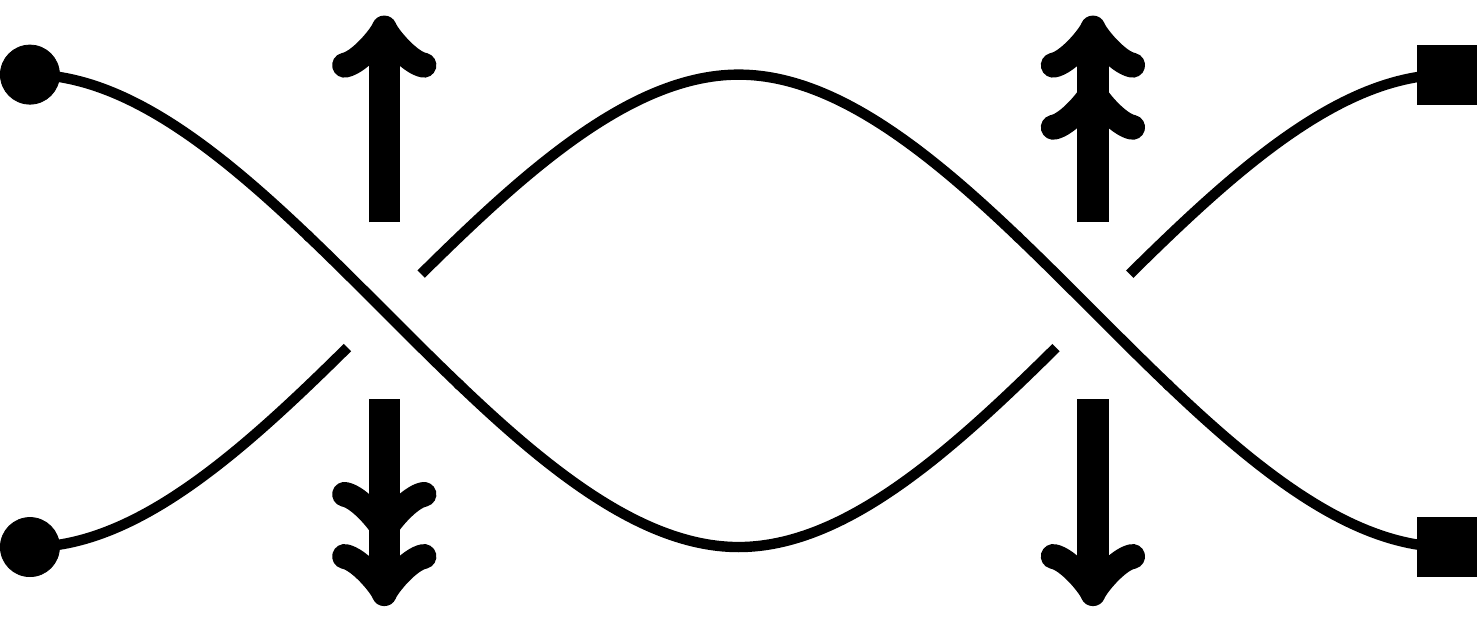}\vspace{1em}

\hfill\includeTang{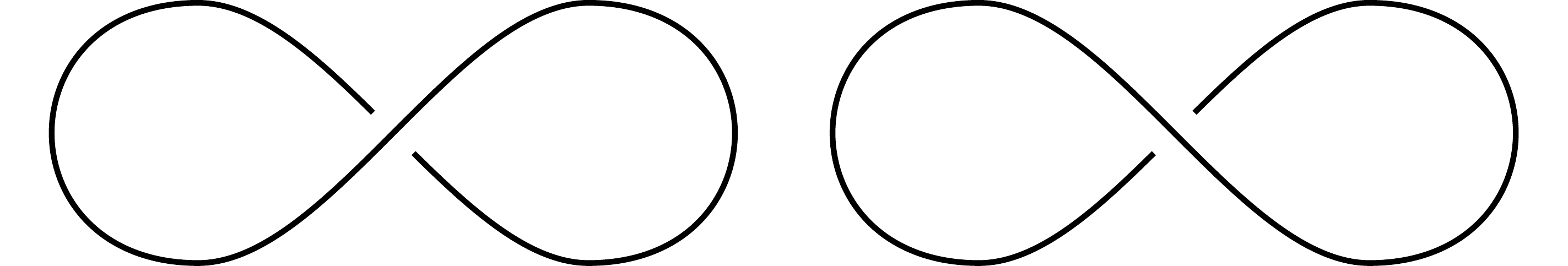}\hfill\includeTang{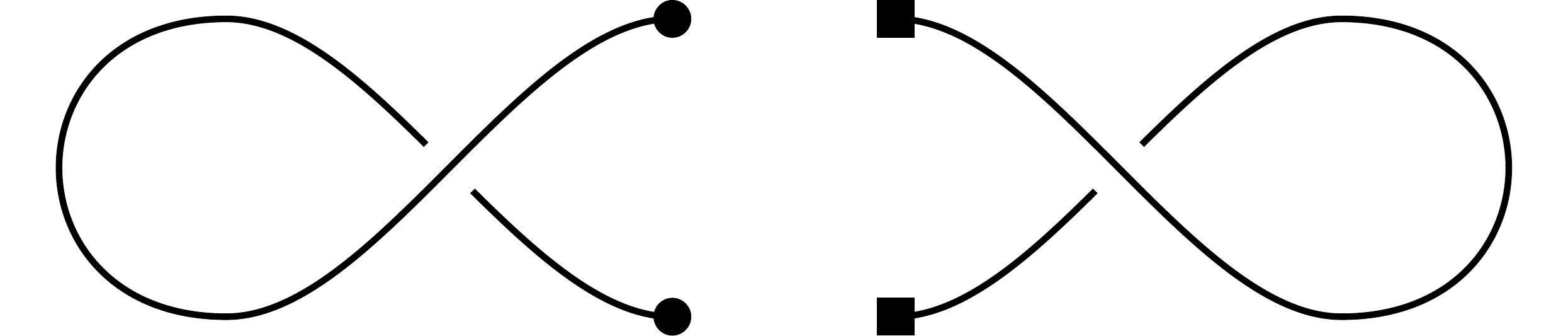}\hfill\phantom{.}\vspace{0.5em}
\rule[0pt]{\textwidth}{1pt}\vspace{0.5em}
$\pi$-commuting Faces (Group $\pi$) --- $\sigma_{\dots{\star}\dots{\star}\dots} = \pi$\vspace{1em}

\hfill\includeTang{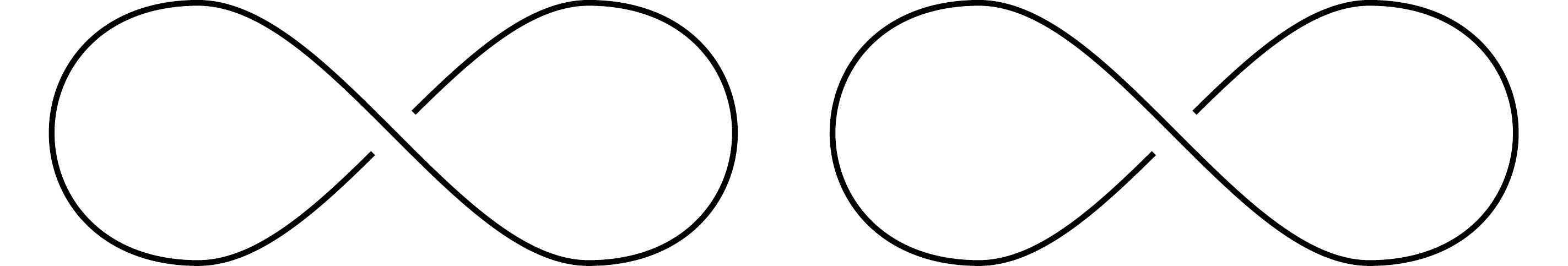}\hfill\includeTang{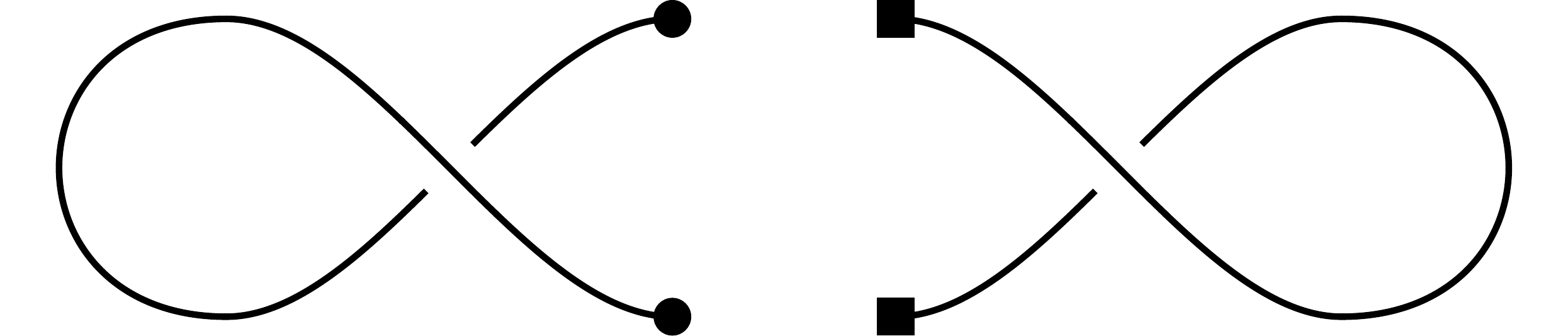}\hfill\phantom{.}\vspace{1em}

\includeTang{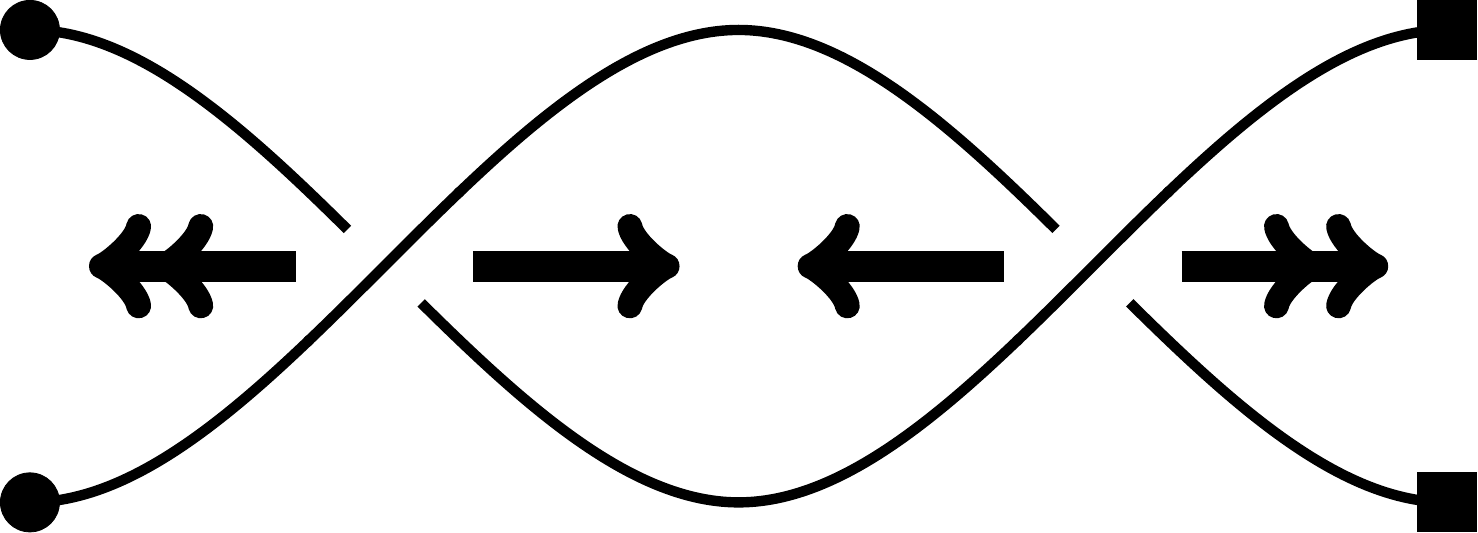} or \includeTang{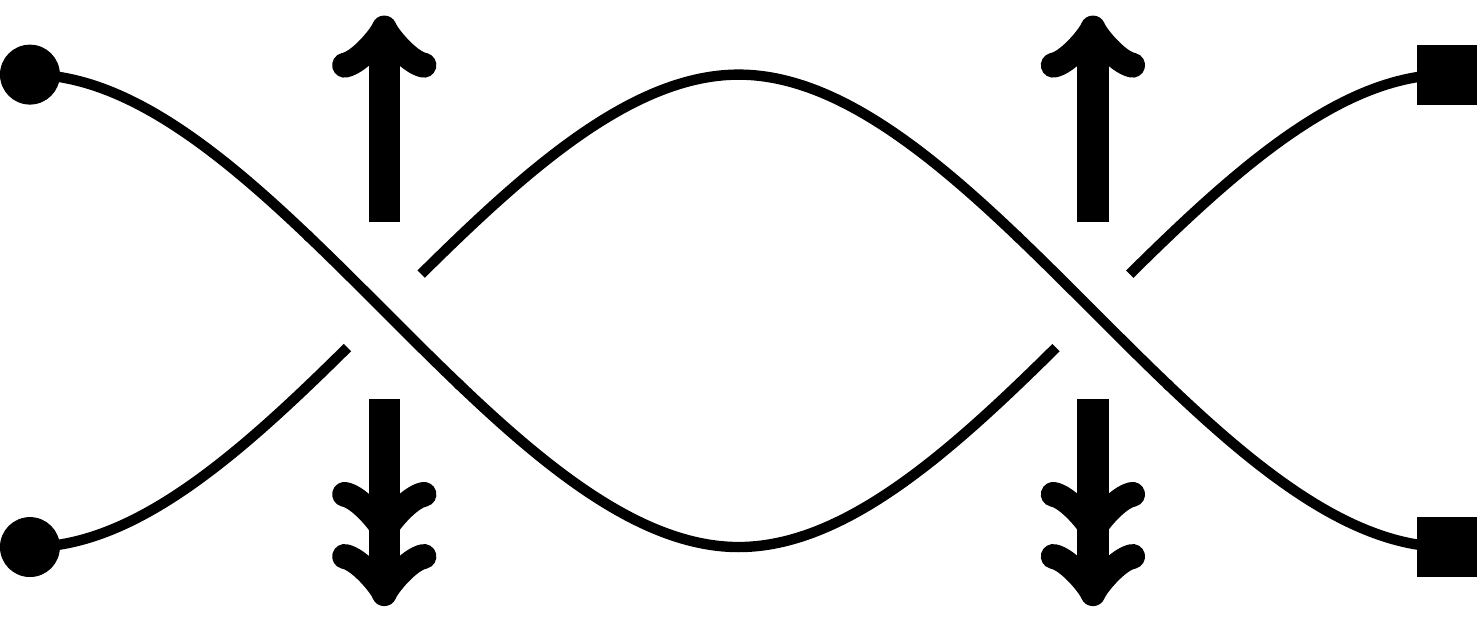}\vspace{0.5em}
\rule[0pt]{\textwidth}{1pt}\vspace{0.5em}
X Face (Group X) --- $\sigma_{\dots{\star}\dots{\star}\dots} = 1$\vspace{1em}

\includeTang{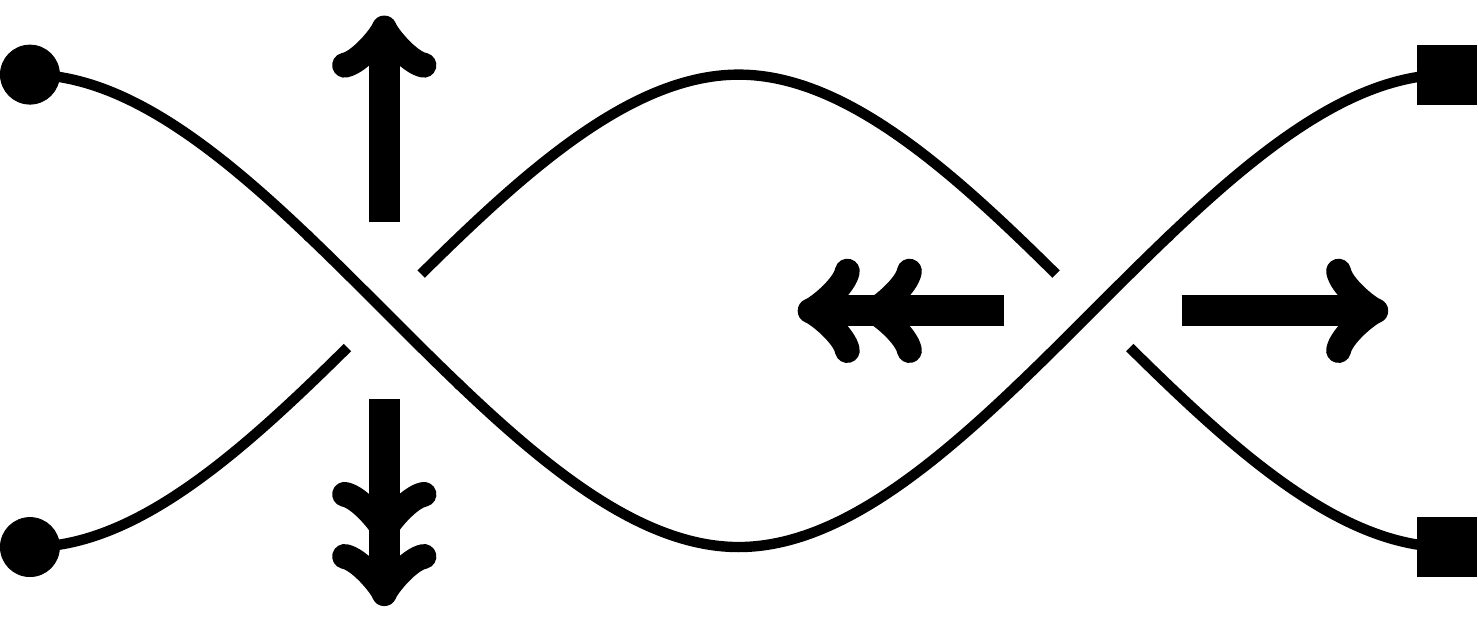} or \includeTang{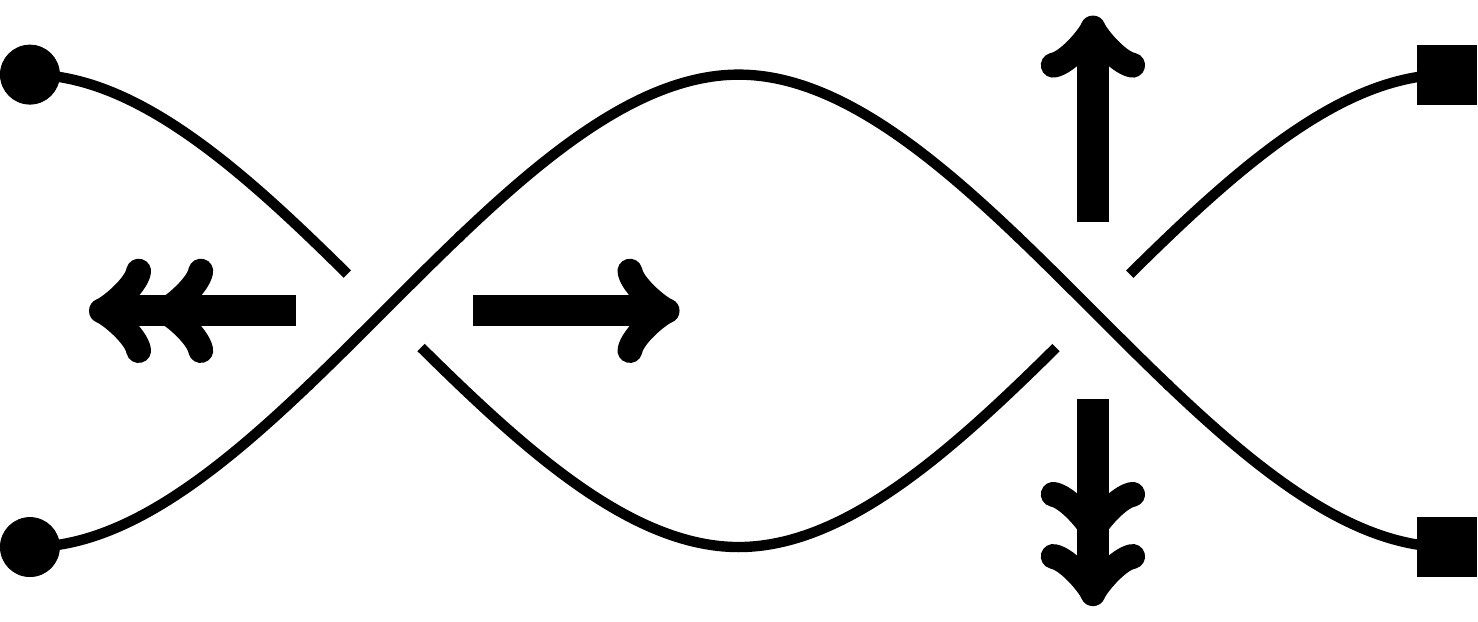}\vspace{0.5em}
\rule[0pt]{\textwidth}{1pt}\vspace{0.5em}
Y Face (Group Y) --- $\sigma_{\dots{\star}\dots{\star}\dots} = \pi$\vspace{1em}

\includeTang{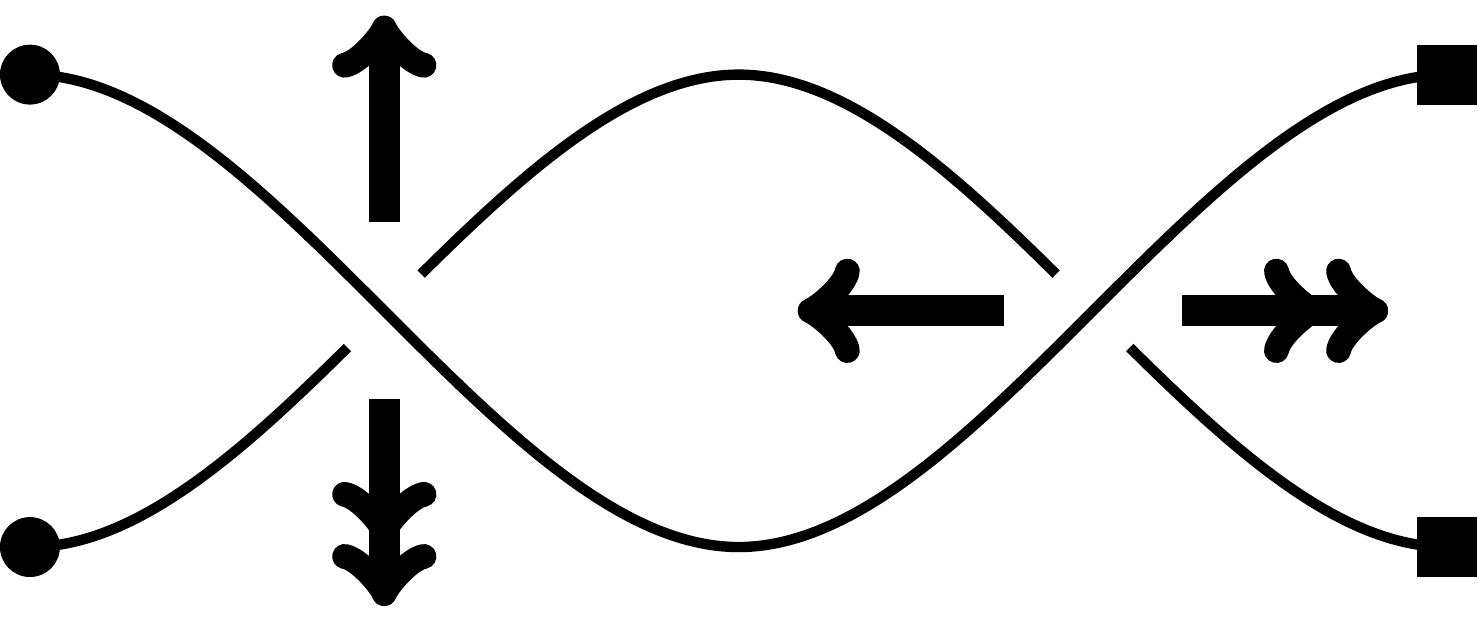} or \includeTang{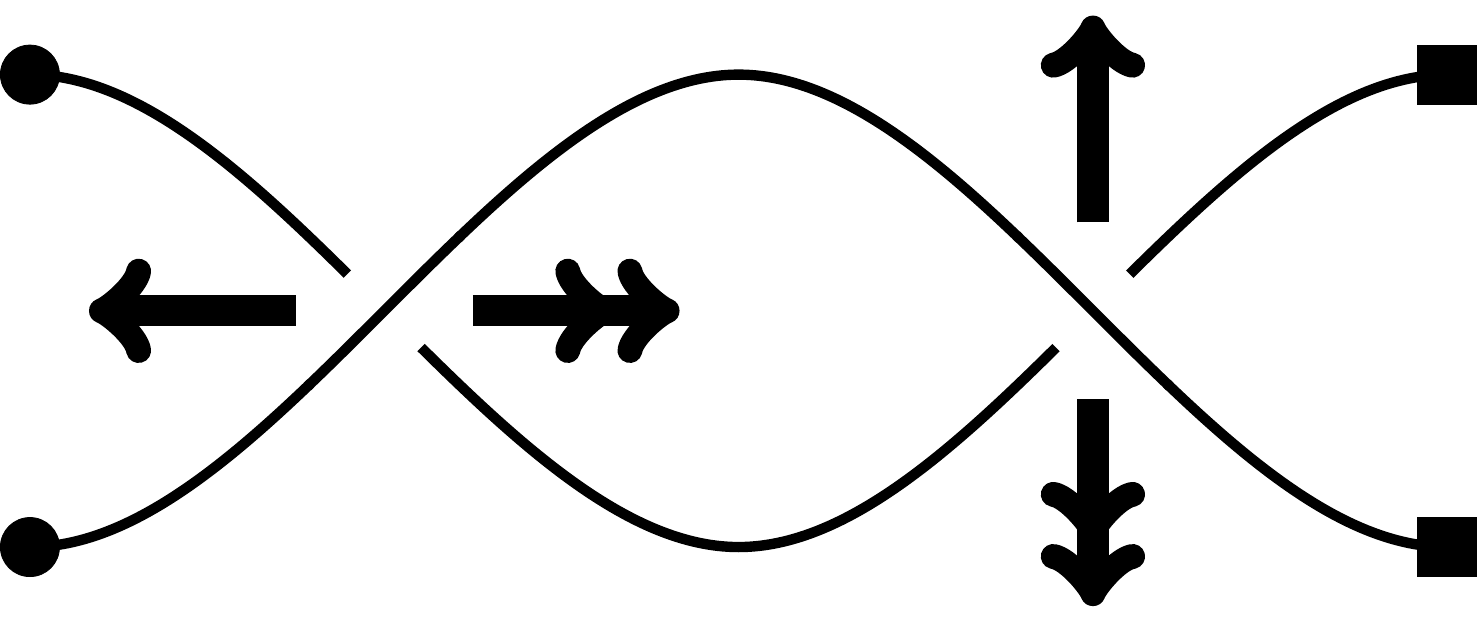}\vspace{0.5em}
\rule[0pt]{\textwidth}{1pt}
\end{figure}

Note that in Figure~\ref{fig:commutativitytypes}, we have depicted the conventions for type Y assignments, where faces of type X
are considered to be naively commuting, while faces of type Y are considered to be naively $\pi$-commuting.
For type X assignments, the latter conventions would be reversed.

To make squares in the cube of resolutions anticommute,
we need to make an assignment of an element $\epsilon_{\dots{\star}\dots}\in\Bbbk^\times=\{\pm 1,\pm\pi\}$ to each edge of the cube such that

\begin{equation}\label{eq:3:validFace}
    \displaystyle\prod_{\{\alpha_i,\alpha_j\}\in\{\{0,{\star}\},\{1,{\star}\},\{{\star},0\},\{{\star},1\}\}}\epsilon_{\dots{\alpha_i}\dots{\alpha_j}\dots} = -\sigma_{\dots{\star}\dots{\star}\dots}
\end{equation}%

Such a sign assignment can always be made (\cite{ORS2007},\cite{Pu2015}), but it is not unique.
This will be discussed in more detail in Subsection~\ref{subs:signs} below.

\begin{definition}
    The \defw{generalized Khovanov cube} of $D$ is the cube of resolutions with each edge
    map $d_{\dots{\star}\dots}$ multiplied by its assigned element
    $\epsilon_{\dots{\star}\dots}\in\{\pm 1,\pm\pi\}$.
\end{definition}

To obtain the generalized Khovanov bracket, we now \enquote{flatten}
the generalized Khovanov cube to a chain complex.
The $i$\ordth chain object in this complex is the formal direct sum
\[
\llbracket D\rrbracket^i:=\bigoplus_{\deg(\alpha)=i+n_-} D_\alpha,
\]
viewed as an object of $\mat{\cobIIIgen}$, where
$n_-$ denotes the number of negative crossings in $D$ (defined as in Figure~\ref{fig:2:crossingDirection}).
The $i$\ordth differential $\partial^i\colon\llbracket D\rrbracket^i\rightarrow\llbracket D\rrbracket^{i+1}$ 
is given by all edge maps
$\epsilon_{\dots{\star}\dots}d_{\dots{\star}\dots}\colon D_\alpha\rightarrow D_\beta$
in the generalized Khovanov cube
that start at a vertex $\alpha$ with $\deg(\alpha)=i+n_-$.
Note that we interpret the resulting generalized Khovanov bracket as an object of the category $\kom{\mat{\cobIIIgen}}$.

\begin{figure}[H]
	\centering
	\includeFig{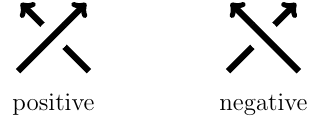}
	\caption{Positive and negative crossings}
		\label{fig:2:crossingDirection}
\end{figure}

\begin{convention}
    The underlying saddle cobordisms in the generalized Khovanov cube will be denoted with $d$ while the edge maps $\epsilon_{\ldots{\star}\ldots}d_{\ldots{\star}\ldots}$ will be denoted with $\partial$.
\end{convention}

\begin{convention}
    For an oriented link diagram $D$ with chosen crossing orientations, the generalized Khovanov bracket will be denoted $\llbracket D\rrbracket$. Additionally, the diagram at vertex $\alpha$ of the cube of resolutions will be denoted by $\llbracket D\rrbracket_\alpha$, so that $\llbracket D\rrbracket_\alpha:=D_\alpha$.
\end{convention}

Note that since all of the $d_{\ldots{\star}\ldots}$ are saddle cobordisms, the differential $\partial^i$ in the generalized
Khovanov bracket a priori has quantum degree $-1$. One can achieve that each $\partial^i$ has quantum degree zero
by formally raising the quantum grading of each $\llbracket D\rrbracket^i$ by $i+n_+-n_-=i+w(D)$ where 
$n_-$ denotes the number of negative crossings (as before), $n_+$ denotes the number of positive crossings, 
and $w(D):=n_+-n_-$ denotes the writhe of $D$.
This construction is described in detail in \cite{Bn2005}
and also ensures that the chain map induced by a link cobordism
is homogeneous of quantum degree given by the Euler characteristic of the link cobordism.
In this paper, we will not pursue this construction any further, but we will
describe an analogous construction, for the supergrading, in the next subsection.

\subsection{Supergraded Refinement}\label{subs:supergraded}

Recall that a $\Bbbk$-linear category is called a \defw{supercategory} if
its morphism sets are $\mathbb{Z}_2$-graded $\Bbbk$-modules
such that the $\mathbb{Z}_2$-grading is additive under composition of morphisms.
A supercategory $\mathcal{C}$ can be extended to a \enquote{refined} supercategory $\mathcal{C}^r$
in which objects come with formal grading shifts.
The objects in the refined supercategory $\mathcal{C}^r$ are symbols $x\{m\}$ with $x\in\mathcal{C}$
and $m\in\mathbb{Z}_2$, and the morphism sets are given by
\[
\operatorname{Hom}(x\{m\},y\{n\}):=\operatorname{Hom}(x,y),
\]
but with the $\mathbb{Z}_2$-grading shifted by $n-m$. Thus, if a morphism $f$
has degree $d\in\mathbb{Z}_2$ in $\operatorname{Hom}(x,y)$, then the same morphism has degree $d+n-m$
when viewed as a morphism in $\operatorname{Hom}(x\{m\},y\{n\})$.

The category $\cobIIIgen$ can be viewed as a supercategory with $\mathbb{Z}_2$-grading given by the superdegrading from Subsection~\ref{subs:targetcategory}.
Working in the refinement of $\cobIIIgen$, we can now define a \defw{refined generalized Khovanov bracket}.
This refined Khovanov bracket is obtained by replacing each $\llbracket D\rrbracket_\alpha=D_\alpha$
by $D_\alpha\{s_\alpha\}$, where $s_\alpha$ is the modulo 2 reduction of

\begin{equation}\label{eq:3:Smap}
S(D,\alpha):=\frac{c(D,\alpha)+\deg(\alpha)+n_+ - 2n_- -|D|}{2}.
\end{equation}

In this formula, $c(D,\alpha)$ denotes the number of circles in $D_\alpha$
and $|D|$ the number of components in the link represented by $D$.
The function $S(D,-)$ is defined such that across a merge differential in the cube of resolutions of $D$,
the increase in $\deg(\alpha)$ and the decrease in $c(D,\alpha)$ cancel out, while across
a split differential, they add up to increment $S(D,\alpha)$ by $1$.
As a consequence, the differentials in the refined generalized Khovanov bracket have
superdegree $0$.
The term
$n_+-2n_--|D|$ in~\eqref{eq:3:Smap} ensures that $S(D,\alpha)$ is an integer.

We note that there would be other (inequivalent) choices for the term $n_+-2n_--|D|$.
In particular, one could replace it by $\pm C(D,\beta)$ for any fixed resolution $\beta$,
such as the all-zero resolution.
However, our choice used in~\eqref{eq:3:Smap} is more natural because it will ensure
that the chain maps induced by oriented link cobordisms are homogeneous
of the correct superdegree (see Remark~\ref{rem:cobordismsuperdegree} below).
In this context, we point out that
the quantity $\deg(\alpha)+n_+ - 2n_-$ in equation \eqref{eq:3:Smap} is equal to the
shift of the quantum grading mentioned at the end of the previous subsection, or to
$i+w(D)$ where $i$ is the cohomological grading and $w(D)$ is the writhe.

In what follows, we 
will usually ignore the shifts of the supergrading, but invoke them whenever
it helps simplify our proofs.
Note that this approach will not limit the generality of our results because
any result proven for the refined generalized Khovanov bracket can be
translated back to the non-refined setting by applying a forgetful functor.
In the remainder of this subsection, we will discuss some more terminology
that will later be used in the refined context.

We recall that a $\Bbbk$-linear functor
between two supercategories is called a \defw{superfunctor}
if it preserves the $\mathbb{Z}_2$-grading. Let $\mathcal{F}$ and $\mathcal{G}$
be two superfunctors between two supercategories $\mathcal{C}$ and $\mathcal{D}$ and $\eta=\{\eta_x\}_{x\in\mathcal{C}}$ be
a collection of morphisms $\eta_x\colon\mathcal{F}(x)\rightarrow\mathcal{G}(x)$ in $\mathcal{D}$, one for
each object $x\in\mathcal{C}$. Assume that all $\eta_x$ are homogeneous of the same $\mathbb{Z}_2$-degree.
We will call such a collection a \defw{$\pi$-supernatural transformation}
if it satisfies
\begin{equation}\label{eqn:supernaturaldefinition}
\eta_y\circ \mathcal{F}(f)=\pi^d \mathcal{G}(f)\circ\eta_x
\end{equation}
for every homogeneuos morphism $f\colon x\rightarrow y$
in $\mathcal{C}$ of $\mathbb{Z}_2$-degree $d$. 
Note that this terminology is slightly different from the one used in \cite{BrundanEllis},
where a supernatural transformation is assumed to split into an even and an odd
component (with the two components behaving differently with respect to homogenous morphisms $f$).
We note that any superfunctor $\mathcal{F}$ between two supercategories $\mathcal{C}$ and $\mathcal{D}$ extends
to the refinements of these supercategories via $\mathcal{F}(x\{m\}):=\mathcal{F}(x)\{m\}$, and any
$\pi$-supernatural transformation $\eta$ between two superfunctors $\mathcal{F}$ and $\mathcal{G}$
extends to a $\pi$-supernatural transformation between the extended superfunctors
via $\eta_{x\{m\}}:=\pi^m\eta_x\colon\mathcal{F}(x)\{m\}\rightarrow\mathcal{G}(x)\{m\}$.

Given a supercategory $\mathcal{C}$,
let $\operatorname{Kom}'(\mathcal{C})$ denote the supercategory
whose objects are chain complexes in $\mathcal{C}$ with homogeneous differentials of $\mathbb{Z}_2$-degree zero,
and whose morphisms are (not necessarily homogenous) chain maps.
Any superfunctor $\mathcal{F}$ between two supercategories $\mathcal{C}$ and $\mathcal{D}$
induces a superfunctor between the supercategories $\operatorname{Kom}'(\mathcal{C})$ and $\operatorname{Kom}'(\mathcal{D})$,
where the induced superfunctor is given by applying $\mathcal{F}$ \enquote{termwise}.
Moreover, any $\pi$-supernatural transformation $\eta$ between two superfunctors
$\mathcal{F},\mathcal{G}\colon\mathcal{C}\rightarrow\mathcal{D}$ extends
to a $\pi$-supernatural transformation between the induced
superfunctors $\operatorname{Kom}'(\mathcal{C})\rightarrow\operatorname{Kom}'(\mathcal{D})$.
Explicitly, let $C=(\ldots \rightarrow C^i\rightarrow C^{i+1}\rightarrow\ldots)$
be a chain complex in $\mathcal{C}$ with differentials of $\mathbb{Z}_2$-degree zero. Then
the morphism $\eta_C$ of the extended $\pi$-supernatural transformation
is given by the chain map with components $\eta_{C^i}\colon\mathcal{F}(C^i)\rightarrow\mathcal{G}(C^i)$.

\subsection{Odd and Generalized Khovanov TQFTs}\label{subs:TQFTs}

There is a chronological TQFT functor $F\colon\cobIIIgen\rightarrow\mbox{Ab}$ which
sends the generalized Khovanov bracket of a link diagram to the odd Khovanov complex
\cite{ORS2007,Pu2015}.
This functor takes
a crossingless diagram $R\subset\mathbb{R}^2$ to the exterior algebra
\[
F(R):=\Lambda^*V(R)
\]
where $V(R)$ is the free $\mathbb{Z}$-module formally generated by the connected components of $R$.
If $S$ is a merge saddle between two components $R_1,R_2\subset R$
then the map $F(S)$ is defined 
 via the obvious quotient map $F(R)\rightarrow F(R)/(R_1=R_2)$.
 If $S$ is a split saddle creating $R_1,R_2$, then $F(S)$ is defined similarly,
 but via the map $F(R)/(R_1=R_2)\rightarrow F(R)$ induced by left-multiplication by $R_1-R_2$
(see \cite{ORS2007}  for details).

There is a generalized TQFT functor $F_{gen}$ which takes
values in the category of $\Bbbk$-modules, and which generalizes the odd TQFT functor
from above
and the even TQFT functor from \cite{Kh1999}.
The functor $F_{gen}$ maps the generalized Khovanov bracket
to the generalized Khovanov complex (see \cite{Pu2015}) and was originally
defined in the language of \cite{Bn2002}.
We will instead define the module $F(R)_{gen}$ intrinsically by replacing the exterior algebra $F(R)=\Lambda^*V(R)$ by the
noncommutative polynomial ring over $\Bbbk$ in the components of $R$, modulo the relations
\[
R_1R_1=0,\qquad R_1R_2=\pi R_2R_1
\]
where $R_1$ and $R_2$ are components of $R$. Note that
\[
(R_1+\pi R_2)R_1=\pi R_2R_1=R_1R_2=(R_1+\pi R_2)R_2,
\]
so that left-multiplication by $R_1+\pi R_2$ identifies $R_1$ and $R_2$.
The map that $F_{gen}$ assigns to a split saddle can thus be defined via the map
$F(R)_{gen}/(R_1=R_2)\rightarrow F(R)_{gen}$ induced by left-multiplication $R_1+\pi R_2$.
Here we are assuming that, near the split saddle, $R_1$ lies to the right
of $R_2$ when looking in the direction of the specified orientation of the split saddle.
The map that $F_{gen}$ assigns to a merge saddle is defined via the quotient map
$F(R)_{gen}\rightarrow F(R)_{gen}/(R_1=R_2)$, as before.

The module $F(R)_{gen}$ comes with
a quantum grading and a supergrading. 
If $R$ has connected components $R_1,\ldots,R_c$, then 
the quantum degree of a generator $R_{i_1}R_{i_2}\cdots R_{i_b}$
is given by \mbox{$c-2b$}, and the superdegree is given by the modulo 2 reduction of $b$. 
In the generalized Khovanov complex, these gradings are shifted by the
shifts described in the previous subsections. One can check that the resulting shifted superdegree
of a generator $g$ agrees with the modulo 2 reduction of $(j-|D|)/2$,
where $j$ is the (shifted) quantum degree of $g$, and $|D|$ is the number of link components.

\subsection{Well Definedness of the Generalized Khovanov Bracket}

To construct the generalized Khovanov bracket one must fix a diagram, fix an orientation on the crossings, and finally fix a sign assignment.  To show that the generalized Khovanov bracket is well defined it is sufficient to show that these choices do not affect the homotopy type of the generalized Khovanov bracket.

\subsection{Dependence on Sign Assignments and Crossing Orientations}\label{subs:signs}
The following lemma was shown in \cite{ORS2007,Pu2015}:

\begin{lemma}
    \label{lem:3:signinv}
    All valid sign assignments for a particular link diagram and fixed crossing orientations produce isomorphic generalized Khovanov brackets.
\end{lemma}

For later reference, we briefly recall the proof of this lemma.

\begin{proof}[Proof of Lemma~\ref{lem:3:signinv}.]
Fix a link diagram $L$ with $n$ crossings, and let $Q$ denote the cube $[0,1]^n$, equipped
with its usual CW structure. Further, let $o$ be a choice of crossing orientations for $L$.
If we regard a sign assignment for $(L,o)$ as a cellular $1$-cochain $\epsilon\in C^1(Q;\Bbbk^\times)$, then
condition \eqref{eq:3:validFace} can be written as
\begin{equation}\label{eq:3:validFacenew}
\delta\epsilon=-\sigma_{L,o},
\end{equation}
where $\sigma_{L,o}$ denotes the cellular 2-cochain on $Q$ defined by the $\sigma_{\ldots{\star}\dots{\star}\ldots}$. To prove the lemma,
we now note that any two solutions $\epsilon$ and $\epsilon'$ of \eqref{eq:3:validFacenew} must differ by a cocycle and hence, since $Q$
is contractible, by a coboundary $\delta\eta$ of a cellular $0$-cochain $\eta\in C^0(Q;\Bbbk^\times)$.
The desired chain isomorphism between the generalized Khovanov brackets of $(L,o,\epsilon)$ and $(L,o,\epsilon')$
is now given by the cubical chain map with components
 $\eta(\alpha)\operatorname{id}_{L_\alpha}\colon\llbracket L\rrbracket_\alpha\rightarrow\llbracket L\rrbracket_\alpha$
(where the term cubical means that the chain map has no nonzero components between different resolutions).
\end{proof}

In  \cite{ORS2007,Pu2015}, it was further shown that changing the orientation of a crossing does not change the generalized Khovanov bracket
as long as one adjusts the sign assignment accordingly. Lemma~\ref{lem:3:signinv} thus implies: 

\begin{lemma}
    \label{lem:3:orientinv}
    All choices of crossing orientations for a particular link diagram produce isomorphic generalized Khovanov brackets.
\end{lemma}

In summary, any two pairs $(o,\epsilon)$ and $(o',\epsilon')$ consisting of a choice crossing orientations and a valid sign assignment
for a link diagram $L$
give rise to a cubical chain isomorphism
\begin{equation}\label{eq:3:cohiso}
f_{o,\epsilon}^{o',\epsilon'}\colon\llbracket L,o,\epsilon\rrbracket\rightarrow \llbracket L,o',\epsilon'\rrbracket,
\end{equation}
where we are writing $\llbracket L,o,\epsilon\rrbracket$ for the generalized Khovanov bracket of $(L,o,\epsilon)$.
By construction or by Lemma~\ref{lem:3:uniqueness} below, this chain isomorphism is canonical up to an overall invertible scalar, and it
satisfies the coherence conditions
\begin{equation}\label{eq:3:cohisoconditions}
f_{o,\epsilon}^{o,\epsilon}=\operatorname{id}\qquad\mbox{and}\qquad f_{o',\epsilon'}^{o'',\epsilon''}\circ f_{o,\epsilon}^{o',\epsilon'}=f_{o,\epsilon}^{o'',\epsilon''}
\end{equation}
up to an overall invertible scalar.

\begin{lemma}\label{lem:3:uniqueness} Let $f,g\colon \llbracket L,o,\epsilon\rrbracket\rightarrow \llbracket L,o',\epsilon'\rrbracket$ be cubical chain isomorphisms with components
$f_\alpha=r_\alpha\operatorname{id}_{L_\alpha}$ and $g_\alpha=s_\alpha\operatorname{id}_{L_\alpha}$ for invertible scalars $r_\alpha,s_\alpha\in\Bbbk^\times$.
Then $f=c g$ for an invertible scalar $c\in\Bbbk^\times$.
\end{lemma}

\begin{proof} It suffices to show that if $f$ is a cubical chain isomorphism of the stated form, then
the scalar $r_\alpha$ appearing at a fixed vertex $\alpha$
determines the scalars $r_\beta$ at all other vertices $\beta$, in the sense that $r_\beta/r_\alpha=c_{\alpha\beta}$ for
a constant $c_{\alpha\beta}\in\Bbbk^\times$ which does not depend on $f$. Indeed, this will imply
$r_\beta/s_\beta=(c_{\alpha\beta}r_\alpha)/(c_{\alpha\beta}s_\alpha)=r_\alpha/s_\alpha=:c$ for all $\beta$ and thus $f=cg$.

Since any two vertices in the resolution cube of $L$ can be connected by a path of edges, it further suffices to consider the
case where the vertices $\alpha$ and $\beta$ are connected by an edge $e\colon\alpha\rightarrow\beta$. In this case,
we have the following commuting square, in which the horizontal arrows are the differentials in $\llbracket L,o,\epsilon\rrbracket$ and $\llbracket L,o',\epsilon'\rrbracket$ and the vertical arrows are the relevant components of $f$:
\[
\begin{tikzcd}
\llbracket L\rrbracket_\alpha \ar[r,"\epsilon_ed_e"]\ar[d,"r_\alpha\operatorname{id}_{L_\alpha}"']& \llbracket L\rrbracket_\beta\ar[d,"r_\beta\operatorname{id}_{L_\beta}"]\\
\llbracket L\rrbracket_\alpha \ar[r,"\epsilon'_ed'_e"']& \llbracket L\rrbracket_\beta
\end{tikzcd}
\]
By the remarks prior to Lemma~\ref{lem:3:orientinv}, we can assume without loss of generality that $o=o'$ and hence $d_e=d'_e$. The commutativity of
the square then implies $(r_\beta\epsilon_e-\epsilon'_er_\alpha)d_e=0$ and hence, since the saddle $d_e$ has trivial annihilator,
$r_\beta\epsilon_e=\epsilon'_er_\alpha$ or $r_\beta/r_\alpha=\epsilon'_e/\epsilon_e=:c_{\alpha\beta}$.
\end{proof}

If $L$, $o$, and $Q=[0,1]^n$ are as in the proof of Lemma~\ref{lem:3:signinv}, then the notion of a valid
sign assignment can be defined for any CW subcomplex $Q'\subseteq Q$. Specifically, a valid sign assignment
on such a subcomplex $Q'$ is given by a cellular 1-cochain $\epsilon'\in C^1(Q';\Bbbk^\times)$ such that
$\epsilon'=-\sigma_{L,o}|_{Q'}$. For later use, we will prove the following lemma.

\begin{lemma}\label{lem:3:CWsubcomplex} Let $Q'\subseteq Q$ be a CW subcomplex with trivial first homology.
Then any valid sign assignment on $Q'$ extends to a valid sign assignment on $Q$.
\end{lemma}

\begin{proof} Let $\epsilon'$ be a valid sign assignment on $Q'$, and choose any valid sign assignment $\epsilon$ on $Q$. Then $\epsilon'$ and $\epsilon|_{Q'}$ differ by a cocycle on $Q'$,
and the assumption on $Q'$ implies that this cocycle is a coboundary, $\delta\eta'$, for a cellular $0$-cochain $\eta'\in C^0(Q';\Bbbk^\times)$.
To obtain a valid sign assignment on $Q$ that extends $\epsilon'$, we now modify $\epsilon$ by $\delta\eta$, where
$\eta\in C^0(Q;\Bbbk^\times)$ is any extension of $\eta'$ to $Q$.
\end{proof}

In the above proof, we can choose the 0-cochain $\eta$ to be trivial on all vertices of $Q\setminus Q'$,
in which case the modified $\epsilon$ agrees with the unmodified $\epsilon$ on
all 1-cells of $Q$ whose closures are disjoint from $Q'$. If we are given any valid sign assignment $\epsilon$ on $Q$, we can thus assume that the extended sign assignment from
Lemma~\ref{lem:3:CWsubcomplex} agrees with $\epsilon$ on all such 1-cells.

\subsection{Generalized Khovanov homology as a diagram and as a colimit}\label{subs:colimit}

For a fixed link diagram $L$, let $M(L)$ denote the set of all pairs $(o,\epsilon)$
where $o$ is a choice of crossing orientations for $L$ and $\epsilon$ is a valid sign assignment for $(L,o)$:
\[
M(L):=\left\{(o,\epsilon)\,\middle\vert\,\parbox{2.55in}{\centering $o$ a choice of crossing orientations for $L$ and $\delta\epsilon=-\sigma_{L,o}$}\,\right\}.
\]
We can repackage the generalized Khovanov brackets $\llbracket L,o,\epsilon\rrbracket$
for the various possible choices of $(o,\epsilon)\in M(L)$
into a single link invariant $\llbracket L\rrbracket_M$,
which is itself functorial.
To define this invariant, let $M(L)'$ denote the grouped whose objects are the elements
of $M(L)$ and which has a single morphism between any two objects. Further, let $M(L)''$
be the set of all functors from
the groupoid $M(L)'$ to $\kombpmp{\mathcal{C}hron\mathcal{C}ob^3_{gen}}$. We then define
\[
\llbracket L\rrbracket_M\in M(L)''
\]
as the functor which sends an object $(o,\epsilon)\in M(L)'$ to the generalized
Khovanov bracket of $(L,o,\epsilon)$ and the unique morphism $(o,\epsilon)\rightarrow(o',\epsilon')$
to the isomorphism from \eqref{eq:3:cohiso}. Note that
this indeed yields a functor because of the conditions in \eqref{eq:3:cohisoconditions}.
Equivalently, we can view this functor as a diagram 
of shape $M(L)'$ in the category $\kombpmp{\mathcal{C}hron\mathcal{C}ob^3_{gen}}$.

As we shall see, the chain maps assigned to elementary link cobordisms
intertwine with the isomorphisms from \eqref{eq:3:cohiso}.
We can rephrase this property by introducing a category
$M''$ of functors valued in $\kombpmp{\mathcal{C}hron\mathcal{C}ob^3_{gen}}$,
where the morphisms between two such functors with source categories $I$ and $J$
is given by the natural transformations between the lifts of these functors
to the product category $I\times J$. A smooth link cobordism between two links with 
diagrams
$L_0,L_1\subset\mathbb{R}^2$ then induces a morphism $\llbracket L_0\rrbracket_M\rightarrow\llbracket L_1\rrbracket_M$
in this category.

By adapting an idea from \cite{BHL2019}, we can further extend the construction of $\llbracket L\rrbracket_M$
to arbitrary links in $\mathbb{R}^3$, including those that are not in general position
with respect to the projection onto the $xy$-plane. For this, consider an arbitrary link $L\subset\mathbb{R}^3$,
and let $N(L)$ denote the set of all triples $(\phi,o,\epsilon)$ such that
$\phi$ is an ambient isotopy of $\mathbb{R}^3$ taking the link $L=\phi_0(L)$ to a link
$\phi_1(L)$ in general position; and $(o,\epsilon)$ is an element of $M(D)$ where $D$ is the link diagram
of $\phi_1(L)$.

Given two triples $(\phi,o,\epsilon),(\phi',o',\epsilon')\in N(L)$, we can perform the isotopies $\phi_{1-t}$
and $\phi'_t$ one after the other to obtain an isotopy from $\phi_1(L)$ to $\phi'_1(L)$.
The latter isotopy corresponds to a link cobordism, which
induces a chain isomorphism
\begin{equation}\label{eq:3:cohisoN}
f_{\phi,o,\epsilon}^{\phi',o',\epsilon'}\colon\llbracket D,o,\epsilon\rrbracket\longrightarrow\llbracket D',o',\epsilon'\rrbracket
\end{equation}
where $D$ and $D'$ denote planar diagrams of the links $\phi_1(L)$ and $\phi'_1(L)$, respectively.
As we did for $M(L)$, we can now replace $N(L)$ by the corresponding groupoid $N(L)'$.
The generalized Khovanov brackets of the $(D,o,\epsilon)$ together with the isomorphisms
from \eqref{eq:3:cohisoN} then determine a diagram of shape $N(L)'$ or, equivalently, an element
\[
\llbracket L\rrbracket_N\in N(L)''
\]
in the set $N(L)''$ of functors from $N(L)'$ to
to $\kombpmp{\mathcal{C}hron\mathcal{C}ob^3_{gen}}$.
By the results of this paper, $\llbracket L\rrbracket_N$ is functorial
in the sense that any smooth link cobordism between two links $L_0,L_1\subset\mathbb{R}^3$ induces a morphism $\llbracket L_0\rrbracket_N\rightarrow\llbracket L_1\rrbracket_N$ in the category $M''$ defined above.

We can obtain a more concrete link invariant
by fixing a bidegree $(i,j)\in\mathbb{Z}^2$ and replacing 
each generalized Khovanov bracket $\llbracket D,o,\epsilon\rrbracket$ that appears in the above construction by 
the corresponding generalized Khovanov homology group in bidegree $(i,j)$ (here $i$ denotes the cohomological
grading and $j$ denotes the quantum grading).
After modding out by the multiplicative action of $\Bbbk^\times$, these generalized Khovanov homology groups
become sets $\mathcal{S}(D,o,\epsilon,i,j)$, 
and thus the invariant $\llbracket L\rrbracket_N$ becomes a diagram in the category of sets.
Taking the colimit of this diagram now yields a set
\[\mathcal{S}(L,i,j):=\left.\left(\quad\,\,\,\,\bigsqcup_{\mathclap{(\phi,o,\epsilon)\in N(L)}} \mathcal{S}(D,o,\epsilon,i,j)\right)\middle /
\left(x\sim g_{\phi,o,\epsilon}^{\phi',o',\epsilon'}(x)\right)\right.\]
where $g_{\phi,o,\epsilon}^{\phi',o',\epsilon'}$ denotes the map $\mathcal{S}(D,o,\epsilon,i,j)\rightarrow \mathcal{S}(D',o',\epsilon',i,j)$ that corresponds to the isomorphism from \eqref{eq:3:cohisoN}.
The set $\mathcal{S}(L,i,j)$ can be defined for arbitrary links in $\mathbb{R}^3$, but not for arbitrary links in $S^3$.
By the results of this paper, it is functorial in the sense that any smooth link cobordism $F\colon L\rightarrow L'$ in $\mathbb{R}^3\times[0,1]$ induces a map $\mathcal{S}(L,i,j)\rightarrow \mathcal{S}(L',i,j+\chi(F))$.

Note that in the remainder of this paper, we will almost always ignore the constructions of this subsection and instead work with the generalized Khovanov bracket $\llbracket L\rrbracket:=\llbracket L,o,\epsilon\rrbracket$
for a fixed link diagram $L\subset\mathbb{R}^2$ and a fixed choice of $(o,\epsilon)\in M(L)$.

\subsection{Invariance under Reidemeister moves}

Since we will need the relevant chain maps later, we will briefly summarize the proof that
the generalized Khovanov bracket is invariant under Reidemeister moves, up to homotopy equivalence.

Note that in the pictures in this subsection, all critical points are assumed to be oriented as in Convention~\ref{conv:2:orientations}, so that
deaths are oriented clockwise, and
 saddles are oriented to the right or to the front. In diagrams \eqref{eq:3:RImap} and \eqref{eq:3:RImapalt} below, we further assume
that the terms sitting vertically above each other are equipped with the same internal sign assignments, and
in diagram \eqref{eq:3:RIImap}, we make the same assumption about the terms that are connected by an identity map.

Lemmas~\ref{lem:3:RIinv} through \ref{lem:3:RIIIinv} below were shown by Putyra \cite{Pu2015}.

\begin{lemma}[Reidemeister I invariance]\label{lem:3:RIinv}
    The chain maps
 depicted in \eqref{eq:3:RImap} and \eqref{eq:3:RImapalt} form strong deformation retractions from the top complexes to the bottom complexes.
\end{lemma}
\begin{equation}
    \label{eq:3:RImap}
    \includeFig{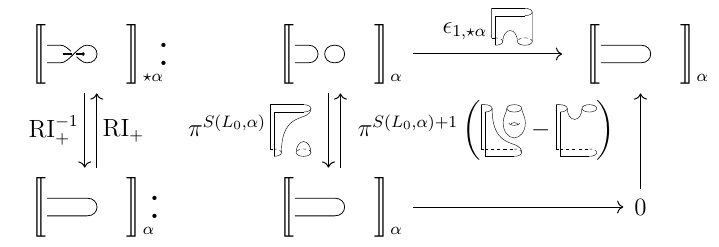}
\end{equation}
\begin{equation}
    \label{eq:3:RImapalt}
    \includeFig{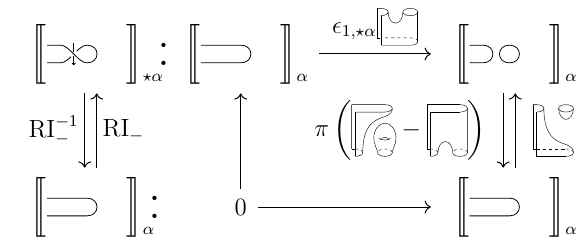}
\end{equation}

In \eqref{eq:3:RImap}, the quantity $S(L_0,\alpha)$ is defined
as in \eqref{eq:3:Smap}.
If in the diagrams in \eqref{eq:3:RImap} and \eqref{eq:3:RImapalt} the orientation of the crossing is flipped, then the orientations
of the saddles in the differentials will
be flipped as well, but the orientations of the saddles in the chain maps would
still the be the ones dictated by Convention~\ref{conv:2:orientations}.

\begin{lemma}[Reidemeister II invariance]\label{lem:3:RIIinv}
    The chain maps depicted in \eqref{eq:3:RIImap} form a strong deformation retraction from the bottom complex to the top complex.
\end{lemma}
\begin{equation}
    \label{eq:3:RIImap}
    \hspace{-1.5em}\includeFig{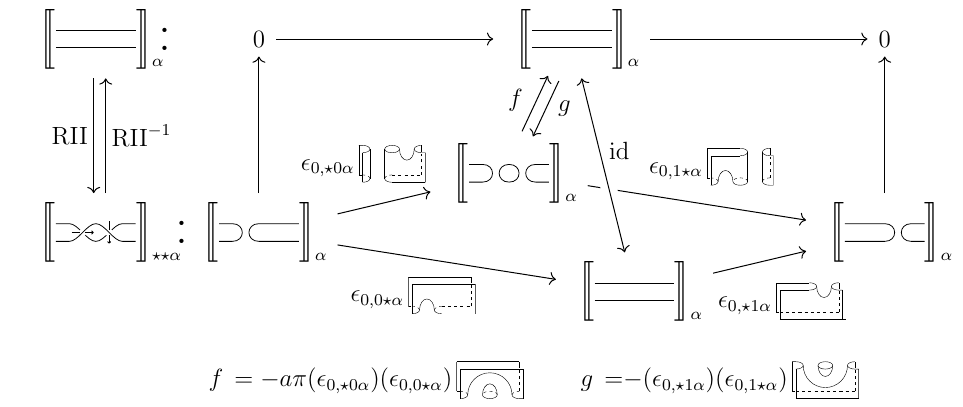}\hspace{-1.5em}
\end{equation}
In the above definition of $f$, the constant $a$ takes on the value 1 if the right crossing is oriented downward---as shown in the diagram---and the value $\pi$ if it is oriented upward. The saddles in $f$ and $g$ inherit their orientations
from the directions of the arrows at the left and right crossing, respectively, and the death in $f$ is always oriented clockwise.

To prove invariance under Reidemeister III moves,
Putyra uses that the generalized Khovanov brackets on the two sides of the 
move
can be viewed as mapping cones, as shown below.
\begin{equation}
    \includeFig{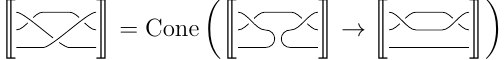}
\end{equation}
\begin{equation}
    \includeFig{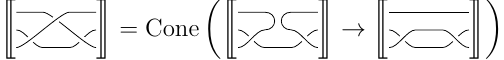}
\end{equation}
The homotopy type of a mapping cone of a chain map does not change if
one composes the chain map with a homotopy equivalence.
Lemma~\ref{lem:3:RIIinv} thus implies that the above mapping cones are homotopy equivalent to
the ones below (where the rightmost maps in \eqref{eq:3:conesL} and \eqref{eq:3:conesR}
are induced by inverse Reidemeister II moves):
\begin{equation}
    \label{eq:3:conesL}
    \includeFig{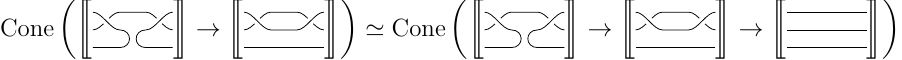}
\end{equation}
\begin{equation}
    \label{eq:3:conesR}
    \includeFig{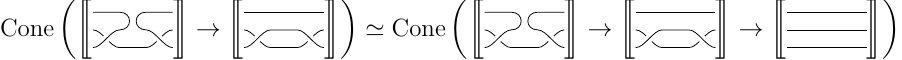}
\end{equation}

Now let $\Phi_L$ and $\Phi_R$ 
denote the composed chain maps that appear in the cones on the right-hand sides of 
\eqref{eq:3:conesL} and \eqref{eq:3:conesR}, respectively. After choosing sign assignments appropriately,
we can assume that $\Phi_L$ and $\Phi_R$ have the same source complex
and the same target complex. Invariance under Reidemeister III moves then follows
from the lemma below.

\begin{lemma}[Reidemeister III invariance]\label{lem:3:RIIIinv}
$\Phi_L$ and $\Phi_R$ are identical, up to a possible overall invertible scalar.
\end{lemma}

Since Putyra's proof of this lemma omitted some of the details, we here give
a more thorough argument, based
on the proof of a corresponding result from \cite{ORS2007}.

\begin{proof}[Proof of Lemma~\ref{lem:3:RIIIinv}]

The diagrams in \eqref{eq:3:RIIIL} and \eqref{eq:3:RIIIR} below depict the maps $\Phi_L$ and $\Phi_R$ more explicitly.
For the present argument, we need to consider the maps which travel up these diagrams. There is a second
version of the Reidemeister III move where the crossing that travels over the lowest strand is reversed, so
that the roles of its $0$- and $1$-resolution are exchanged. For that version of the move,
one would need to consider the maps 
traveling down the diagrams, but the argument
would otherwise be essentially identical.

\begin{equation}
    \label{eq:3:RIIIL}
    \hspace{-1em}\includeFig{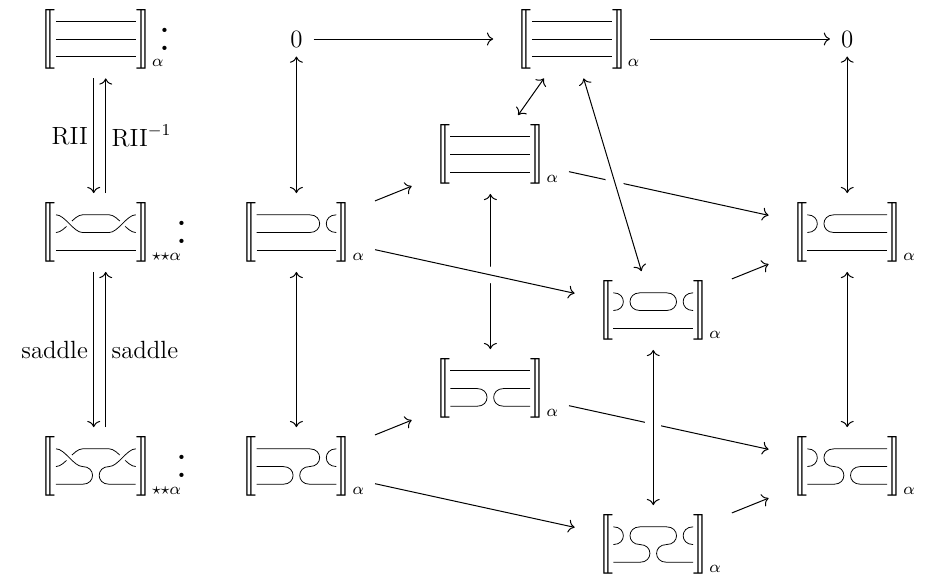}\hspace{-1em}
\end{equation}
\begin{equation}
    \label{eq:3:RIIIR}
    \hspace{-1em}\includeFig{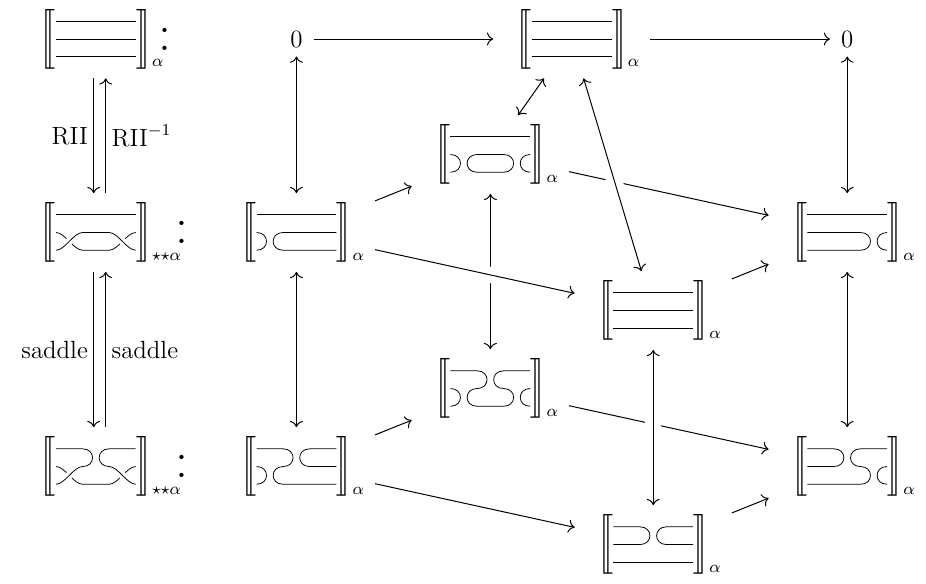}\hspace{-1em}
\end{equation}

A careful investigation of the chain maps in \eqref{eq:3:RIImap}, \eqref{eq:3:RIIIL}, and \eqref{eq:3:RIIIR} now
reveals that the nonzero components of
$\Phi_L$ and $\Phi_R$ (and of the maps in reverse direction) are given by the following underlying cobordisms:

\begin{minipage}{0.5\textwidth}
    \begin{equation}
        \includeFig{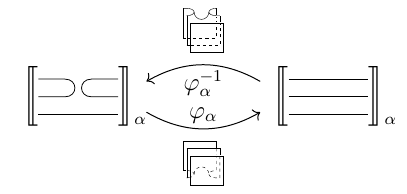}
    \end{equation}
\end{minipage}%
\begin{minipage}{0.5\textwidth}
    \begin{equation}
        \includeFig{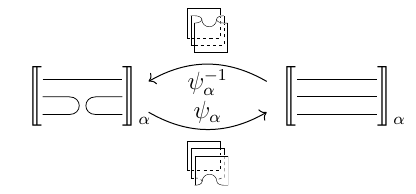}
    \end{equation}
\end{minipage}

This shows that the maps $\Phi_L$ and $\Phi_R$ are the same when coefficients in $\Bbbk^\times$ are stripped away.
What remains is to show that the coefficients in the cone in \eqref{eq:3:conesL} are consistent with those in the cone in \eqref{eq:3:conesR}.
To see this, note that the cone on the right-hand side of \eqref{eq:3:conesL} contains two obvious quotient complexes,
which are both cubical. The first of these is given by the bottom layer of \eqref{eq:3:RIIIL},
whereas the second one is obtained from the entire mapping cone by removing the term that sits at the rightmost vertex in the
bottom layer of \eqref{eq:3:RIIIL}.

Corresponding to these two quotient complexes, we consider two $(n-1)$-dimensional cubes $Q\cong[0,1]^{n-1}$ and $Q'\cong[0,1]^{n-1}$, where $n$ denotes the number of crossings in the link diagrams that show up on either side of the Reidemeister III move.
Let $P\coloneqq Q\cup Q'$ be the CW complex obtained by gluing the cubes $Q$ and $Q'$ along the two $(n-2)$-dimensional subcubes that correspond to the two edges on the left side of the bottom layer of  \eqref{eq:3:RIIIL}.

It is easy to see that the signs coming from \eqref{eq:3:conesL} define a valid sign assignment on $P$, in an obvious sense.
We can think of this sign assignment as a cellular 1-cochain $\epsilon\in C^1(P;\Bbbk^\times)$, and there is a similar
cellular 1-cochain $\epsilon'\in C^1(P;\Bbbk^\times)$ coming from the signs in \eqref{eq:3:conesR}.
Because of our choices, it follows that $\epsilon$ and $\epsilon'$ must differ by a relative cocycle
in $Z^1(P,Q\cup T;\Bbbk^\times)$, where $T\cong[0,1]^{n-2}$ corresponds to the top-most term in \eqref{eq:3:RIIIL}.
The lemma now follows because $Z^1(P,Q\cup T;\Bbbk^\times)\cong H^1(P/(Q\cup T);\Bbbk^\times)\cong H^1(S^1;\Bbbk^\times)\cong\Bbbk^\times$.
\end{proof}

While in principle, there could be many different homotopy equivalences between the complexes on the two sides of a Reidemeister move,
the specific homotopy equivalences considered in this subsection are unique (or unique up to overall invertible scalars).
In each of the moves, we placed some restrictions on the relationship between the chosen sign assignments before and after the move.
These restrictions can easily be lifted by composing the chain maps discussed in this subsection with the isomorphisms
from \eqref{eq:3:cohiso}. The resulting chain maps are still unique up to invertible scalars and natural with respect to
the choice of $(o,\epsilon)$, 
and thus induce well-defined morphisms $\llbracket L_0\rrbracket_M\rightarrow\llbracket L_1\rrbracket_M$ in the category $M''$ from
the previous subsection.

\subsection{Equivalence of type X and type Y assignments}\label{subs:XYequivalence}

We still need to prove Theorem~\ref{thm:equivalence} about the equivalence of type X and type Y theories. 

In the following, let $D\subset\mathbb{R}^2$ be a link diagram in the $xy$-plane representing a link $L\subset\mathbb{R}^3$.
Moreover, let $D'$ be the diagram of the rotated link $L':=\rho_1(L)$, where
$\rho_t\colon\mathbb{R}^3\rightarrow\mathbb{R}^3$
is the rotation by $180t$ degrees about the $y$-axis $\lambda\subset\mathbb{R}^3$.
Note that, on the level of diagrams in the $xy$-plane, $D'$ can be obtained from $D$ by first reflecting
$D$ along $\lambda$, and then reversing the over/undercrossing information
at each crossing of the resulting reflected diagram.

Let $o$ be a choice of crossing orientations for $D$, and $o'$ be the corresponding
reflected choice for $D'$. Moreover, let $\epsilon$ and $\epsilon'$ be fixed
type X sign assignments for $(D,o)$ and $(D',o')$, respectively.
The links $\rho_t(L)$ for $t\in[0,1]$ define an ambient isotopy from $L$ to $L'$, and
there is thus an induced homotopy equivalence
\begin{equation}\label{eq:XY1}
\llbracket D,o,\epsilon\rrbracket\simeq\llbracket D',o',\epsilon'\rrbracket,
\end{equation}
which is defined uniquely up to homotopy and invertible scalars.
Now consider any smooth link cobordism $F\colon L_0\rightarrow L_1$, and let
\begin{equation}\label{eq:XY2}
F':=(\rho_1\times\operatorname{id}_I)(F)\subset\mathbb{R}^3\times I
\end{equation}
denote the corresponding rotated link cobordism
between the rotated links $L_0'$ and $L_1'$.
Further, let $\Phi^Y_F$ and $\Phi^Y_{F'}$ denote the type X maps that $F$ and $F'$ induce between
the generalized Khovanov brackets of type Y. We then claim:

\begin{lemma}\label{lem:XY1}
The homotopy equivalence in \eqref{eq:XY1} is natural, in the sense that it intertwines
the maps $\Phi^Y_F$ and $\Phi^Y_{F'}$ up to homotopy and an overall invertible scalar.
\end{lemma}

\begin{proof} This follows directly from the functoriality of the generalized Khovanov bracket
because the relevant compositions that need to be compared to establish the naturality of \eqref{eq:XY1} are induced by ambient
isotopic link cobordisms.
\end{proof}

Now note that the crossings of $D'$ correspond canonically to those of $D$,
and, under this correspondence, the type Y assignment $\epsilon'$ for $(D',o')$ can be seen as a type X
assignment for the original diagram $(D,o)$. We also have
\begin{equation}\label{eq:XY3}
\mathcal{R}(\llbracket D',o',\epsilon'\rrbracket)=\llbracket D,o,\epsilon'\rrbracket,
\end{equation}
where $\mathcal{R}$ denotes the endofunctor of $\komb{\mathcal{C}hron\mathcal{C}ob^3_{gen}}$ given by reflecting
objects across $\lambda\subset\mathbb{R}^2$, and chronological cobordisms $S\subset\mathbb{R}^2\times I$ (including the orientations of their critical points)
across $\lambda\times I\subset\mathbb{R}^2\times I$.
Given a link cobordism $F$, let $\Phi^X_F$ denote the type X map that $F$ induces
on the generalized Khovanov bracket of type X.
Assuming that $F'$ is the rotated link cobordism as in equation \eqref{eq:XY2}, we have:

\begin{lemma}\label{lem:XY2}
The equality in \eqref{eq:XY3} is natural, in the sense that
$\mathcal{R}(\Phi^Y_{F'})$ coincides with $\Phi^X_F$ up to an overall invertible scalar.
\end{lemma}

\begin{proof} It suffices to prove this when $F$ is an elementary link cobordism.
If $F$ is a birth, then the statement is obvious. If $F$ is a death, then
$\mathcal{R}(\Phi^Y_{F'})$ and $\Phi^X_F$ differ by a factor of $\pi$ because $\mathcal{R}$ reverses
the orientation of each planar death cobordism that appears in the chain map $\Phi^Y_{F'}$.

If $F$ is a saddle cobordism, then the chain maps $\Phi^X_F$ and $\Phi^Y_{F'}$
are defined via type X and type Y assignments $\widetilde{\epsilon}$ and $\widetilde{\epsilon}'$
on the cubes of larger link diagrams $(\widetilde{D},\widetilde{o})$ and $(\widetilde{D}',\widetilde{o}')$,
respectively. In this case, the lemma follows because $\mathcal{R}$ takes
$\llbracket\widetilde{D},\widetilde{o},\widetilde{\epsilon}\rrbracket$ to
$\llbracket\widetilde{D}',\widetilde{o}',\widetilde{\epsilon}'\rrbracket$.

It remains to consider the case where $F$ is induced by a Reidemeister move.
For a negative Reidemeister I move, the maps $\mathcal{R}(\Phi^Y_{F'})$ and $\Phi^X_F$ are the same as in each map the orientation reversals on split saddles and deaths will cancel.
For a positive Reidemeister I move, the maps $\mathcal{R}(\Phi^Y_{F'})$ and $\Phi^X_F$ differ by $\pi$ as orientation reversals on splits and deaths appear alone.
For a Reidemeister II move, the maps $\mathcal{R}(\Phi^Y_{F'})$ and $\Phi^X_F$ feature the same orientations on the saddle in the non-identity maps in the map with a death, the orientation change on the death is canceled by the change in the ``a'' term, and overall the maps are the same.

Finally, the components of the Reidemeister III chain maps
are given by identity maps, Reidemeister II chain maps, and compositions
af homotopies, differentials, and Reidemeister II chain maps.
we have already seen that components coming from differentials
and Reidemeister II chain maps are the same in
$\mathcal{R}(\Phi^Y_{F'})$ and $\Phi^X_F$.

The see that the same is true for components coming from homotopies,
we note that the relevant homotopies are given by the death in $f$
accompanied by the sign $-a\pi\epsilon_{0,0\star\alpha}$, and
by the birth in $g$, accompanied by the sign $-\epsilon_{0,1\star\alpha}$,
where we are referring to the notations from \eqref{eq:3:RIImap}.
The same analysis that we used for the maps $f$ and $g$ thus
Shows that the components coming from these homotopies
are also the same in $\mathcal{R}(\Phi^Y_{F'})$ and $\Phi^X_F$.
\end{proof}

Working in the refined setting from Subsection~\ref{subs:supergraded}, we will
now show that the functor $\mathcal{R}$ is related to the identity functor via a $\pi$-supernatural
isomorphism. Theorem~\ref{thm:equivalence} will then follow from equations \eqref{eq:XY1} and \eqref{eq:XY3} and from the preceding lemmas.

For an object $C\subset\mathbb{R}^2$ of the embedded category, let $S_C\colon C\rightarrow\mathcal{R}(C)$ denote
a cobordism which connects each component of $C$ to the corresponding reflected component of $\mathcal{R}(C)$
using a genus zero cobordism with no critical points. 

\begin{lemma}
Any choice of map $S_C$ is equal to any other choice.
\end{lemma}

\begin{proof}
The diagram $C$ in $\mathbb{R}^3$ consists of nested circles.  We can enumerate the circles in $C$, and divide the circles into disjoint sets corresponding to which other circle they are directly nested inside.
The diagram $\mathcal{R}(C)$ can inherit the enumeration from $C$ and thus for each set of circles $X_i$ with $k_i$ circles, the cobordism $S_C$ associates an element of $\mathcal{B}_{k_i}$ the braid group with $k_i$ generators.
Overall to $C$, $S_C$ associates an element of the group $\displaystyle\prod_{i\in I}\mathcal{B}_{k_i}$.  If all the circle in a set $X_i$ have no circles nested inside them, the four-tube relation gives us the following:
\begin{equation}
    \includeFig{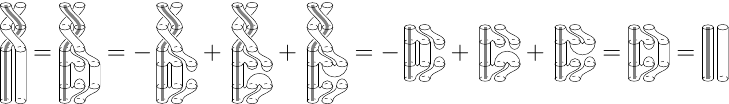}
\end{equation}
In the braid group this corresponds with the extra relation that generators square to identity.  
Making this additional identification we are left with $\mathcal{S}_{k_i}$ the symmetric group with $k_i$ generators.
If we could do this for each $X_i$, $S_C$ would induce an element of $\displaystyle\prod_{i\in I}\mathcal{S}_{k_i}$ in turn indicating that all choices for $S_C$ are equivalent as they correspond to the identity permutation of the circles.

If we look at the relation from before there is a gray region in the core of the right cobordism.  We can fill fill this region in with a cobordism indicating that the additional relation also applies to pairs of circles, as long as one them contains no other circles.

In the event that two circles in the set $X_i$ themselves contain circles consider the following decomposition
\begin{equation}\label{eq:wallDecomp}
\begin{tabular}{c}
\includeFig{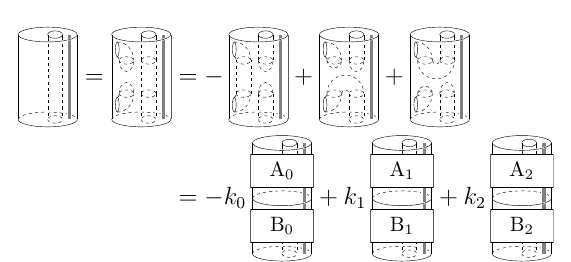} \\[6em]
\includeFig{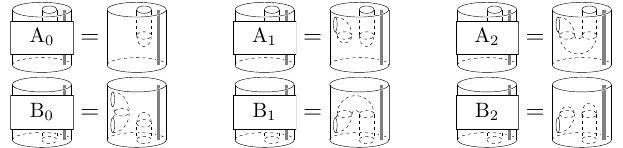}
\end{tabular}
\end{equation}

Using this decomposition the cobordisms inside of a cobordism can be ``absorbed'' into the wall of the outer cobordism leaving us with a sum of cobordism each of which can be seperated into a $B_i$
portion below and an $A_i$ portion above.  For each term of the sum we can leave the $B_i$ portion in place, and push the $A_i$ up past the braiding of the outside circle incuring an invertible scalar $k_i$.  
We can now apply the earlier relation, then slide the $B_i$ and $A_i$ sections back together again incurring the same invertible scalar $k_i$ as we never need to change the orientations on a saddle or change the chronology when applying the relation.
\end{proof}

\begin{figure}[H]
	\centering
	\includeFig{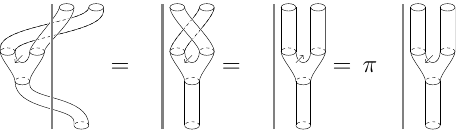}
	\hspace{4em}
	\includeFig{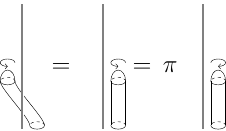}
	\caption{The conjugation of reflected split and death cobordisms with the map $S_C$ induce multiplication by $\pi$ relative to the unreflected split and death cobordisms. The vertical grey line denotes the plane of reflection}
	\label{eq:Relections}
\end{figure}

By making appropriate choices for the cobordisms $S_C$, one can check that the $S_C$ intertwines births and merges with their reflections,
and intertwine deaths and splits with their reflections up to a factor of $\pi$ (see Figure \eqref{eq:Relections}).
As such, the $S_C$ define a $\pi$-supernatural isomorphism between the identity functor and $\mathcal{R}$.
We can extend this
$\pi$-supernatural isomorphism to a $\pi$-supernatural isomorphism $\eta$
defined on the refinement of $\cobIIIgen$ by setting $\eta_{C\{m\}}:=\pi^m S_C\colon C\{m\}\rightarrow\mathcal{R}(C)\{m\}$ for $C$
as above and $m\in\mathbb{Z}_2$.
By the discussion at the end of Subsection~\ref{subs:supergraded},
we then get a chain isomorphism
\begin{equation}\label{eq:XY4}
\Psi:=\eta_{\llbracket D',o',\epsilon'\rrbracket} \colon\llbracket D',o',\epsilon'\rrbracket\longrightarrow\mathcal{R}(\llbracket D',o',\epsilon'\rrbracket)
\end{equation} 
with components
$\Psi_\alpha=\pi^{S(D',\alpha)}S_{D'_\alpha}$,
where here $\llbracket D',o',\epsilon'\rrbracket$ denotes the refined generalized
Khovanov bracket.

\begin{lemma}\label{lem:XY3} The chain isomorphism in \eqref{eq:XY4} is natural, in the sense that it intertwines the maps $\Phi^Y_{F'}$ and $\mathcal{R}(\Phi^Y_{F'})$ up to a possible factor of $\pi$.
\end{lemma}

\begin{proof}
By Remark~\ref{rem:cobordismsuperdegree}, the map $\Phi^Y_{F'}$ that the link cobordism $F'$ induces on the refined generalized Khovanov bracket
is homogeneous with respect to the supergrading. Since $\Psi$ is a component of a $\pi$-supernatural
transformation (by its definition in~\eqref{eq:XY4}), Lemma~\ref{lem:XY3} now follows from equation~\eqref{eqn:supernaturaldefinition} in the definition of a $\pi$-supernatural
transformation.
\end{proof}

Combing equations \eqref{eq:XY1}, \eqref{eq:XY3}, and \eqref{eq:XY4}, we now obtain a homotopy equivalence
\begin{equation}\label{eq:XY5}
\llbracket D,o,\epsilon\rrbracket\,\simeq\,\llbracket D',o',\epsilon'\rrbracket\,\stackrel{\Psi}{\cong}\,
\mathcal{R}(\llbracket D',o',\epsilon'\rrbracket)\,=\,\llbracket D,o,\epsilon'\rrbracket
\end{equation}
between the type Y complex $\llbracket D,o,\epsilon\rrbracket$ and the type X complex
$\llbracket D,o,\epsilon'\rrbracket$. Moreover,
Lemmas~\ref{lem:XY1}, \ref{lem:XY2}, and \ref{lem:XY3}
imply that this
homotopy equivalence is natural up to invertible scalars,
which proves Theorem~\ref{thm:equivalence}.
As an immediate corollary, we obtain a correct proof of the following lemma,
which was first stated as Lemma 2.4 in \cite{ORS2007}.

\begin{lemma}\label{lem:XY4} Type X and type Y sign assignments yield isomorphic odd Khovanov complexes.
\end{lemma}

\begin{proof} By \eqref{eq:XY5}, the odd Khovanov complexes of $(D,o,\epsilon)$ and $(D,o,\epsilon')$ are homotopy equivalent.
The two complexes are also bounded complexes of finitely generated free abelian groups, which share the same chain ranks.
An argument related to the Smith normal form of an integer matrix now shows that two such complexes are homotopy equivalent if and only if they are isomorphic.
The lemma thus follows.
\end{proof}

Note that the isomorphism from the proof of Lemma~\ref{lem:XY4} might not be natural.

\subsection{Definition of the Generalized Khovanov Functor}
In this subsection, we will define the generalized Khovanov chain maps assigned to each elementary link cobordism. Note that the definition
of such maps was already sketched in \cite{Pu2015}.

For a link cobordism $F$ from the link diagram $L_0$ to the link diagram $L_1$, we will denote the associated chain map on the generalized Khovanov bracket by $\Phi=\Phi_F\colon\llbracket L_0\rrbracket\rightarrow\llbracket L_1\rrbracket$. The chain maps that we will assign to birth, death, and saddle
cobordisms will preserve the cubical structure of the Khovanov bracket, meaning that their nonzero components
will be between corresponding vertices in the cubes of $L_0$ and  $L_1$.
For such maps $\Phi$, the particular component going from $\llbracket L_0\rrbracket_\alpha$ to $\llbracket L_1\rrbracket_\alpha$ will be denoted $\Phi_\alpha$, while the planar cobordism underlying $\Phi_\alpha$ will be denoted $\varphi_\alpha$.  

\subsubsection{Birth Cobordisms}

Let $F\colon L_0\rightarrow L_1$ be a four-dimensional birth cobordism, so that $L_1$ is the disjoint union of $L_0$ with a circle.  For each $\alpha$ we define $\varphi_\alpha$ as the identity cobordism on the components of $\llbracket L_0\rrbracket_\alpha$ together with planar birth cobordism ending in the new circle in $\llbracket L_1\rrbracket_\alpha$.  As births are even, the $\varphi_\alpha$ already commute with the differentials,
and we thus obtain a chain map $\Phi\colon\llbracket L_0\rrbracket\rightarrow\llbracket L_1\rrbracket$ by setting $\Phi_\alpha:=\varphi_\alpha$.

\subsubsection{Death Cobordisms}

 Let $F\colon L_0\rightarrow L_1$ be a four-dimensional death cobordism, so that $L_0$ is the disjoint union of $L_1$ with a circle.  For each $\alpha$ we define $\varphi_\alpha$ as the identity cobordism on components of $\llbracket L_0\rrbracket_\alpha$ that are also present in $\llbracket L_1\rrbracket_\alpha$ and as a planar death cobordism on the lone circle in $\llbracket L_0\rrbracket_\alpha$.  Unlike births, deaths can either commute or $\pi$-commute with differentials. To make them commute, we need to multiply the $\varphi_\alpha$ by $\pi^{S(L_0,\alpha)}$, where
 $S(L_0,\alpha)$ denotes the quantity from equation~\eqref{eq:3:Smap}.
 Thus, the chain map $\Phi\colon\llbracket L_0\rrbracket\rightarrow\llbracket L_1\rrbracket$ assigned to an elementary death cobordism has
 components $\Phi_\alpha:=\pi^{S(L_0,\alpha)}\varphi_\alpha$.

\begin{remark}\label{rem:deathpisupernatural} The chain map $\Phi_F$ assigned to a death cobordism has the following
categorical interpretation. Let $L_0$ and $L_1$ be as above, and let $D\subset\mathbb{R}^2$ be
a disk that does not intersect $L_1$ and that contains the unique component $c$ in $L_0\setminus L_1$.
Let $\mathcal{C}$ be the category defined in the same way as $\cobIIIgen$,
but with the requirement that
the objects of $\mathcal{C}$ are contained in $\mathbb{R}^2\setminus D$, and the morphisms
are contained in $(\mathbb{R}^2\setminus D)\times I$. The inclusion $\mathbb{R}^2\setminus D\rightarrow\mathbb{R}^2$
induces a functor $\mathcal{C}\rightarrow\cobIIIgen$, which is faithful by \cite[Cor.~88]{NW2024} but not full. 
In particular, this functor allows us to identify $\mathcal{C}$ with a subcategory of $\cobIIIgen$.
Now consider the functor $\mathcal{F}\colon\mathcal{C}\rightarrow\cobIIIgen$ given by sending an object of $\mathcal{C}$ to its union with $c$, and a morphism
to its union with $c\times I$. For each object $x\in\mathcal{C}$, there is an obvious
morphism $\eta_x\colon\mathcal{F}(x)\rightarrow x$ given by a death cobordism in $\mathbb{R}^2\times I$
that annihilates $c$. Working in the supergraded refinement from Subsection~\ref{subs:supergraded},
we can further define a morphism
$\eta_{x\{s\}}\colon \mathcal{F}(x\{s\})\rightarrow x\{s\}$ by $\eta_{x\{s\}}:=\pi^s\eta_x$,
where $x\{s\}$ denotes the object $x$ with supergrading shifted by $s\in\mathbb{Z}_2$.
Using that deaths have odd superdegree, it is then easy to see
that the morphisms $\eta_{x\{s\}}$ provide a $\pi$-supernatural transformation
between  (the supergraded refinements of) the
functor $\mathcal{F}$ and the inclusion-induced functor $\mathcal{C}\rightarrow\cobIIIgen$.
As explained in Subsection~\ref{subs:supergraded}, any $\pi$-supernatural transformation
extends to chain complexes, and the death map $\Phi_F$ can now be identified with
the component
$\Phi_F=\eta_{\llbracket L_1\rrbracket}\colon\llbracket L_0\rrbracket\rightarrow\llbracket L_1\rrbracket$
of the extended $\pi$-supernatural transformation.
\end{remark}

\subsubsection{Saddle Cobordisms}

Let $F\colon L_0\rightarrow L_1$ be a four-dimensional saddle cobordism.  For each $\alpha$, the planar diagrams $\llbracket L_0\rrbracket_\alpha$ and $\llbracket L_1\rrbracket_\alpha$ are
related by a saddle in the neighborhood where $L_0$ and $L_1$ differ. 
Consider the larger link diagram $L$ which has one extra crossing in this neighborhood, so that
replacing this crossing by its 0- or 1-resolution yields the diagrams $L_0$ and $L_1$, respectively.
By construction, the resolution cube of $L$ contains the resolution cubes of $L_0$ and $L_1$ as codimension-1 subcubes,
and by Lemma~\ref{lem:3:CWsubcomplex}, we can find a sign assignment for the resolution cube of $L$ which restricts
to the given sign assignments for $L_0$ and $L_1$. Consider the components
$\partial_{\alpha\star}\colon\llbracket L_0\rrbracket_\alpha\rightarrow\llbracket L_1\rrbracket_\alpha$ of the differential in $\llbracket L\rrbracket$ 
that pass between the two subcubes.
These components are of the form
\begin{equation}\label{eq:3:saddlemap}
\partial_{\alpha\star}=\epsilon_{\alpha\star}d_{\alpha\star}
\end{equation}
for planar saddle cobordisms
$d_{\alpha\star}\colon\llbracket L_0\rrbracket_\alpha\rightarrow\llbracket L_1\rrbracket_\alpha$, where the coefficients $\epsilon_{\alpha,\star}\in\Bbbk^\times$ ensure that the $\partial_{\alpha\star}$ anticommute 
with the differentials in $\llbracket L_0\rrbracket$ and $\llbracket L_1\rrbracket$. We can thus define
a chain map  $\Phi\colon\llbracket L_0\rrbracket\rightarrow\llbracket L_1\rrbracket$ by setting $\Phi_\alpha:=(-1)^{\deg(\alpha)}\partial_{\alpha\star}$.

We claim that the above definition specifies the saddle map $\Phi$ uniquely up to an overall invertible scalar. 
Indeed, let $Q\cong [0,1]^{n+1}$ denote the hypercube
corresponding to $L$, and $Q_0$ and $Q_1$ denote the subcubes corresponding
to $L_0$ and $L_1$. For fixed sign assignments on $Q_0$ and $Q_1$,
any two choices
for the $\epsilon_{\alpha\star}$ in \eqref{eq:3:saddlemap} then differ by a relative 1-cocycle in
$Z^1(Q,Q_0\cup Q_1;\Bbbk^\times)$. The claim now follows because
\[
Z^1(Q,Q_0\cup Q_1;\Bbbk^\times)
\cong H^1(Q/(Q_0\cup Q_1);\Bbbk^\times)\cong H^1(S^1;\Bbbk^\times)\cong\Bbbk^\times.
\]
We further observe that changing the orientation of the extra crossing in $L$
leaves the chain complex of $L$ and hence the chain map $\Phi$
unchanged if one adjusts the sign assignment for $L$ appropriately
(this follows essentially from the proof Lemma 2.3 in {\cite{ORS2007}}). Similarly,
changing the orientation of a crossing in $L_0$ or $L_1$ leaves
$\Phi$ unchanged if one changes the sign assignments for $L$ appropriately.
Finally, $\Phi$ is natural in the choice of the sign assignments for $L_0$ and $L_1$,
and thus it induces a morphism $\llbracket L_0\rrbracket_M\rightarrow \llbracket L_1\rrbracket_M$
in the category $M''$ from Subsection~\ref{subs:colimit}
(the latter is also trivially true for the chain maps that we assigned
to birth and death cobordisms).

We briefly consider the special case in which $L_0$ is a union of two disjoint link diagrams
and $F$ merges these diagrams into a connected sum diagram $L_1$.
In this case, we can use the same sign assignments for $L_0$ and $L_1$
because these diagrams have the same commutativity cocylces $\sigma_{i,j}$.
To construct the saddle map $\Phi$, we then do note need to pass to the larger
link diagram $L$. Instead, the components of $\Phi$ are given by
the same underlying saddle cobordisms $\varphi_\alpha\colon\llbracket L_0\rrbracket_\alpha\rightarrow\llbracket L_1\rrbracket_\alpha$ 
as in the general case, but now all of these saddles are merges, and the
$\epsilon_{\alpha\star}$ are trivial
(i.e., $\Phi_\alpha:=\varphi_\alpha=d_{\alpha\star}$). The case where 
$L_0$ is a connected sum and $F$ splits this diagram into the corresponding
disjoint union of two diagrams is similar.
We can again use the same sign assignments for $L_0$ and $L_1$, but 
now all of the cobordisms $\varphi_\alpha\colon\llbracket L_0\rrbracket_\alpha\rightarrow\llbracket L_1\rrbracket_\alpha$
are split saddles, and the coefficients $\epsilon_{\alpha\star}$ are given explicitly by
the same terms that we used for death cobordisms (i.e., $\Phi_\alpha:=\pi^{S(L_0,\alpha)}\varphi_\alpha=\pi^{S(L_0,\alpha)}d_{\alpha\star}$).

\subsubsection{Reidemeister Type Cobordisms}

To Reidemeister I and II type cobordisms, we assign the chain maps
shown in  \eqref{eq:3:RImap}, \eqref{eq:3:RImapalt},
and \eqref{eq:3:RIImap}. To Reidemeister III type cobordisms,
we assign the chain map induced by the homotopy
equivalences in \eqref{eq:3:conesL} and \eqref{eq:3:conesR}. Note that
this chain map has the following general form:

\begin{equation}
    \includeFig{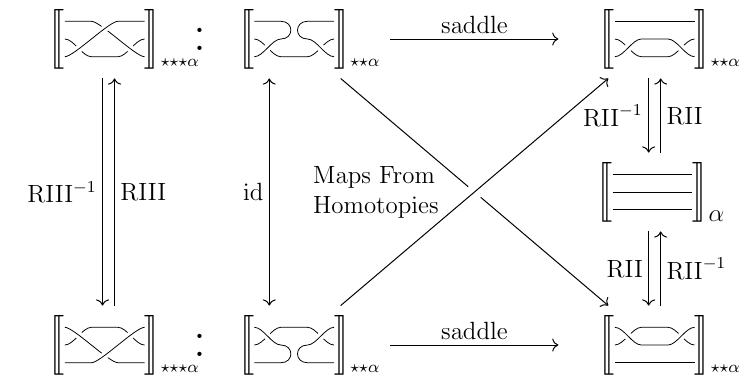}
\end{equation}

\begin{remark}\label{rem:cobordismsuperdegree} One can check that the chain map assigned to an oriented link cobordism $F\subset\mathbb{R}^3\times I$ is homogeneous
with respect to the supergrading if one works with the supergraded refinement of the generalized
Khovanov bracket defined in Subsection~\ref{subs:supergraded}.
Moreover, the superdegree of this chain map coincides with the superdegree of $F$ itself.
Here it is understood that a birth cobordism in $\mathbb{R}^3\times I$ has superdegree $0$,
and a death cobordism in $\mathbb{R}^3\times I$ has superdegree $1$.
A saddle cobordism in $\mathbb{R}^3\times I$ has superdegree $0$ if it
merges two link components into a single link component, and superdegree $1$ if it splits a link component into two.
\end{remark}

\section{Proof of Functoriality}\label{s:functoriality}

We now have all the elements in place to prove Theorem~\ref{thm:main} from the introduction.
More precisely, we will show:

\begin{theorem}\label{thr:GKHfunctor}
    The generalized Khovanov bracket extends to a functor from the category $\cobIV$ to the category
    $\textnormal{Kom}^b_{\pm h,\pm\pi h}(\mat{\cobIIIgen})$.
\end{theorem}

As an immediate consequence of Theorem~\ref{thr:GKHfunctor}, we obtain
Corollary~\ref{cor:main}
from the introduction, which states that odd Khovanov homology is functorial in $\mathbb{R}^3\times I$ up to sign.

In Section 3 we specified what the generalized Khovanov bracket assigns to movies of link cobordisms.  The entirety of this section is devoted to the proof of Theorem \ref{thr:GKHfunctor}, where we show how our construction respects all the possible 
movies moves from Theorem~\ref{thm:CS}.
We will begin with a discussion of the type I movie moves, followed by a general argument to show that each side of the type II movie moves produces homotopic chain maps.  Next, we will examine type III movie moves with individual arguments in the forward direction, in the reverse direction, and for alternative variants.  Finally, we will show the functoriality of odd Khovanov homology with respect to chronological movie moves. We do not need to spend additional time ensuring our functor respects
planar ambient isotopy moves.
The chain maps induced by the two sides of such moves have the same matrix entries in $\cobIIIgen$
since the underlying chronological cobordisms are ambient isotopic in $\mathbb{R}^2\times I$ via
an ambient isotopy that preserves the chronological structure.

\subsection{Type I Movie Moves}

\begin{figure}[H]
    \centering
    \begin{tabular}{ c }
        Type \\
        I \\ 
        Movie \\
        Moves
    \end{tabular}
    \(= \left[
    \addstackgap[0.25em]{\begin{tabular}{c}\includeMov{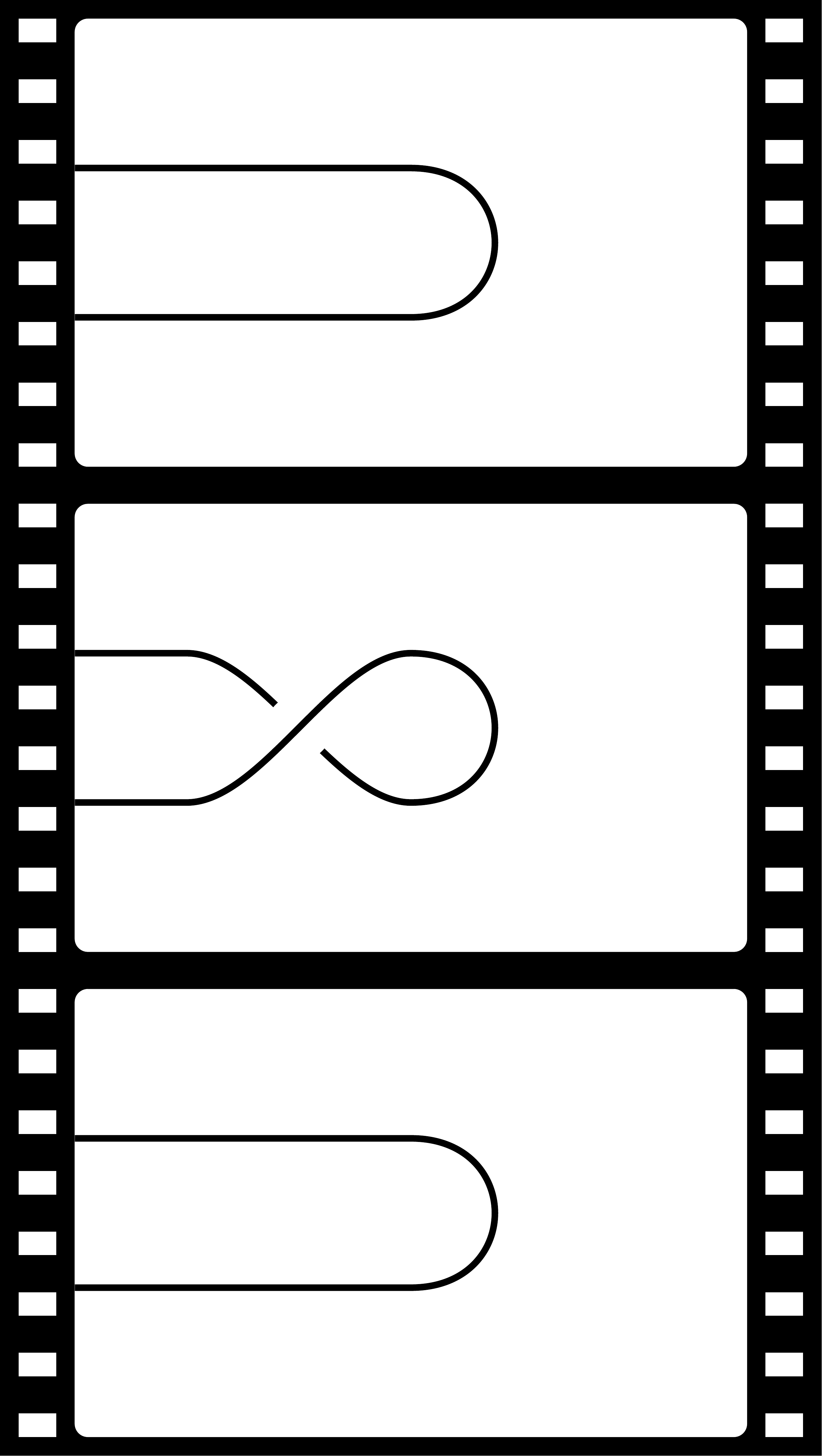}\\MM1\end{tabular}}
    \begin{tabular}{c}\includeMov{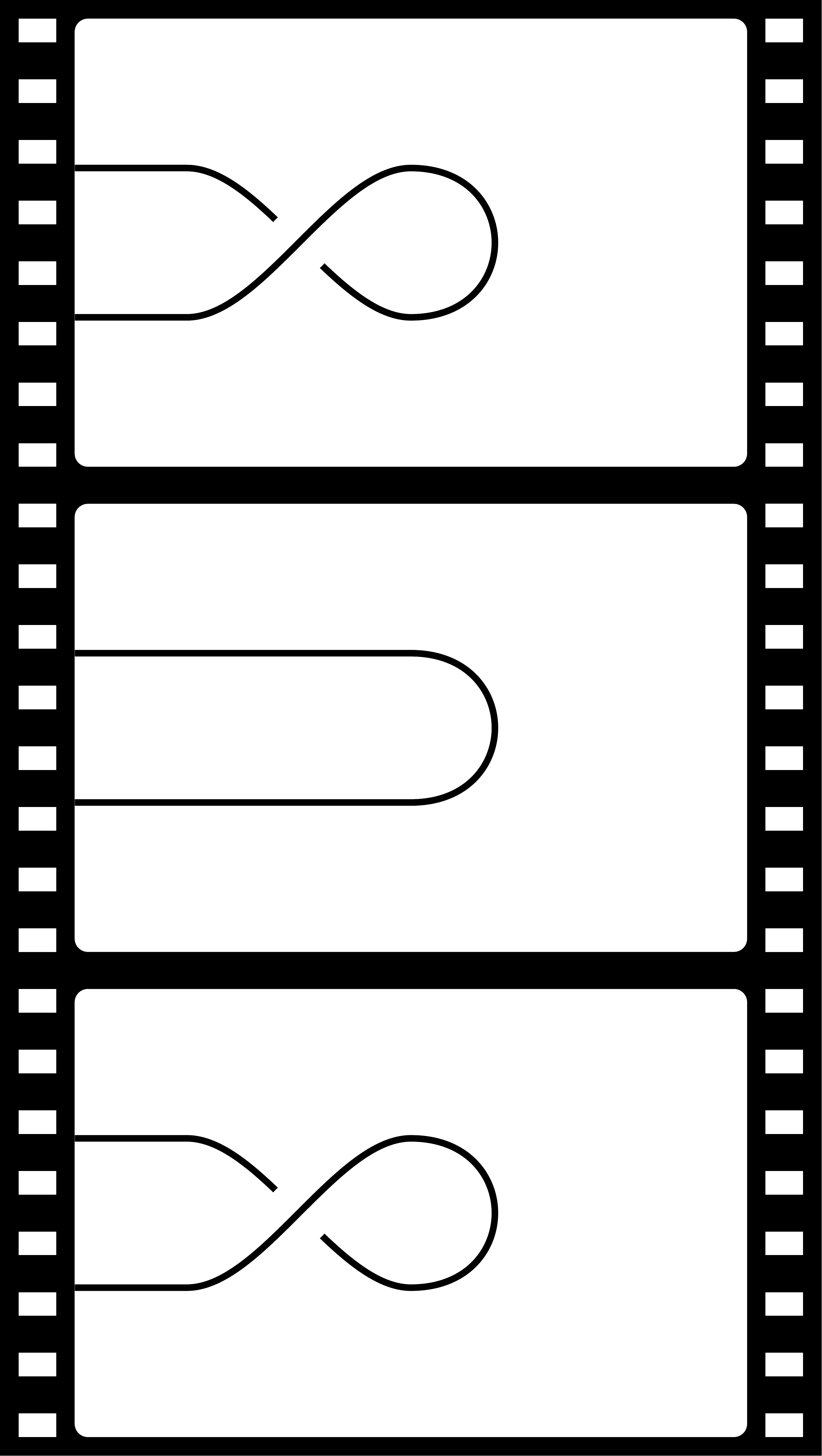}\\MM2\end{tabular} 
    \begin{tabular}{c}\includeMov{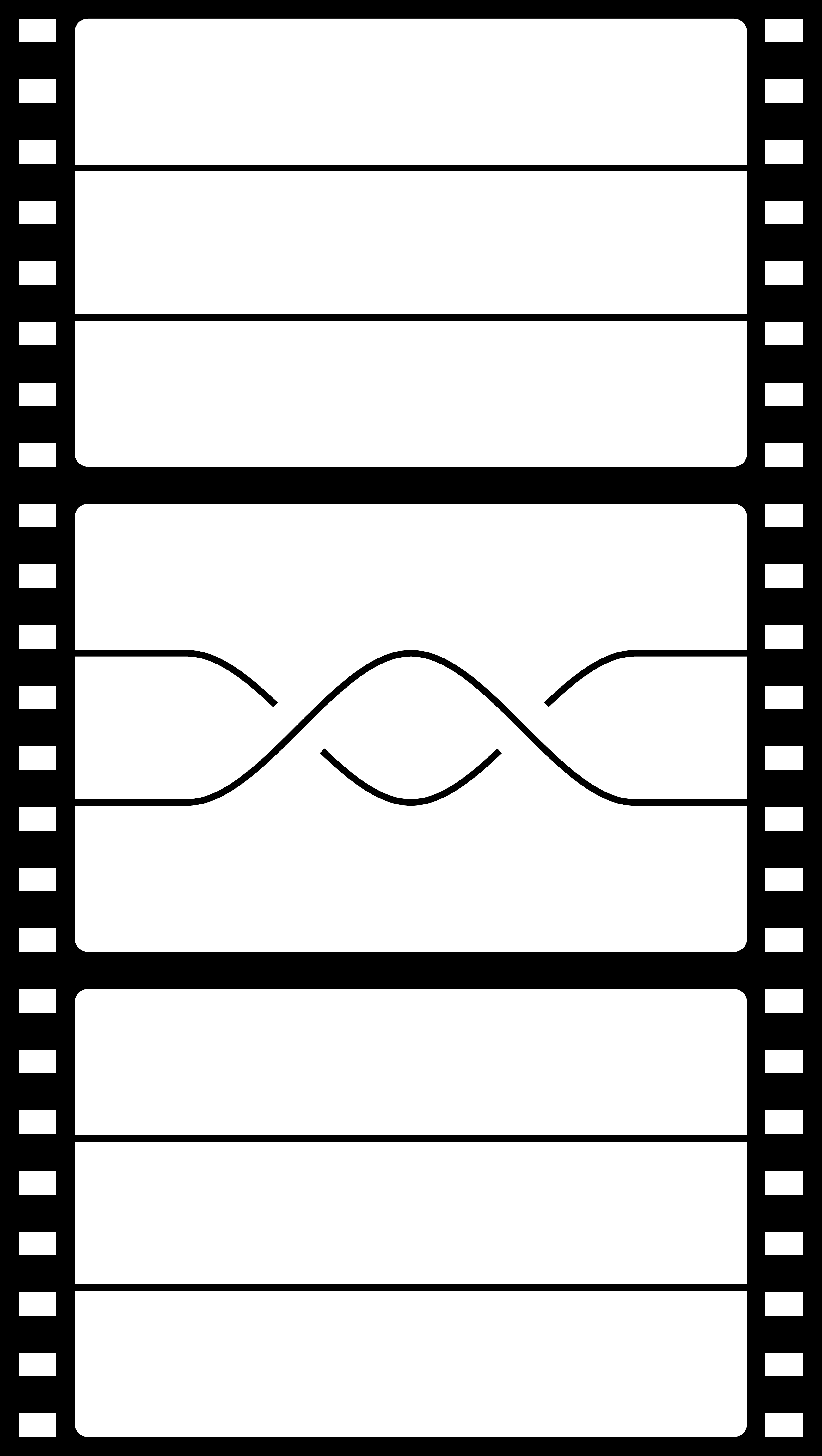}\\MM3\end{tabular} 
    \begin{tabular}{c}\includeMov{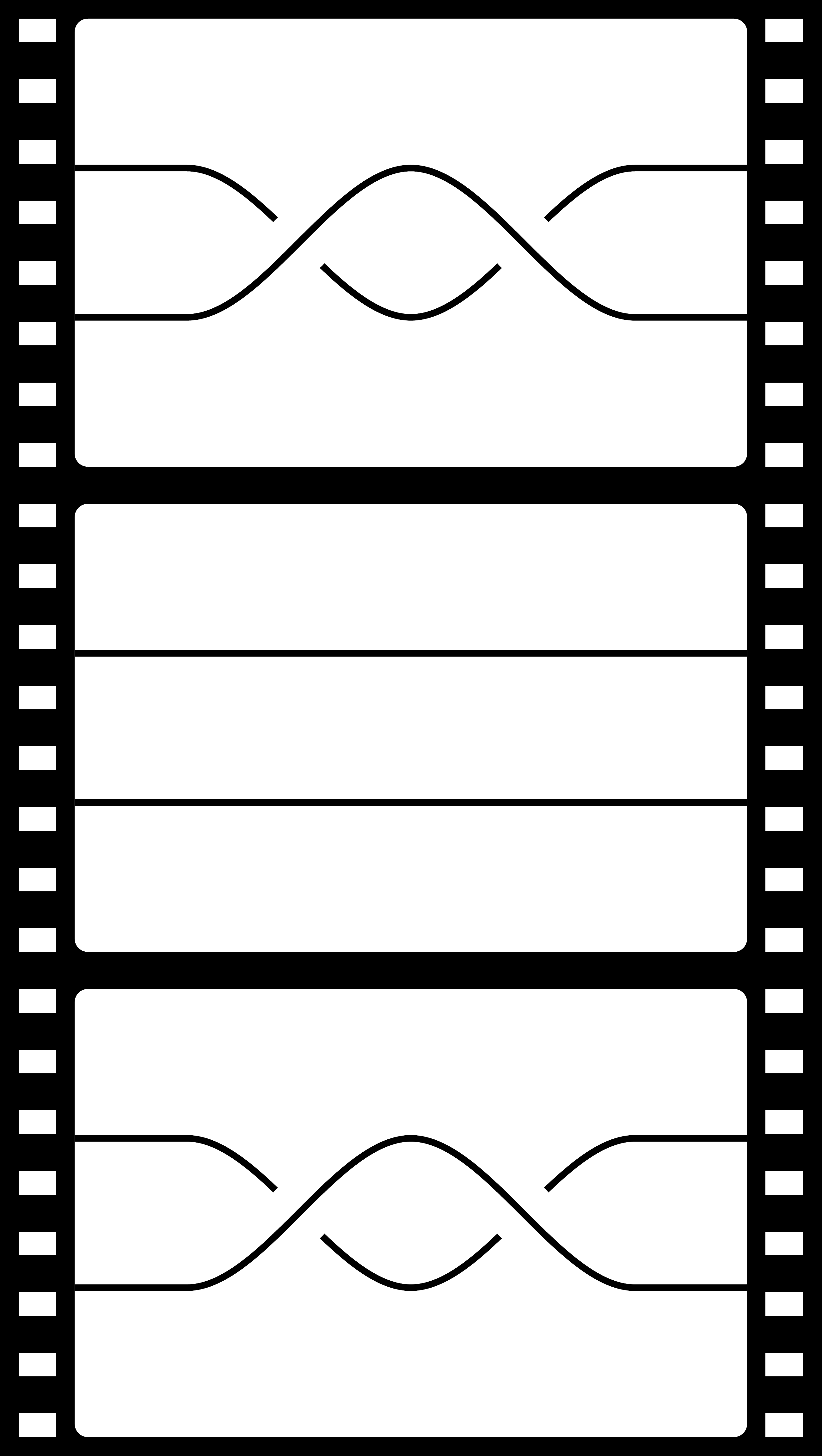}\\MM4\end{tabular} 
    \begin{tabular}{c}\includeMov{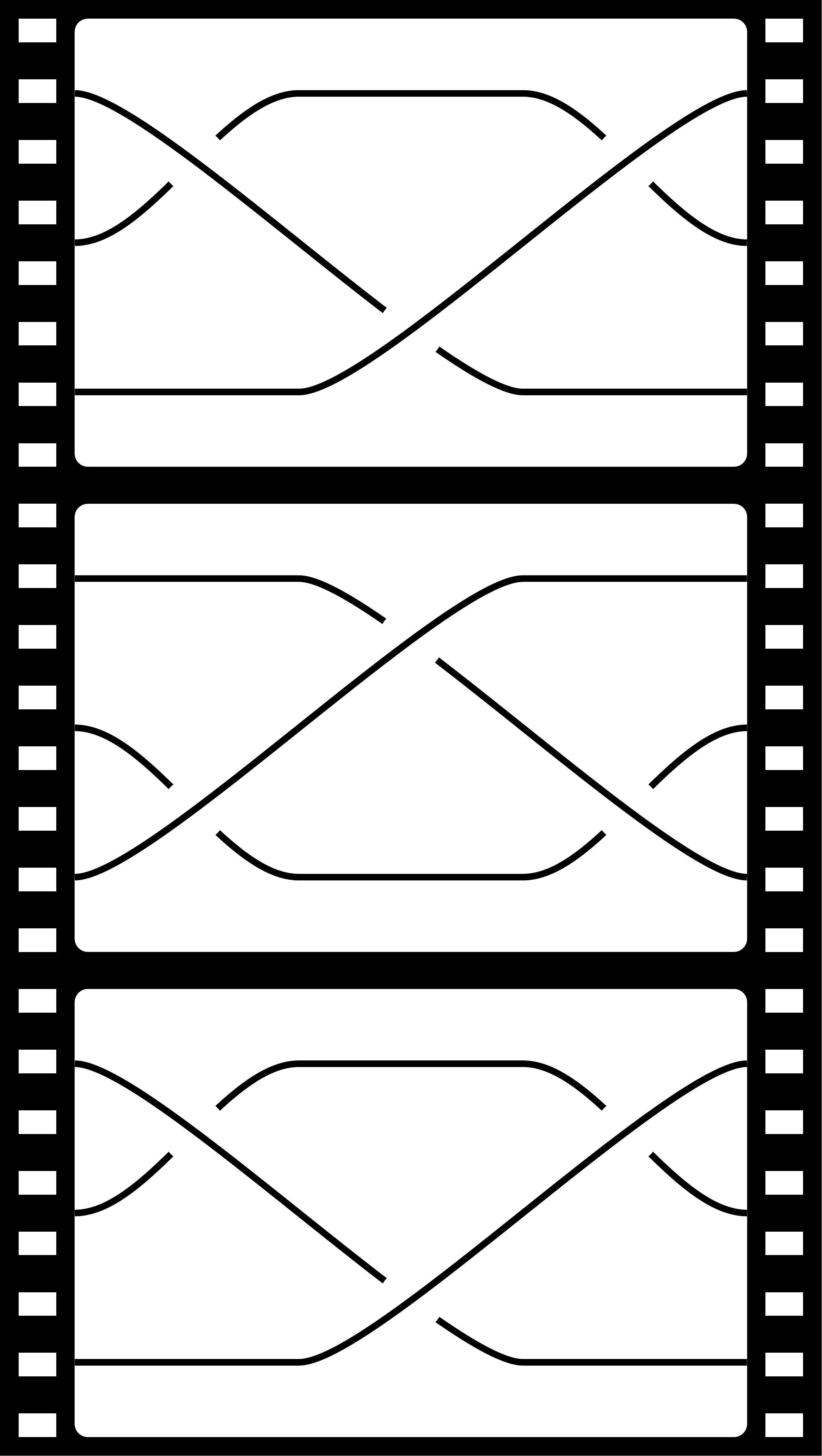}\\MM5\end{tabular} 
    \right]\)
    \caption{Movie moves 1 through 5}
    \label{fig:4:TIMM}
\end{figure}

The left-hand sides of the first five movie moves are given by doing and undoing a Reidemeister move, while the right-hand sides (which are not shown in Figure \ref{fig:4:TIMM}) are given by trivial movies of identity cobordisms. Since the chain maps assigned to Reidemeister moves are homotopy
equivalences, it is clear that the left-hand sides induce chain maps homotopic to the identity, and thus our functor respects these movie moves.

\subsection{Type II Movie Moves}\label{subs:typeII}
 
\begin{figure}[H]
    \centering
    \begin{tabular}{ c }
        Type \\
        II \\ 
        Movie \\
        Moves
    \end{tabular}
    \(= \left[
    \begin{tabular}{c}\includeMov{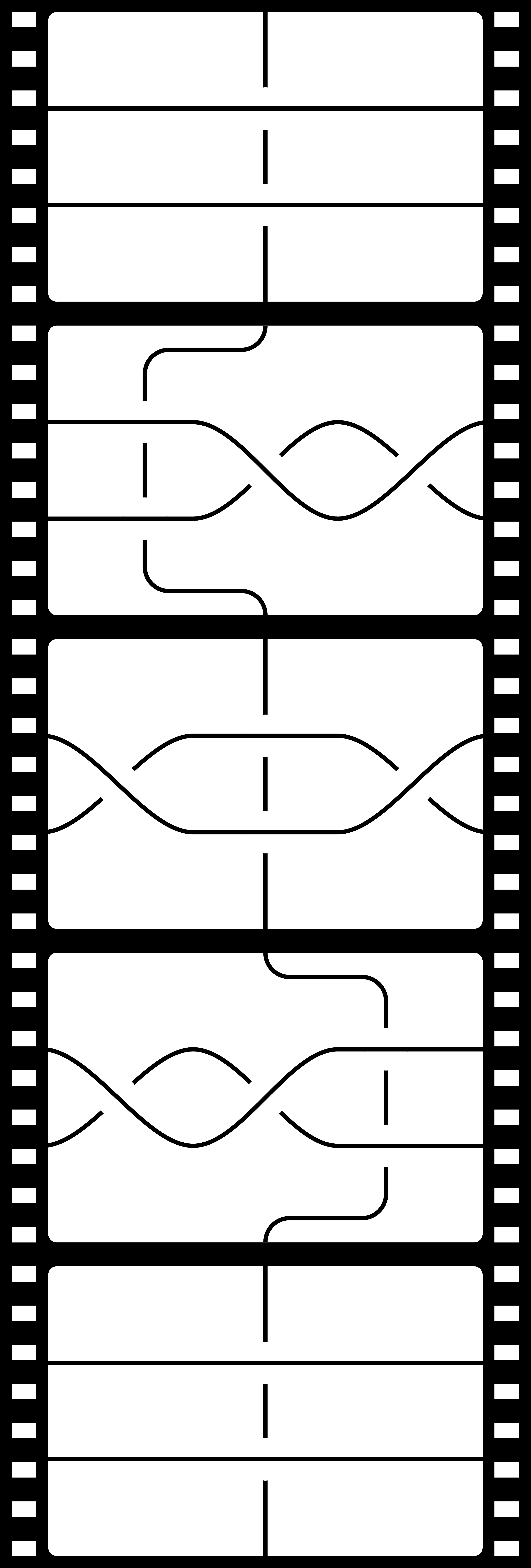}\\MM6\end{tabular}
    \begin{tabular}{c}\includeMov{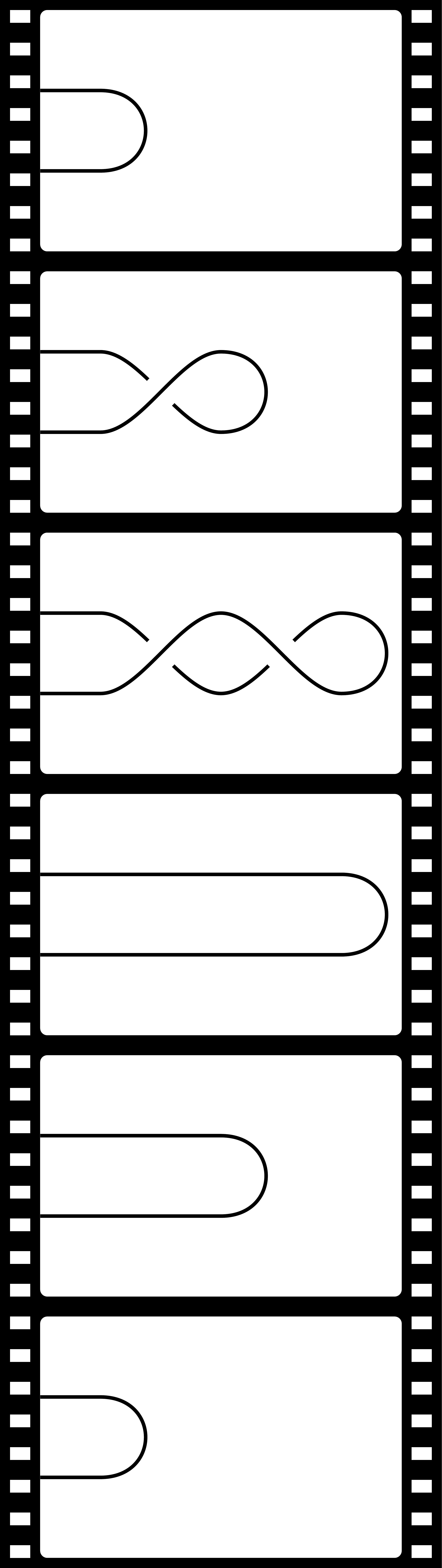}\\MM7\end{tabular} 
    \begin{tabular}{c}\includeMov{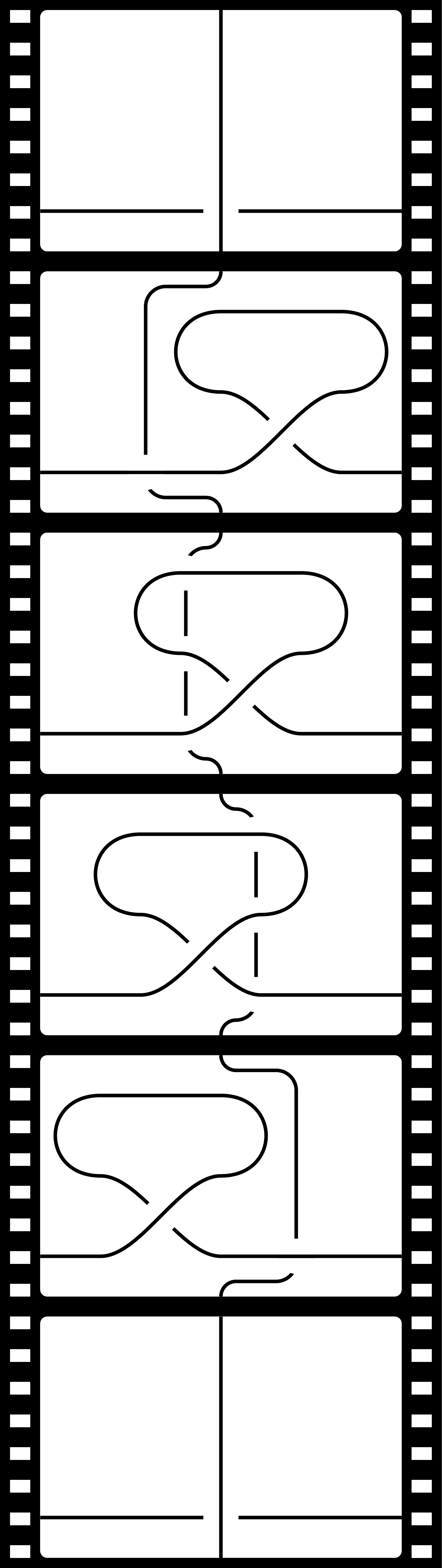}\\MM8\end{tabular} 
    \begin{tabular}{c}\includeMov{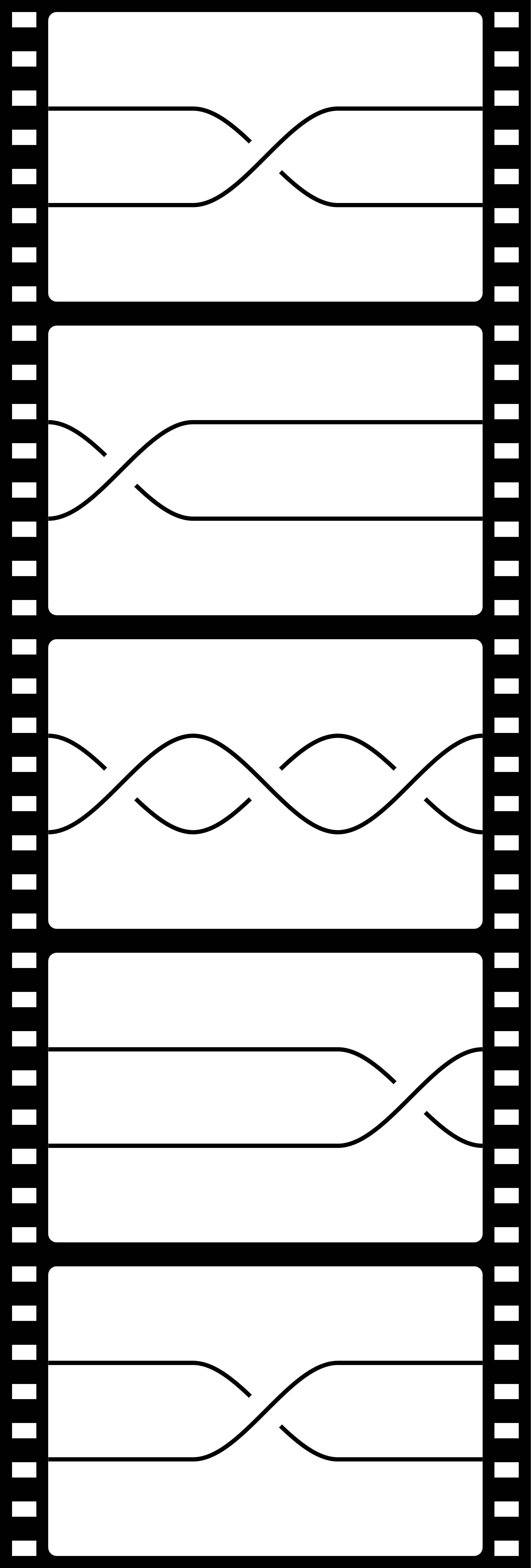}\\MM9\end{tabular} 
    \addstackgap[0.25em]{\begin{tabular}{c}\includeMov{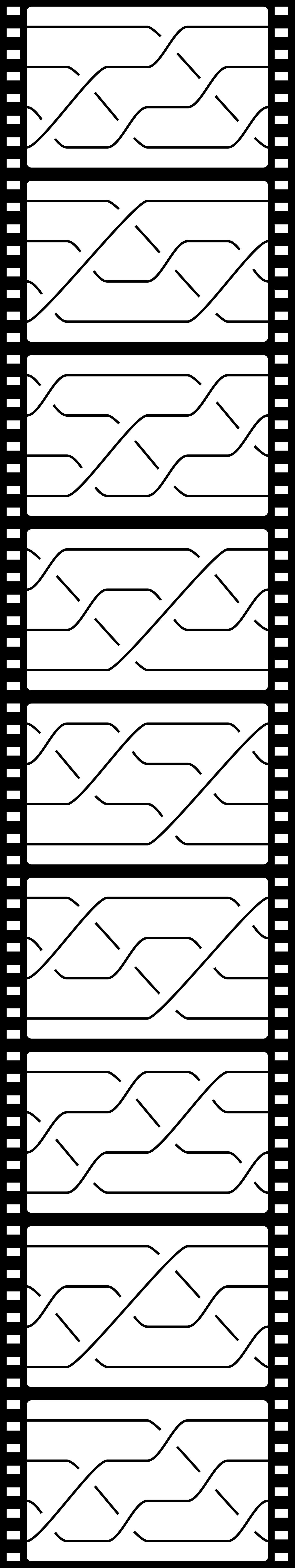}\\MM10\end{tabular}} 
    \right]\)
    \caption{Movie moves 6 through 10}
    \label{fig:4:TIIMM}
\end{figure}

Any link diagram on which a movie move is carried out consists of two tangles glued together: namely the inside part $t$ where the movie move is carried out, and an outside part $T$ which is carried through the movie by identity. 
Type II movie moves permit a slightly stricter decomposition wherein the
inside tangle is isotopic to a braid diagram.

The left-hand sides of these movie moves are shown
in Figure~\ref{fig:4:TIIMM} and are given by sequences 
of Reidemeister moves that start and end on identical frames,
while the right-hand sides are given by trivial movies of identity cobordisms.
To show invariance under these moves, we will establish a general lemma
pertaining to the case where the inside tangle is a braid diagram.
More generally, we will allow the inside tangle to be of the form
$t=C\cup\beta\subset B\cup A$
where
\begin{itemize}
\item $B$ is a disk region and $A$ is an annular region surrounding it,
\item $C\subset B$ is a crossingless tangle with no closed components,
\item $\beta\subset A$ is an affine braid, consisting of strands that are monotonous
in the radial coordinate on $A$ and that connect the points of $\partial C$ to
points on the outer boundary of $A$.
\end{itemize}
We then have:

\begin{lemma}\label{lem:t2}
    Let $F$ be a link cobordism with a link diagram $D$ as both its initial and terminal frame. Furthermore, suppose $F$ is generated by performing a sequence of Reidemeister moves on a tangle $t\subset D$ of the form $t=C\cup\beta$ where $C$ and $\beta$ are as above.
     Then the chain map that $F$ induces on $\llbracket D\rrbracket$ is homotopic to $\pm\id$ or $\pm\pi\id$.
\end{lemma}

\begin{proof}
Let $t=C\cup\beta=:C\beta$ be as in the lemma, and $T$ the arbitrary tangle on the outside.
The link cobordism $F$ is the identity cobordism on the tangle $T$, and on the tangle $C\beta$ it is a cobordism $f$ generated by
the sequence of Reidemeister moves. Let $\Phi=\Phi_F$ be the map that $F$ induces on the generalized Khovanov bracket of $D$:
\begin{equation}
\includeFig{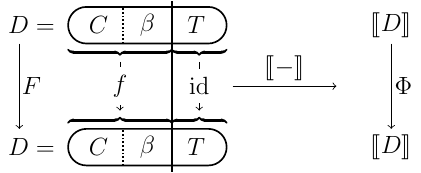}
\end{equation}
To prove the lemma, we must show that $\Phi\simeq c\operatorname{id}$ for a $c\in\Bbbk^\times$.  Consider the alternative diagram $D'$ of the same link obtained by gluing the affine braid $\beta^{-1}\beta$ onto $T$, as shown in \eqref{dig:4:proofII1}. The sequence of Reidemeister moves considered previously gives rise to a cobordism $F'$ from $D'$ to $D'$, and this cobordism induces a map $\Phi'$ on the generalized Khovanov bracket of $D'$:
\begin{equation}\label{dig:4:proofII1}
\includeFig{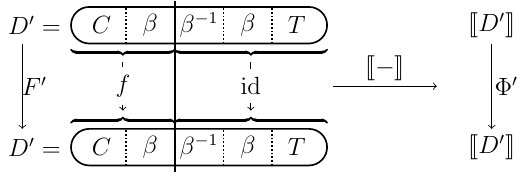}
\end{equation}
We can also consider the cobordism $G$ from $D$ to $D'$ which is the identity cobordism on $C\beta$ and a cobordism $g$ comprised of many Reidemeister II type cobordisms from $T$ to $\beta^{-1}\beta T$.  Let $\Psi$ be the homotopy equivalence that $G$ induces between the generalized Khovanov brackets of $D$ and $D'$:
\begin{equation}
\includeFig{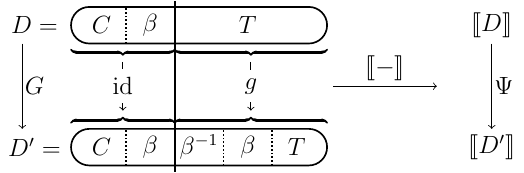}
\end{equation}
The chain maps induced by $F$, $F'$, and $G$ fit into the following diagram:
\begin{equation}\label{dig:phiphiprime}
\includeFig{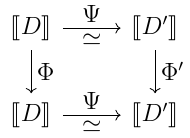}
\end{equation}
\begin{claim}\label{claim:t21}
The preceding diagram commutes up to homotopy and overall invertible scalars.
\end{claim}

We will wait until the end of the proof to prove this claim. The claim implies that to prove Lemma \ref{lem:t2}, it is sufficient to
show:

\begin{claim}\label{claim:t22}
$\Phi'\simeq c\operatorname{id}$ for $c\in\Bbbk^\times$.
\end{claim}

To prove the latter claim, we consider an alternative decomposition of $D'$ in which the affine braid $\beta\beta^{-1}$ is glued onto the perimeter of $C$ as shown in the bottom half of \eqref{dig:4:proofII2}. Let $G'$ be the link cobordism from $D$ to $D'$ which is given by the identity cobordism of $\beta T$ and by a cobordism $g'$ comprised of many Reidemeister II type cobordisms from $C$ to $C\beta\beta^{-1}$.

The link cobordism $G'$ induces an alternative homotopy equivalence $\Psi'$ between the generalized Khovanov brackets of $D$ and $D'$:

\begin{equation}\label{dig:4:proofII2}
\includeFig{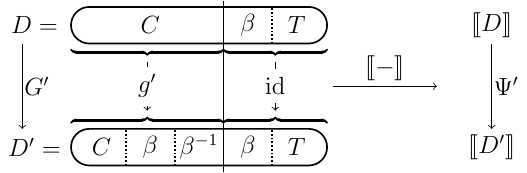}
\end{equation}
Now consider the following automorphism of $\llbracket D\rrbracket$ where $\left(\Psi'\right)^{-1}$ is the homotopy inverse of $\Psi'$:
\begin{equation}
    \Phi''\coloneqq\left(\Psi'\right)^{-1}\circ\Phi'\circ\Psi'
\end{equation}
By the definition of $\Phi''$, the following diagram commutes up to homotopy.
\begin{equation}
\includeFig{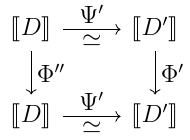}
\end{equation}

To prove the Claim~\ref{claim:t22}, it is thus sufficient to show the following.

\begin{claim}\label{claim:t23}
$\Phi''= c\operatorname{id}$ for $c\in\Bbbk^\times$.
\end{claim}

To prove the latter, we will use that the nontrivial part of the chain map $\Phi''$ is localized to the left-hand tangle $C$ in the decomposition of $D$.  That is, there is a  cobordism $F''$---that induces $\Phi''$ on the generalized Khovanov bracket---given by composing $G'$ played in reverse with $F'$ and $G'$.  

\begin{equation}
\includeFig{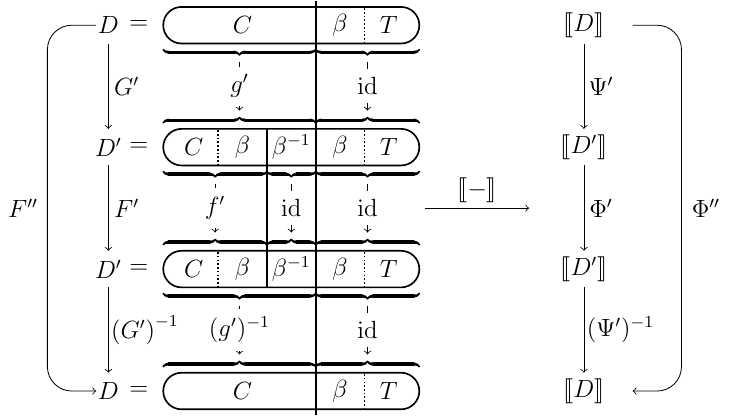}
\end{equation}
Now note that the link cobordism $F''$ acts entirely on the crossingless tangle $C$. 
In particular, this implies that the nonzero components of the induced
chain map $\Phi''$ preserve the resolutions of all crossings of $\beta T$.
Since $D$ has no other crossings, this further implies
that the chain map $\Phi''$ preserves the cubical structure of $\llbracket D\rrbracket$,
and is thus given
by components $\Phi_\alpha''\colon\llbracket D\rrbracket_\alpha\rightarrow\llbracket D\rrbracket_\alpha$.
We further note:

\begin{itemize}
\item $\Phi''$ is a homotopy equivalence since it is a composition of homotopy equivalences.
\item The corresponding homotopies also preserve the cubical structure.
\item Since homotopies need to have homological degree $-1$, this implies that these homotopies must be
zero maps.
\item It follows that the homotopy equivalence $\Phi''$ is actually a chain isomorphism.
\item It follows that each $\Phi''_\alpha$ is an isomorphism.
\item Since $F''$ is an identity cobordism on $\beta T$, the components $\Phi''_\alpha$
must be given by identity cobordisms above the resolutions $(\beta T)_\alpha$.
\item The portion of $\Phi''_\alpha$ that lies above $C$ must necessarily be a scalar multiple of the identity cobordism of $C$
for an invertible scalar $r_\alpha\in\Bbbk^\times$.
This follows because $\Phi''_\alpha$ is an isomorphism (of quantum degree zero) and $C$ has no closed components.
\end{itemize}

In conclusion, we see that $\Phi''$ is a chain isomorphism which has the structure described in Lemma~\ref{lem:3:uniqueness}. It thus follows
that $\Phi''=c\operatorname{id}$ for an overall scalar $c\in\Bbbk^\times$.

All that is left is to prove Claim \ref{claim:t21}.
Since $\Psi$ is induced by a sequence of  Reidemeister II moves, we can assume without loss of generality
that each of the affine braids $\beta$ and $\beta^{-1}$ contains a single crossing, and $\Psi$ is induced by a single  Reidemeister II move.
To prove Claim \ref{claim:t21}, we then need to consider the composition
\[
\Psi^{-1}\circ\Phi'\circ\Psi
\]
where $\Psi^{-1}$ denotes the homotopy inverse of $\Psi$ from~\eqref{eq:3:RIImap}.
An examination of the chain maps in~\eqref{eq:3:RIImap} shows that they are given by identity components and the components $f$ and $g$
from \eqref{eq:3:RIImap} (not to be confused with the equally-named cobordisms considered before).
Thus, the above composition is equal to
\[
\Phi'_{10}+f\circ\Phi'_{01}\circ g
\]
where
$\Phi'_{10}$ and $\Phi'_{01}$ denote the restrictions of $\Phi'$ to the lower and upper vertex in the middle
of the square at the bottom of \eqref{eq:3:RIImap}, respectively. Now note that at the locations where $f$ and $g$ are nontrivial,
$\Phi'_{01}$ is given by identity cobordisms, and thus each component of $f\circ\Phi'_{01}\circ g$ contains a 2-sphere coming from the death in $f$
and the birth in $g$. Because of the (S) relation, this means that these components are zero, whence
\[
\Psi^{-1}\circ\Phi'\circ\Psi=\Phi'_{10}.
\]
Finally, $\Phi'$ and $\Phi$ are induced by the same sequence of Reidemeister moves. This implies that the chain maps
$\Phi'_{10}$ and $\Phi$ coincide, except when $F$ contains a positive Reidemeister I move, in which case they
may differ by a global factor of $\pi$ coming from the factors of $\pi$ in \eqref{eq:3:RImap}. Either way, Claim~\ref{claim:t21} follows.
\end{proof}

The lemma just shown immediately implies invariance
 under type II movie moves up to 
 homotopy and overall invertible scalars.
 It can be seen as a replacement for Bar-Natan's results about $Kh-$simple tangles from
 \cite{Bn2005}.
 We remark that an argument vaguely related to
 our proof 
 of Lemma~\ref{lem:t2} 
 has appeared in~\cite{Sa2018} in a somewhat different context, although we were unaware of the
 latter when we first found our proof.

\subsection{Type III Movie Moves}

The last five movie moves involve births or deaths in addition to Reidemeister type cobordisms.  Additionally, these movie moves viewed either forward or in reverse must be treated as independent moves.  
\begin{convention} The following conventions will be used throughout the proof of invariance under movie moves 11 through 15.
    \begin{enumerate}[label=\textnormal{\bf{\alph*.}}]
        \item The initial frame of a movie is the $0$\ordth frame.
        \item Specific resolutions of the $i$\ordth frame $D_i$ will be denoted by $D_{i,\alpha}$.
        \item For components of differentials within $\llbracket D_i\rrbracket$, we will use
        $d_{i,{\star}\alpha}$ to denote the underlying planar cobordism and $\partial_{i,{\star}\alpha}\coloneqq\epsilon_{i,{\star}\alpha}d_{i,{\star}\alpha}$ for the actual component.
        \item For components of chain maps from $\llbracket D_i\rrbracket$ to $\llbracket D_{i+1}\rrbracket$, we will use $\varphi_{i,\alpha}^{i+1,\beta}$ to denote the underlying planar cobordisms and $\Phi_{i,\alpha}^{i+1,\beta}\coloneqq\epsilon_{i,\alpha}^{i+1,\beta}\varphi_{i,\alpha}^{i+1,\beta}$ for the actual component.
        \item Within a movie move, we may affix an arrow to any of the following symbols $\mathbin{\Diamond}\in\{D,d,\epsilon,\partial,\varphi,\Phi\}$
        to denote if it comes from the left movie ($\overleftarrow{\mathbin{\Diamond}}$) or the right movie ($\overrightarrow{\mathbin{\Diamond}}$).
    \end{enumerate}
\end{convention}

We will often use that the maps assigned to link cobordisms are natural
in the choice of the additional data $(o,\epsilon)\in M(L)$.
To show that these maps these maps are invariant under a particular movie move, it is therefore sufficient
to consider a particular choice of $(o,\epsilon)$.

\subsubsection{Movie Move 11}

\begin{figure}[H]
    \centering
    \includeFig{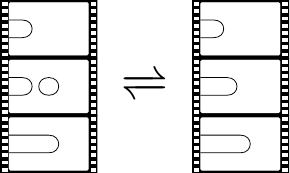}
    \caption{Movie move 11}
    \label{fig:4:MM11}
\end{figure}

All variations of movie move 11, viewed in either direction and up to orientations are the relations we referred to as \enquote{pruning} relations earlier which at most induce multiplication by $\pi$.  The direction we view the cobordism and orientations placed on the critical points determine if a $\pi$ appears in the relation. These are global decisions, thus $\overleftarrow{\Phi}\pieq\overrightarrow{\Phi}$.

\subsubsection{Movie Move 12}

\begin{figure}[H]
    \centering
    \includeFig{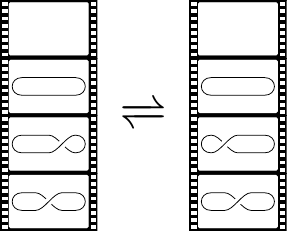}
    \caption{Movie move 12}
    \label{fig:4:MM12}
\end{figure}

Let $\overleftarrow{\Phi}$ and $\overrightarrow{\Phi}$ be the chain maps assigned to the two sides
of Figure~\ref{fig:4:MM12}, and $B$ denote the birth chain map that corresponds to the first transition in each movie. Moreover, let $R$
denote the homotopy inverse of the Reidemeister I map that appears in $\overleftarrow{\Phi}$.
Note that since $R$ is a proper left-inverse of this Reidemeister I map, we have $R\circ \overleftarrow{\Phi}=B$.

On the other hand, $R\circ\overrightarrow{\Phi}$ is a chain map that preserves the cubical structure, and whose
components are locally given by morphisms $\emptyset\rightarrow\bigcirc$ of quantum degree $1$.
Now note that the only such morphisms are scalar multiples of planar birth cobordisms, and an argument
similar to the one used in the proof of Lemma~\ref{lem:t2} shows that the multiple appearing at one vertex
of the cube determines the multiple at all other vertices (this uses that any planar cobordism consisting of a birth and a distant saddle has trivial annihilator in $\Bbbk$). Thus,
\[
R\circ\overrightarrow{\Phi}=aB=aR\circ\overleftarrow{\Phi}
\]
for a global scalar $a\in\Bbbk$, and so $\overrightarrow{\Phi}\simeq a\overleftarrow{\Phi}$. Reversing the roles of $\overleftarrow{\Phi}$ and $\overrightarrow{\Phi}$, we also get $\overleftarrow{\Phi}\simeq b\overrightarrow{\Phi}$ for
a global scalar $b\in\Bbbk$. Consequently, $\overleftarrow{\Phi}\simeq ab\overleftarrow{\Phi}$, and post-composing with $R$ yields
$B\simeq abB$. Since the latter homotopy is cubical, we further have $B=abB$, and restricting to any vertex of the cube now yields $ab=1$ because planar birth cobordisms have trivial annihilator in $\Bbbk$. In conclusion,
\[
\overrightarrow{\Phi}\simeq a\overleftarrow{\Phi}
\]
for an invertible scalar $a\in\Bbbk$. The same argument also works mutatis mutandis for all other version of movie move 12.

\subsubsection{Movie Move 13}

\begin{figure}[H]
    \centering
    \includeFig{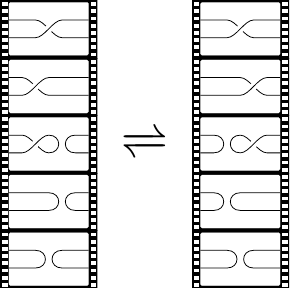}
    \caption{Movie move 13}
    \label{fig:4:MM13}
\end{figure}

\subsubsection*{Movie Move 13 Forward:}

{\noindent}Let $L$ be the larger link diagram that appears on the left-hand side of equation \eqref{fig:MM13FEQ}.
The chain maps that appear in the two cones of this equation correspond precisely to the maps induced
by the saddles on the two sides of Figure~\ref{fig:4:MM13}.
\begin{equation}
    \label{fig:MM13FEQ}
    \includeFig{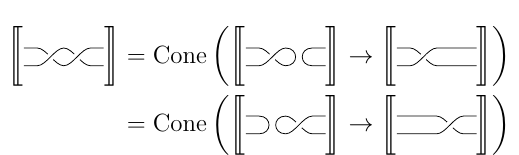}
\end{equation}
We will orient the crossings as in the following diagram. This produces canonical orientations on all saddles.

\begin{equation}
    \includeTang{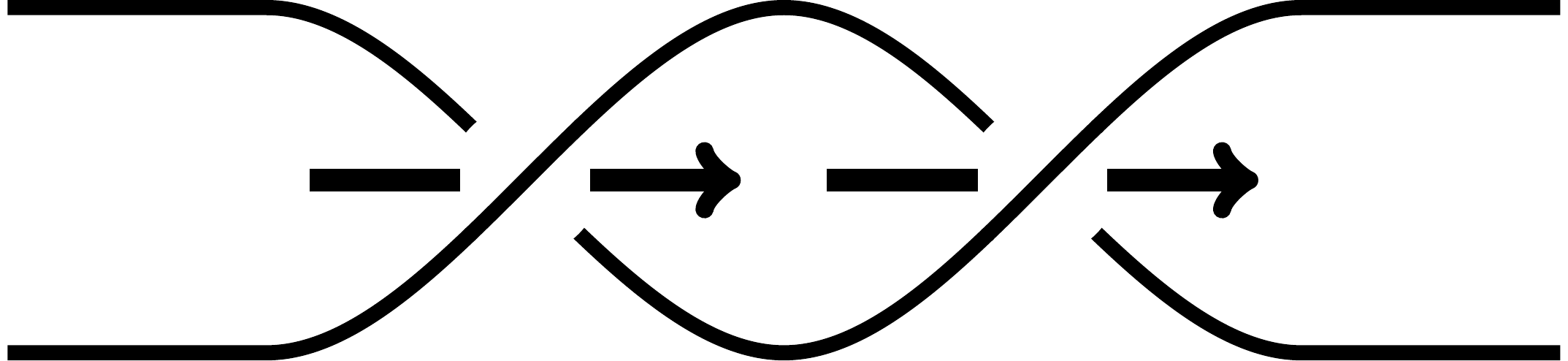}
\end{equation}

Consider the generalized Khovanov cube of the larger link diagram $L$
with signs $\psi^1_\alpha,\psi^2_\alpha,\psi^3_\alpha,\psi^4_\alpha\in\Bbbk^\times$ on the shown edges. 
\begin{equation}
    \includeFig{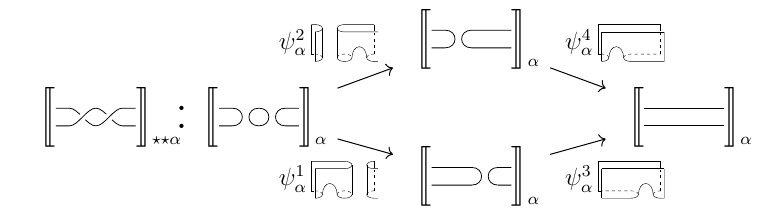}
    \label{fig:MM13Cube}
\end{equation}
In this cube, we can choose identical internal sign assignments for the three diagrams
that correspond to the left three vertices of the square in \eqref{fig:MM13Cube}.
Since the left two edge maps in \eqref{fig:MM13Cube} are always merges,
the faces that they can form when paired with maps from an outside crossing
will be from group C.
This means they are commuting faces, and thus we can choose
$\psi^1_\alpha=\psi^2_\alpha=(-1)^{deg(\alpha)}$.
The right two edge maps in \eqref{fig:MM13Cube} have essentially the same underlying saddle cobordism, and thus
for any valid sign assignment, we can get a new valid sign assignment by exchanging each $\psi^3_\alpha$ with each $\psi^4_\alpha$.

To compute the saddle maps from the movie move, we can therefore make
the following assignments, where the extra terms involving $deg(\alpha)$ are to ensure that these
maps commute with the differentials.
\begin{equation}
\addtolength{\arraycolsep}{2em}    
\begin{array}{cc}
\overleftarrow{\epsilon}_{1,{\star}\alpha}=\overrightarrow{\epsilon}_{1,{\star}\alpha}=\psi_\alpha^1=(-1)^{\deg(\alpha)} &
\overleftarrow{\epsilon}_{1,0\alpha}^{2,0\alpha}=\overrightarrow{\epsilon}_{1,0\alpha}^{2,0\alpha}=(-1)^{\deg(\alpha)}\psi_\alpha^1=1\\
\overleftarrow{\epsilon}_{1,1\alpha}^{2,1\alpha}=\overrightarrow{\epsilon}_{1,1\alpha}^{2,1\alpha}=(-1)^{\deg(\alpha)+1}\psi_\alpha^3 &
\overleftarrow{\epsilon}_{2,{\star}\alpha}=\overrightarrow{\epsilon}_{2,{\star}\alpha}=\psi_\alpha^4
\end{array}
\addtolength{\arraycolsep}{-2em}    
\end{equation}

We arrive at the following diagrams describing the chain maps induced by the two sides of the movie move.

\begin{equation}
    \includeFig{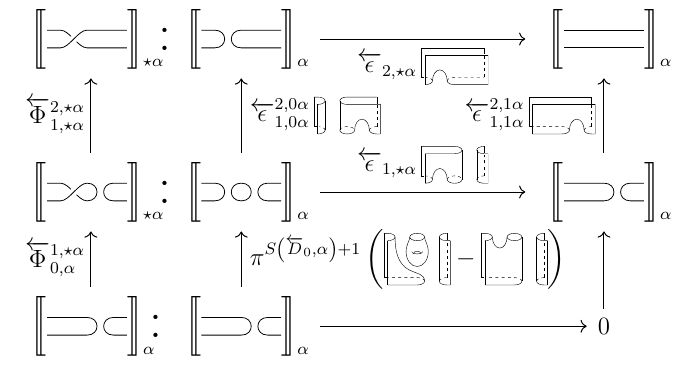}
    \label{fig:MM13FL}
\end{equation}

\begin{equation}
    \includeFig{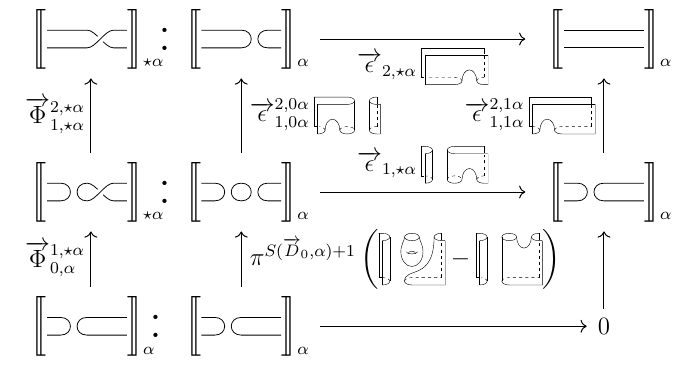}
    \label{fig:MM13FR}
\end{equation}

Note that $\overleftarrow{D}_{0,\alpha}$ and $\overrightarrow{D}_{0,\alpha}$ are isotopic and thus produce the same values on the $S$-map. Using the (4Tu) and (2H) relations, we arrive at the following equation for a fixed degree $\alpha$

\begin{equation}
    \includeFig{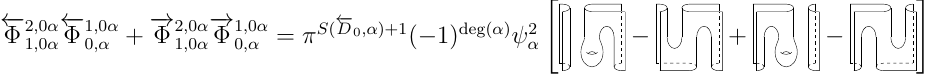}
\end{equation}

To see that this is zero, we look at the two possible ways to close the tangle. If the closure is \includeSym{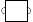} then

\begin{equation}
    \includeFig{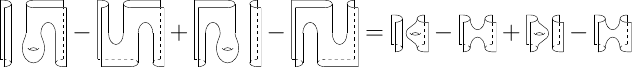}
\end{equation}

which equals zero by a variation of the (4Tu) relation.  Alternatively if we use the other closure, \includeSym{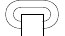}, then 

\begin{equation}
    \includeFig{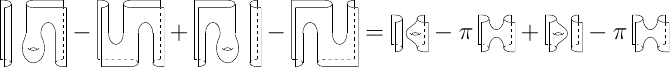}
\end{equation}

which also becomes zero if we apply the (H) relation and then the same variation of the (4Tu) relation.  If we define the overall sign assignment as discussed for each $\alpha$, we thus obtain $\overleftarrow{\Phi}=-\overrightarrow{\Phi}$.

\subsubsection*{Movie Move 13 Reverse:}

{\noindent}In the reverse direction, we can still view the saddles in the left and right movies
as coming from larger link diagrams, but not from the same larger link diagram:
\begin{equation}
    \includeFig{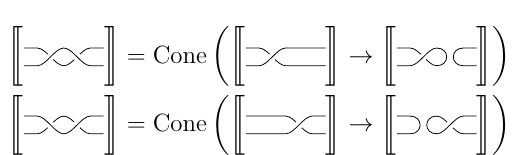}
    \label{fig:MM13REQ}
\end{equation}

We will use the following crossing orientation so that both sides of the movie feature faces from either a group X or Y for any fixed resolution $\alpha$ of the outside crossings.

\begin{equation}
    \includeTang{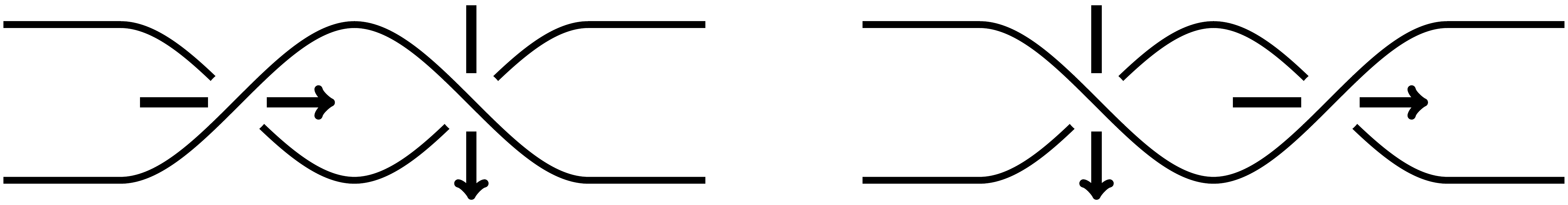}
\end{equation}

The chain maps induced by the two sides of the movie move are now given
by the following diagrams, where the split cobordisms $\overleftarrow{\varphi}_{0,0\alpha}^{1,0\alpha}$ and $\overrightarrow{\varphi}_{0,0\alpha}^{1,0\alpha}$ have canonical orientations (pointing towards the reader) because of the above orientation choice.

\begin{equation}
    \includeFig{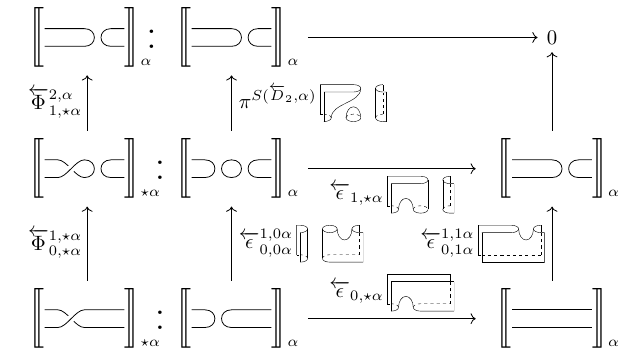}
    \label{fig:MM13RL}
\end{equation}

\begin{equation}
    \includeFig{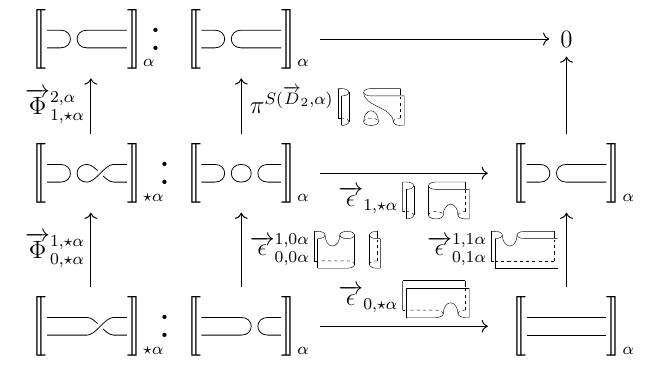}
    \label{fig:MM13RR}
\end{equation}

As in the forward direction, we use the same internal sign assignments for any vertices 
whose associated diagrams are ambient isotopic, or ambient isotopic up to adding a small circle.
The cobordisms $\overleftarrow{\varphi}_{0,{\star}\alpha}$ and $\overrightarrow{\varphi}_{0,{\star}\alpha}$ in the bottom rows of \eqref{fig:MM13RL} and \eqref{fig:MM13RR} are essentially the same. Likewise, $\overleftarrow{\varphi}_{0,1\alpha}^{1,1\alpha}$ and $\overrightarrow{\varphi}_{0,1\alpha}^{1,1\alpha}$ are essentially the same.  

The merge cobordisms $\overleftarrow{\varphi}_{1,{\star}\alpha}$ and $\overrightarrow{\varphi}_{1,{\star}\alpha}$ are not identical, but together with maps from external crossings, they form faces that are either from group C or X and are thus always commuting, and therefore can share sign assignments.  Similarly, in relation to maps that come from a given external crossing, the split cobordisms $\overleftarrow{\varphi}_{0,0\alpha}^{1,0\alpha}$ and $\overrightarrow{\varphi}_{0,0\alpha}^{1,0\alpha}$ are either both commuting---forming face from group C---or both $\pi$-commuting, forming a face from group $\pi$. It follows that for any given valid sign assignment for \eqref{fig:MM13RL}, we can define a valid sign assignment for
\eqref{fig:MM13RR} by

\begin{equation}
\addtolength{\arraycolsep}{1.5em}    
\begin{array}{cccc}
\overrightarrow{\epsilon}_{0,{\star}\alpha} = \overleftarrow{\epsilon}_{0,{\star}\alpha} &
\overrightarrow{\epsilon}_{0,0\alpha}^{1,0\alpha} = \overleftarrow{\epsilon}_{0,0\alpha}^{1,0\alpha} &
\overrightarrow{\epsilon}_{0,1\alpha}^{1,1\alpha} = \overleftarrow{\epsilon}_{0,1\alpha}^{1,1\alpha} &
\overrightarrow{\epsilon}_{1,{\star}\alpha} = \overleftarrow{\epsilon}_{1,{\star}\alpha}
\end{array}
\addtolength{\arraycolsep}{-1.5em}    
\end{equation}

This in turn yields the following equation.

\begin{equation}
    \includeFig{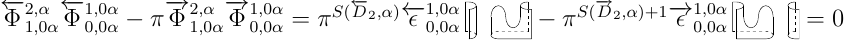}
\end{equation}

The sign assignment defined for a given $\alpha$ propagates to the entire cube, thus $\overleftarrow{\Phi}=\pi\overrightarrow{\Phi}$.

\subsubsection*{Movie Move 13 Alternative Variants:} 

{\noindent}There are additional variants of this move that must be considered.  The move is built off of a saddle and a Reidemeister I move, so we need to consider the movie moves shown in \eqref{fig:MM13AF} and \eqref{fig:MM13AR}, where negative crossings appear in the Reidemeister I moves in place of the positive crossings.  
\begin{equation}
    \includeFig{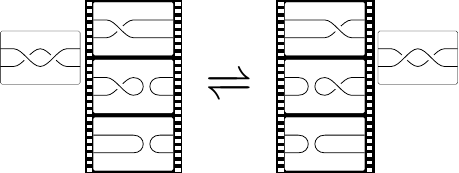}
    \label{fig:MM13AF}
\end{equation}
\begin{equation}
    \includeFig{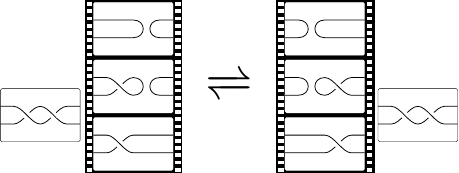}
    \label{fig:MM13AR}
\end{equation}
The arguments will still work as before, but now the maps will be supported in degree one, not zero.  In the forward direction, the larger link diagrams resemble the two versions of a Reidemeister II move, which were present in the original reverse direction. Therefore we use a similar argument to the original reverse direction.  Likewise, in the reverse direction, the larger link diagram resembles that from the original forward direction, but with opposite crossings. The argument thus follows that of the original, forward direction.

\subsubsection{Movie Move 14}

\begin{figure}[H]
    \centering
    \includeFig{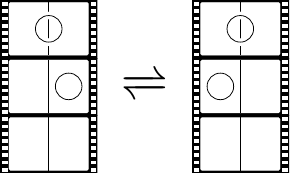}
    \caption{Movie move 14}
    \label{fig:4:MM14}
\end{figure}

For our computations, we will use a version of this move where all frames are rotated $90^\circ$ clockwise.

\subsubsection*{Movie Move 14 Forward:} 

{\noindent}The final diagram of both sides decomposes into the following square with the same shared sign assignment on both sides of the move:
\begin{equation}
    \includeFig{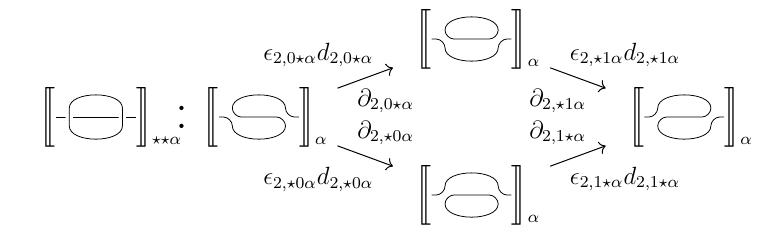}
    \label{fig:MM14topsquare}
\end{equation}
We will endow the crossings with the following orientations so that the above square forms a face from group Y.
\begin{equation}
    \includeTang{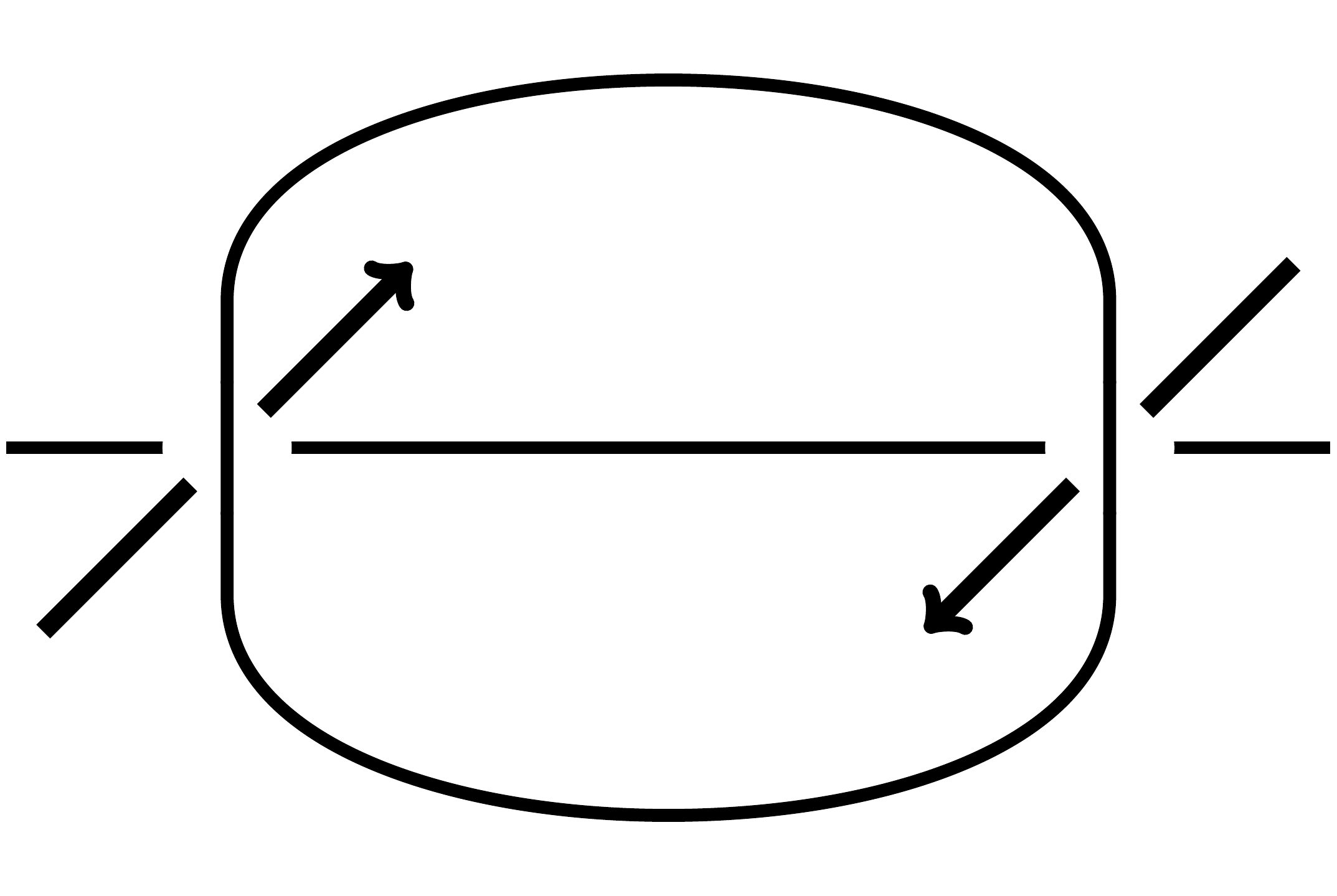}
    \label{dig:MM14squareOrient}
\end{equation}
We will further equip the two vertices in the middle of \eqref{fig:MM14topsquare} with the same internal sign assignments. Note that this is possible because of Lemma~\ref{lem:3:CWsubcomplex}, and because the diagrams corresponding to these vertices represent ambient isotopic tangles up to changing the location of the trivial circle. Since the two right maps in \eqref{fig:MM14topsquare} are always merges, and always commute
with maps from outside crossings, we can choose $\epsilon_{2,1\star\alpha}=\epsilon_{2,\star 1\alpha}$ (and similarly $\epsilon_{2,\star0\alpha}=-\pi\epsilon_{2,0\star\alpha}$ since the square in \eqref{fig:MM14topsquare} is from group Y).

As each side of the move begins with no crossings in the tangle, the chain maps assigned to the two sides of the move are supported only in the zero degree.  We will consider the diagrams of each side of the move restricted to the relevant degree.

\begin{equation}
    \includeFig{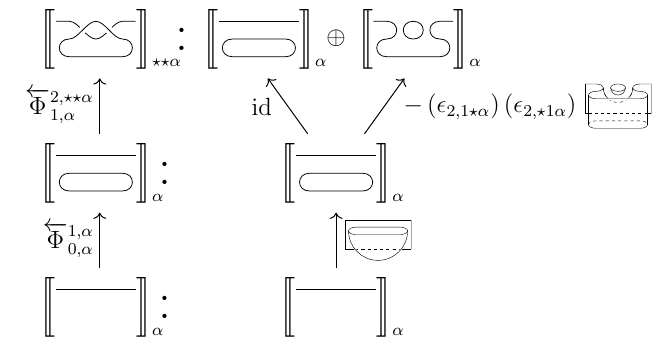}
    \label{fig:MM14FL}
\end{equation}

\begin{equation}
    \includeFig{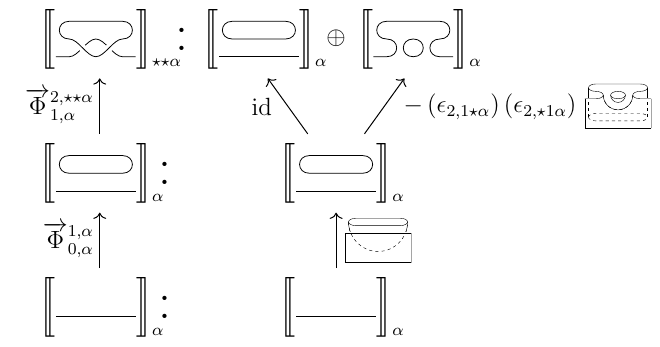}
    \label{fig:MM1FR4}
\end{equation}

Using our orientation conventions, the Reidemeister II chain maps naturally induce the canonical orientation on the saddle on the left-hand side of the move, while they generate a non-canonically oriented saddle on the right-hand side.  We can correct this orientation on the saddle at no cost, as the saddle is always a merge cobordism. In particular, the cobordisms underlying the right maps in each diagram simplify in the following manner.

{\noindent}\begin{minipage}{0.5\textwidth}
    \begin{equation}
    \includeFig{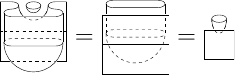}
    \end{equation}
\end{minipage}%
\begin{minipage}{0.5\textwidth}
    \begin{equation}
    \includeFig{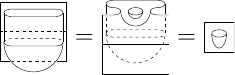}
    \end{equation}
\end{minipage}

By our sign choices, we further have $(\epsilon_{2,\star 1\alpha})(\epsilon_{2,1\star\alpha})=1$.
Thus, for each $\alpha$ we have the following equation

\begin{equation}
    \includeFig{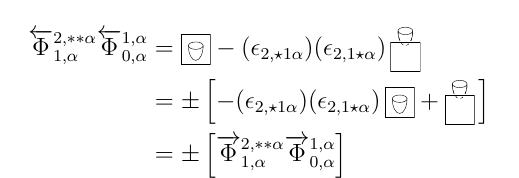}
\end{equation}

and thus $\overleftarrow{\Phi}=-\overrightarrow{\Phi}$.

\subsubsection*{Movie Move 14 Reverse:}

{\noindent}In the reverse direction, an almost identical argument applies, but there is one additional consideration. Namely, we have

{\noindent}\begin{minipage}{0.5\textwidth}
    \begin{equation}
    \includeFig{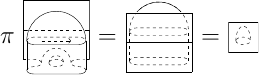}
    \end{equation}
\end{minipage}%
\begin{minipage}{0.5\textwidth}
    \begin{equation}
    \includeFig{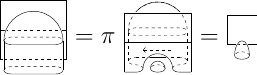}
    \end{equation}
\end{minipage}

In the above equations, the simplifications involve eliminating a split followed by a death, which can induce a factor of $\pi$.  These factors of $\pi$ do not impose any issues for our argument as a factor of $\pi$ is incurred on both sides of the move.  All other parts of the argument in the forward direction apply to the reverse direction.

\subsubsection*{Movie Move 14 Alternative Variant:}

{\noindent}There is an additional variant of movie move 14 in which the circle passes under the strand instead of over. The argument for the initial variant also applies in this alternative setting.

\subsubsection{Movie Move 15}

\begin{figure}[H]
    \centering
    \includeFig{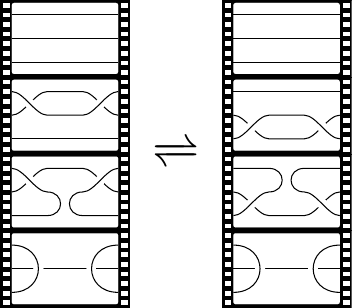}
    \caption{Movie move 15}
    \label{fig:4:MM15}
\end{figure}

Invariance under movie move 15 follows immediately from Lemma~\ref{lem:3:RIIIinv} because the maps
assigned to the two sides of this move are precisely the maps $\Phi_L$ and $\Phi_R$ from the lemma
(or the reverse directions of these maps, which can be treated similarly).

\subsection{Chronological Movie Moves}

The link cobordism category allows ambient isotopies that produce changes in the chronology of the planar cobordisms that appear in the induced chain maps.  Therefore, in order to show the functoriality of generalized Khovanov homology up to sign, we must show that exchanging the order of pairs of distant link cobordisms at most induces an overall change in sign.

\subsubsection {Commuting Deaths with other Cobordisms}
If we work with the refined generalized Khovanov bracket from
Subsection~\ref{subs:supergraded}, then all chain maps
assigned to oriented link cobordisms are homogeneous with respect
to the supergrading by Remark~\ref{rem:cobordismsuperdegree}.
Moreover, the death map on the refined Khovanov bracket can be
viewed as a component of a $\pi$-supernatural transformation
by Remark~\ref{rem:deathpisupernatural}. 

We further claim that
 the functor $\mathcal{F}$ from Remark~\ref{rem:deathpisupernatural}
 takes the chain map $\Phi_F$ induced by a link cobordism
$F\subset(\mathbb{R}^2\setminus D)\times\mathbb{R}\times I$ to the chain map induced by the link cobordism $F\cup (c\times I)$
(where $c$ and $D$ are as in the remark).
Indeed, this is obvious when $F$ is a birth or a saddle cobordism or a Reidemeister type cobordism
that is not of positive type I. When $F$ is a death cobordism or a positive Reidemeister I cobordism,
then the chain map $\Phi_F$ involves factors of $\pi$ that depend on
the $S$-map from equation~\eqref{eq:3:Smap}. However, the equation
$\mathcal{F}(\Phi_F)=\Phi_{F\cup(c\times I)}$
still holds in this case because the value of the $S$-map does not change if one adds the trivial circle $c$ to a link diagram.

In conclusion,
equation~\eqref{eqn:supernaturaldefinition} from the definition
of a $\pi$-supernatural transformation 
implies that for any link cobordism $F\subset(\mathbb{R}^2\setminus D)\times\mathbb{R}\times I$ as above,
the death map intertwines with the maps $\Phi_{F\cup(c\times I)}$ and $\Phi_F$, up to a possible overall factor of $\pi$ that depends
on the superdegree of $F$.

\subsubsection{Commuting Births with other Cobordisms}
One can check directly that the birth map is a component of an (ordinary) natural transformation
from the inclusion functor $\mathcal{C}\rightarrow\cobIIIgen$ from Remark~\ref{rem:deathpisupernatural} to the functor $\mathcal{F}$.
Repeating the argument that we just used for death maps, one thus see that
birth maps commute with maps induced by other
link cobordisms.

\subsubsection{Commuting a Pair of Saddle Cobordisms}

We will show that commuting two saddle cobordisms induces an overall sign.  Consider the odd Khovanov bracket of the larger link with two additional crossings corresponding with the two saddles.  If we fix a resolution of outside crossings, then the square corresponding with the resolution of the two distinguished vertices anticommutes:

\begin{equation}
    \Psi_{{\star}1\alpha}\Psi_{0{\star}\alpha}=-\Psi_{1{\star}\alpha}\Psi_{{\star}0\alpha}
\end{equation}

The two ways we travel around the face correspond with the rearrangement of the saddles' order. To produce the final saddle chain maps, we need to multiply by signs on odd degrees so that the maps commute with the differentials.

\begin{equation}
    \hspace{-0.5in}\includeFig{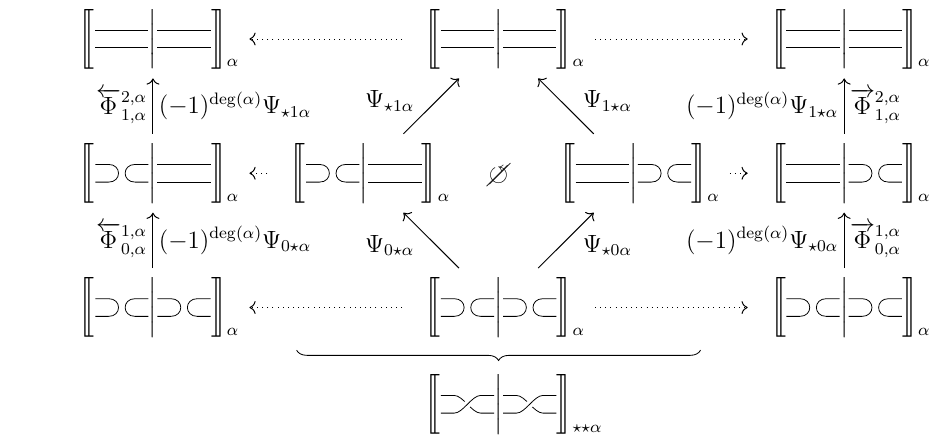}
\end{equation}

This yields the following equation.

\begin{equation}
\begin{split}
    \overleftarrow{\Phi}_{1,\alpha}^{2,\alpha}\overleftarrow{\Phi}_{0,\alpha}^{1,\alpha} &=
    (-1)^{\deg(\alpha)}\Psi_{{\star}1\alpha}(-1)^{\deg(\alpha)}\Psi_{0{\star}\alpha}\\ &=-
    \left[(-1)^{\deg(\alpha)}\Psi_{1{\star}\alpha}(-1)^{\deg(\alpha)}\Psi_{{\star}0\alpha}\right]\\ &=
    -\overrightarrow{\Phi}_{1,\alpha}^{2,\alpha}\overrightarrow{\Phi}_{0,\alpha}^{1,\alpha}
\end{split}
\end{equation}

Therefore, generally for commuting two saddle cobordisms, $\overleftarrow{\Phi}=-\overrightarrow{\Phi}$.

\subsubsection{Commuting Saddles with Reidemeister Type Cobordisms}

Will show that a saddle chain map and a Reidemeister chain map commute or anticommute.
Consider the larger link diagram $L$ with an additional crossing corresponding to the saddle. 
Then look at the map on the Khovanov bracket given by performing the Reidemeister move on $L$, yielding
a new larger link diagram $L'$. Since this map is a chain map, it intertwines with the differentials in $\llbracket L\rrbracket$ and $\llbracket L'\rrbracket$,
and hence with the components of these differentials that correspond to the saddle maps.
The actual saddle maps are obtained from these components by multiplying by prefactors of the form
$(-1)^{deg(\alpha)}$ and $(-1)^{deg(\beta)}$, where $\alpha$ and $\beta$ denote
indices of vertices in (layers of) the cubes of $L$ and $L'$.
The main observation is now that the components of the Reidemeister chain map preserve the homological
degree, and thus preserve $deg(\alpha)$ up to an overall shift. In the relevant calculations, we can therefore assume that
$deg(\beta)=deg(\alpha)+c$ for an overall constant $c$. It follows that the Reidemeister chain map commutes with the
saddle chain map if $c$ is even and anticommutes otherwise.

\subsubsection{Commuting a Pair of Reidemeister Type Cobordisms}

We are left to show that changing the chronological order of a pair of distant Reidemeister type cobordisms at most incurs an overall
factor of $\pm 1$ or $\pm\pi$. Let $F$ and $G$ be two cobordisms arising from Reidemeister moves performed in disjoint disks $D_F$ and $D_G$.
Further, let $L$ be the initial link diagram for these cobordisms.
We first observe the following.

\begin{lemma}\label{lem:worm}
Let $x_F$ and $x_G$ be distinguished points on the boundaries of $D_F$ and $D_G$ away from the strands of $L$.
There exists a disk $D\subset\mathbb{R}^2$ containing both $D_F$ and $D_G$ such that $\partial D_F\cup \partial D_G$
is contained in $\partial D$ except for arbitrarily small neighborhoods around the distinguished points,
and such that the tangle $L\cap D$ differs from $L\cap(D_F\cup D_G)$ by adding finitely many extra
strands and no extra crossings.
\end{lemma}

\begin{proof}
    Choose a non-self-intersecting path $p\subset\mathbb{R}^2$ from $x_F$ to $x_G$ that
    avoids the crossings of $L$ and meets $D_F$ and $D_G$ only at its endpoints, and that
     intersects $L$ transversely at finitely many places. The desired disk $D$ is then given by
     $D=D_F\cup D_G\cup P$, where $P\subset\mathbb{R}^2$ is a narrow band obtained by thickening $p$.
\end{proof}

Depending on the Reidemeister move performed on each of the disks $D_F$ and $D_G$,
we will now choose the points $x_F$ and $x_G$ on the boundaries of $D_F$ and $D_G$ according to the rules of acceptable point placement shown in Figure \ref{fig:wormrules}.  

\begin{figure}[H]
        \centering
        \captionsetup{width=4.5in}
        \includeTang{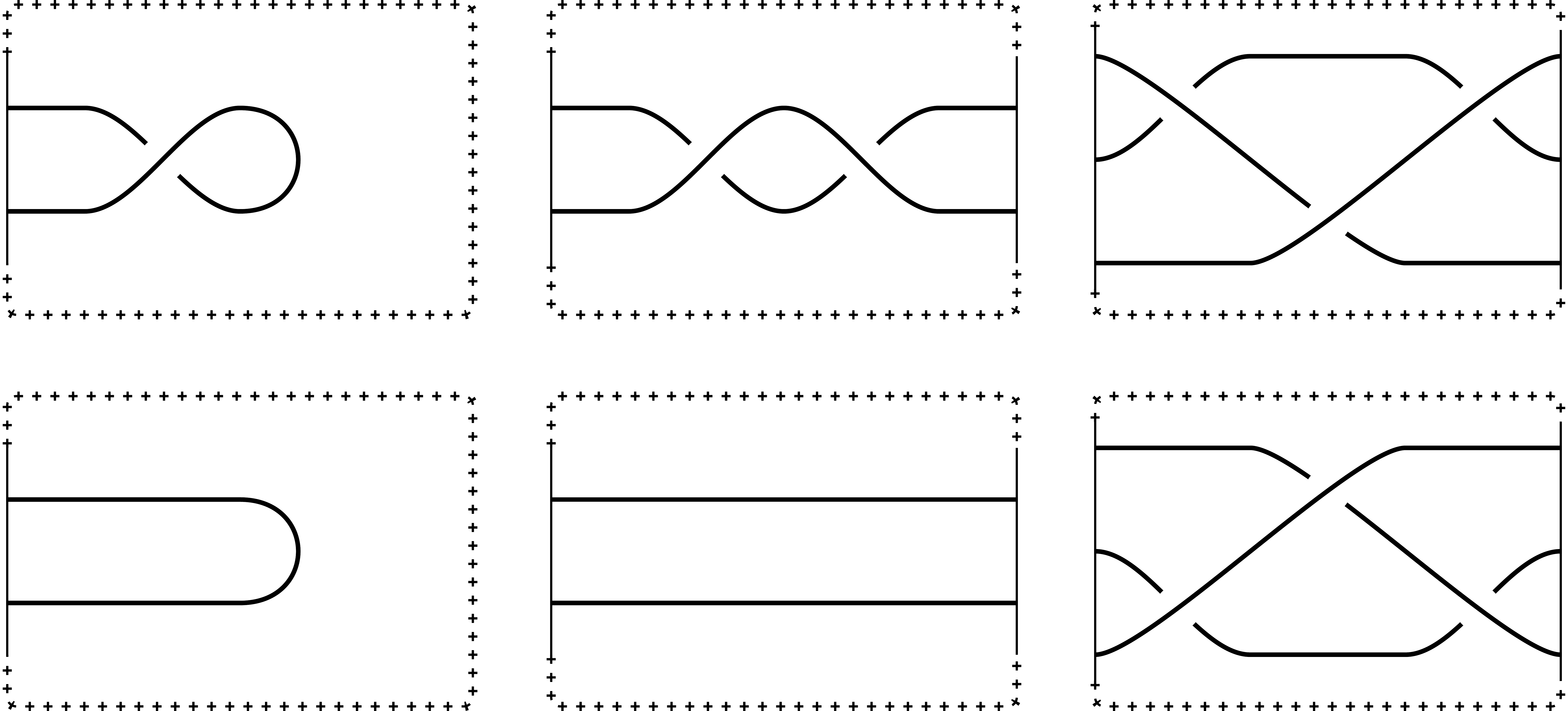}
        \caption{The ticked boundary indicates where the path should connect to the disk containing each Reidemeister move}
        \label{fig:wormrules}
\end{figure}

Let $D$ be the larger disk produced by applying the Lemma \ref{lem:worm} to our link diagram $L$ with the disks $D_F$ and $D_G$ and the chosen points $x_F$ and $x_G$.  We have ostensibly dug a tunnel from the disk containing the tangle $L\cap D_F$ to the disk containing the tangle $L\cap D_G$ in order to produce a single larger tangle $T:=L\cap D$.
Consider the commutator of cobordisms $H\coloneqq FGF^{-1}G^{-1}$, which begins and ends at identical diagrams,
and which is nontrivial only above $D$.
The tangle $T$ satisfies the assumptions of Lemma~\ref{lem:t2}, and thus this lemma implies that $H$
induces a chain map homotopic to $\pm\id$ or $\pm\pi\id$. 
It follows that either order of composing $F$ and $G$ produces homotopic chain maps up to a global invertible scalar.

\section{Non-invariance in $S^3\times I$}\label{s:noninvariance}
In this section, we prove Theorem~\ref{thm:noninvariance} by giving a family of examples
showing that the odd cobordism maps are not invariant under smooth 
ambient isotopies in $S^3\times I$. As mentioned in the introduction, the same examples
will also show that the odd cobordism maps are not invariant under ribbon moves.

Following \cite{Ka2022}, let
$U_2=A\sqcup B$ denote a 2-component unlink consisting of an inner component $A$ and an outer component $B$, and
let $D_1$ and $D_2$ be the disks shown below.

\begin{figure}[H]
    \centering
    \includeCob{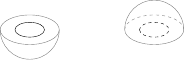}
    \caption{The disks $D_1$ and $D_2$ on the left and right respectively}
\end{figure}

Note that each of these disks is disjoint from $A$, and can thus be viewed as a Seifert surface for $B$
in $S^3\setminus A$. As such, the disks $D_1$ and $D_2$ are ambient isotopic because
we can transform one into the other by pushing it through the point $\infty\in S^3$.
By pushing the interior of each disk into the 4th dimension, we can further regard $D_1$
and $D_2$ as link cobordisms $\emptyset\rightarrow B$, and adding the identity
cobordism of $A$, we obtain link cobordisms $F_1,F_2\colon A\rightarrow U_2$.
These cobordisms are still ambient isotopic in $S^3\times I$, but we
claim that they induce different maps on odd Khovanov homology.

To see this, note that $F_1$ and $F_2$ are represented by the following movies:

\begin{figure}[H]
    \centering
    \includeFig{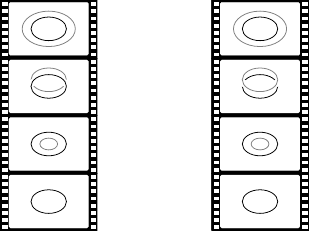}
    \caption{The cobordisms $F_1$ and $F_2$ on the left and right respectively}
\end{figure}

The first half of each movie is given by the birth of a new component, whereas
in the second half, this component is moved over or under the other component
(so that the two components switch places). Using notation from \cite{ORS2007},
the odd Khovanov homology groups of $A$ and $U_2$ are given by
$OKh(A)=\operatorname{Span}\{1,A\}$ and $OKh(U_2)=\operatorname{Span}\{1,A,B,A\wedge B\}$.
A direct calculation now shows that the maps
\[
OKh(A)\longrightarrow OKh(U_2)
\]
induced by $F_1$ and $F_2$ are given by
\[1 \mapsto 1,\qquad A\mapsto A\pm (A-B)\]
where
the sign denoted $\pm$ is different for
$F_1$ and $F_2$, and where we are ignoring possible overall signs.
On the generator $A$, the values of the two maps thus differ by $\pm 2(A-B)\neq 0$ or by $\pm 2A\neq 0$,
depending on whether the overall signs are assumed to be the same or different.

We can easily generalize this example to obtain an infinite family of link cobordisms, as described in
Theorem~\ref{thm:noninvariance}. To this end, let $R\colon U_2\rightarrow U_2$ be the link cobordism
$R=R_2\circ R_1^{-1}$, where the cobordisms $R_1$ and $R_2$ are given by the second halves of the movies of $F_1$
and $F_2$, respectively. Then the cobordisms $F_{n}:=R^{(n-1)}\circ F_1$ are all ambient isotopic in $S^3\times I$, but
the maps that they induce on odd Khovanov homology are different since they are given by
\[1 \mapsto 1,\qquad A\mapsto A+ (3-2n)(A-B)\]
(the computations are left the reader since a conceptual interpretation of the above formulas will be given in \cite{MW2024b}).
We end this section by noting that the cobordisms $F_1$ and $F_2$ are related by a ribbon move
(this can be seen, for example, by looking at their broken surface diagrams). It thus follows that the odd cobordism
maps are not invariant under ribbon moves.

\printbibliography

\end{document}